\documentclass[psamsfonts,12pt]{amsart}
\makeatletter
\def\@tocline#1#2#3#4#5#6#7{\relax
  \ifnum #1>\c@tocdepth % then omit
  \else
    \par \addpenalty\@secpenalty\addvspace{#2}%
    \begingroup \hyphenpenalty\@M
    \@ifempty{#4}{%
      \@tempdima\csname r@tocindent\number#1\endcsname\relax
    }{%
      \@tempdima#4\relax
    }%
    \parindent\z@ \leftskip#3\relax \advance\leftskip\@tempdima\relax
    \rightskip\@pnumwidth plus4em \parfillskip-\@pnumwidth
    #5\leavevmode\hskip-\@tempdima
      \ifcase #1
       \or\or \hskip 1em \or \hskip 2em \else \hskip 3em \fi%
      #6\nobreak\relax
    \hfill\hbox to\@pnumwidth{\@tocpagenum{#7}}\par% <---- \dotfill -> \hfill
    \nobreak
    \endgroup
  \fi}
\makeatother

\usepackage[margin=1in]{geometry}
\usepackage[english]{babel}
\usepackage{amssymb,amsfonts,amsthm,graphicx,tikz-cd,color,enumitem,mathrsfs,amsmath,mathtools}
\usepackage{caption}
\usepackage{subcaption}
\usepackage{enumitem}
\usepackage{mathbbol}
\usepackage{mathrsfs}
\usepackage{float}
\usepackage{comment}
\usepackage{graphicx}
\usepackage{hyperref}
\usepackage{quiver}

\newcommand\e{\varepsilon}

\newcommand\bR{\mathbb{R}}
\newcommand\R{\mathbb{R}}
\newcommand\bS{\mathbb{S}}

\newcommand\Z{\mathbb{Z}}

\newcommand{\CS}{\mathcal{CS}}
\newcommand{\bk}{\mathbb{k}}
\newcommand{\dotr}{\mbox{$\boldsymbol{\cdot}$}}

\DeclareMathOperator{\holim}{holim}
\DeclareMathOperator{\PH}{PH}
\DeclareMathOperator{\PHT}{PHT}
\DeclareMathOperator{\pht}{PHT}
\DeclareMathOperator{\ECT}{ECT}
\DeclareMathOperator{\BCT}{BCT}
\DeclareMathOperator{\CF}{CF}
\DeclareMathOperator{\im}{im}
\DeclareMathOperator{\dom}{dom}

\DeclareMathOperator{\Rad}{\mathcal{R}}
\DeclareMathOperator{\Dgm}{\mathbb{Dgm}}
\DeclareMathOperator{\Fun}{\mathbb{Fun}}
\DeclareMathOperator{\Open}{Open}
\newcommand{\F}{\mathcal{F}}
\newcommand{\calG}{\mathcal{G}}
\newcommand{\D}{\mathcal{D}}
\newcommand{\bX}{\mathbb{X}}
\newcommand{\bY}{\mathbb{Y}}
\newcommand{\cU}{\mathcal{U}}

\newcommand{\cD}{\mathcal{D}}
\newcommand{\cC}{\mathcal{C}}
\newcommand{\cF}{\mathcal{F}}
\newcommand{\cQ}{\mathcal{Q}}
\newcommand{\cK}{\mathcal{K}}
\newcommand{\calB}{\mathcal{B}}

\DeclareMathOperator{\Acat}{\mathbf{A}}
\DeclareMathOperator{\Bcat}{\mathbf{B}}
\DeclareMathOperator{\Set}{\mathbf{Set}}
\DeclareMathOperator{\Top}{\mathbf{Top}}
\DeclareMathOperator{\Vect}{\mathbf{Vect}}
\DeclareMathOperator{\vect}{\mathbf{vect}}
\DeclareMathOperator{\Dat}{\mathbf{Dat}}
\DeclareMathOperator{\pshf}{ps}
\DeclareMathOperator{\cost}{cost}
\DeclareMathOperator{\shff}{sh}
\DeclareMathOperator{\PShv}{\mathbf{PShv}}
\DeclareMathOperator{\Shv}{\mathbf{Shv}}
\DeclareMathOperator{\scrS}{\mathscr{S}}
\DeclareMathOperator{\cha}{\mathrm{ch}(\mathcal{A})}

\newcommand{\define}[1]{\textbf{#1}}

\newtheorem{thm}{Theorem}[section]
\newtheorem{cor}[thm]{Corollary}
\newtheorem{prop}[thm]{Proposition}
\newtheorem{lem}[thm]{Lemma}

\theoremstyle{definition}
\newtheorem{defn}[thm]{Definition}

\newtheorem{exmp}[thm]{Example}

\newtheorem{rmk}[thm]{Remark}

\newcommand\rightthreearrow{%
        \mathrel{\vcenter{\mathsurround0pt
                \ialign{##\crcr
                        \noalign{\nointerlineskip}$\rightarrow$\crcr
                        \noalign{\nointerlineskip}$\rightarrow$\crcr
                        \noalign{\nointerlineskip}$\rightarrow$\crcr
                }%
        }}%
}

\makeatletter
\let\c@equation\c@thm
\makeatother
\numberwithin{equation}{section}

\title{A Sheaf-Theoretic Construction of Shape Space}

\author{Shreya Arya}
\address{Department of Mathematics, Duke University; Durham, NC USA}
\email{shreya.arya@duke.edu}

\author{Justin Curry}
\address{Department of Mathematics and Statistics, University at Albany SUNY, Albany, NY USA}
\email{jmcurry@albany.edu}

\author{Sayan Mukherjee}
\address{Departments of Statistical Science, Mathematics, Computer Science, Biostatistics \& Bioinformatics, Duke University; Durham, NC USA; Center for Scalable Data Analytics and Artificial Intelligence Universit\"at Leipzig and the Max Planck Institute for Mathematics in the Natural Sciences; Leipzig Germany}
\email{sayan@stat.duke.edu}

\begin{document}

\begin{abstract}
 We present a sheaf-theoretic construction of shape space---the space of all shapes. We do this by describing a homotopy sheaf on the poset category of constructible sets, where each set is mapped to its Persistent Homology Transforms (PHT). Recent results that build on fundamental work of Schapira have shown that this transform is injective, thus making the PHT a good summary object for each shape. Our homotopy sheaf result allows us to “glue” PHTs of different shapes together to build up the PHT of a larger shape. In the case where our shape is a polyhedron we prove a generalized nerve lemma for the PHT. Finally, by re-examining the sampling result of Smale-Niyogi-Weinberger, we show that we can reliably approximate the PHT of a manifold by a polyhedron up to arbitrary precision.

%\keywords{First keyword \and Second keyword \and More}
\end{abstract}
\maketitle

\tableofcontents

% \shreya{I will create a bib tex. Feel free to put references as bibitems in any order. I will take care of the ordering etc. }

\section{Introduction and Main Results}\label{sec:intro}

Shape spaces are intended to provide a single framework for comparing shapes.
Different shapes are rendered as different points in shape space and comparisons of shapes can be formalized in terms of distances between these different points.
One of the first examples of a shape space was pioneered in \cite{Kendall77,Kendall84} and takes the perspective that two shapes can be compared by first placing landmarks down on each shape.
If these landmarks are related by a similarity transformation---a rotation or dilation---then they are regarded as the same shape.
Non-identical shapes are then compared in an appropriately defined quotient space, assuming the number of landmarks are the same.
Another popular model of shape space dispenses with the landmark selection process and considers shapes as immersed submanifolds \cite{cervera} and then tries to compare them via diffeomorphisms of the ambient space, possibly generated by some optimal transport or control problem \cite{DepuisGrenander98}.

However, common to both the landmark and diffeomorphism-based approaches to shape space is a quotient operation that naturally lends a fiber bundle structure to these data representations and comparisons.
Fiber bundles provide a convenient mathematical language for shape comparison, but previous work \cite{Jost-fiber} also illustrates their insufficiency for general morphological comparison.
Indeed, the implicit assumption for both of these models---and many more models for shape space discussed below in Section \ref{sec:prior-shape-space}---is that the shapes under consideration have the same topology.
% Often one wants to decorate shape space with extra structure, such as a set of landmarks (as in the Kendall approach \cite{Kendall77,Kendall84}) or with a choice of parameterization (as in the Grenander approach \cite{DepuisGrenander98}), but this extra structure is regarded as lying orthogonal to the base manifold of shapes; see Figure \ref{fig:Shape Spaces}.
% Fiber bundles provide a language for formalizing this orthogonality~\cite{Jost-fiber} and can be used to unify Kendall's and Grenander's constructions, as reviewed below, but they are limiting as well.
% Implicit in both the Kendall and Grenander approach to shape space is the assumption that each pair of shapes can be related to one another via one-to-one correspondences; 
% for Kendall these are correspondences of landmarks, and for Grenander these are smooth diffeomorphisms. 
% \shreya{I think this is correct, even though it does not require shapes to be diffeo it requires there is a diffeo in ambient space that connects them. }

This assumption severely limits the applicability of these approaches to many datasets of interest.
For example, in a dataset of fruit fly wings, some mutant flies have extra lobes of veins \cite{Miller}; or, in a dataset of brain arteries, many of the arteries cannot be continuously mapped to each other \cite{bendich2016}. 
Indeed, in large databases such as MorphoSource and Phenome10K \cite{Boyer:2016aa,Goswami:2015aa}, the CT scans of skulls across many clades are not diffeomorphic or even homeomorphic.
Although some methods have been developed to address these issues, a general theory of shape comparison is still lacking.

In this paper we introduce an algebraic construction of shape space that circumvents some of the problems discussed above.
Topologically different subsets of $\bR^d$ can be viewed simultaneously and compared in our framework. 
We accomplish this by passing from the land of fiber bundles to the world of \emph{sheaves}, which replaces the local triviality condition of fiber bundles with the local continuity condition of sheaves.
This passage requires two preparatory steps of categorical generalization:
\begin{enumerate}
    \item Instead of a ``base manifold'' of shapes we work with a ``base poset'' of constructible sets $\CS(\bR^d)$ ordered by inclusion. This poset is equipped with a notion of continuity via a Grothendieck topology.
    \item Each shape---that is, each point $M\in \CS(\bR^d)$---is equivalently regarded via its persistent homology transform $\PHT(M)$, which is an object in the derived category of sheaves $\D^b(\Shv(\bS^{d-1}\times \R))$.
\end{enumerate}

With these observations in place, our main result can be summarized as follows.
\begin{thm}\label{thm:htpy-sheaf}
    The following assignment is a homotopy sheaf:
    \[
	\F:\CS(\bR^d)^{op} \to \D^b(\Shv(\bS^{d-1}\times \bR)) \qquad
	M \mapsto  \PHT(M).
	\]
\end{thm}
This is precisely stated and proved as Theorem~\ref{thm:main-hpty-sheaf} below. 
Intuitively, this result allows us to interpolate between shapes in a continuous way via their persistent homology transforms; continuity is mediated via a Grothendieck topology on $\CS(\bR^d)$.
More precisely, our main result establishes \emph{\v{C}ech descent} for the persistent homology transform, which is a generalization of the sheaf axiom that holds for higher degrees of homology.
In one concrete form, our main result implies the following theorem, which appears as Corollary~\ref{cor:nerve-lemma} below.
\begin{thm}[Nerve Lemma for the PHT]\label{thm:nerve-lemma}
    If $ M\in \CS(\bR^d)$ is a polyhedron, i.e.~it can be written as a finite union of closed linear simplices $ \mathcal{M} = \{\sigma_i\}_{i\in \Lambda}$, then the $n$-th cohomology sheaf of the $PHT$, i.e., the $n^{th}$ PHT sheaf of Definition \ref{defn:pht-sheaf}, written $\PHT^n(M)$, is isomorphic to the $n$-th cohomology of the following complex of sheaves:
	\begin{equation*}
	0\to \bigoplus_{I\subset \Lambda \,s.t.\, |I|=1}\pht^0(\mathcal{M}_I) \to \bigoplus_{I\subset \Lambda \,s.t.\, |I|=2}\pht^0(\mathcal{M}_I) \to \cdots
	\end{equation*}
	Here $\mathcal{M}_I$ with $|I|=k$ denotes the disjoint union of depth $k$ intersections of closed simplices appearing in the cover $\mathcal{M}$.
\end{thm}

% We motivate both of these results with a simpler ``local-to-global principle'' for the Euler characteristic transform (ECT, cf.~Definition \ref{defn:ECT}) that takes the form of a generalized inclusion-exclusion principle (Theorem \ref{thm:decat}).
% In Remark \ref{rmk:decat} we detail exactly how the above two results decategorify to inclusion-exclusion for the ECT.

\begin{figure}
     \centering
     \begin{subfigure}[b]{0.3\textwidth}
         \centering
         \includegraphics[width=\textwidth]{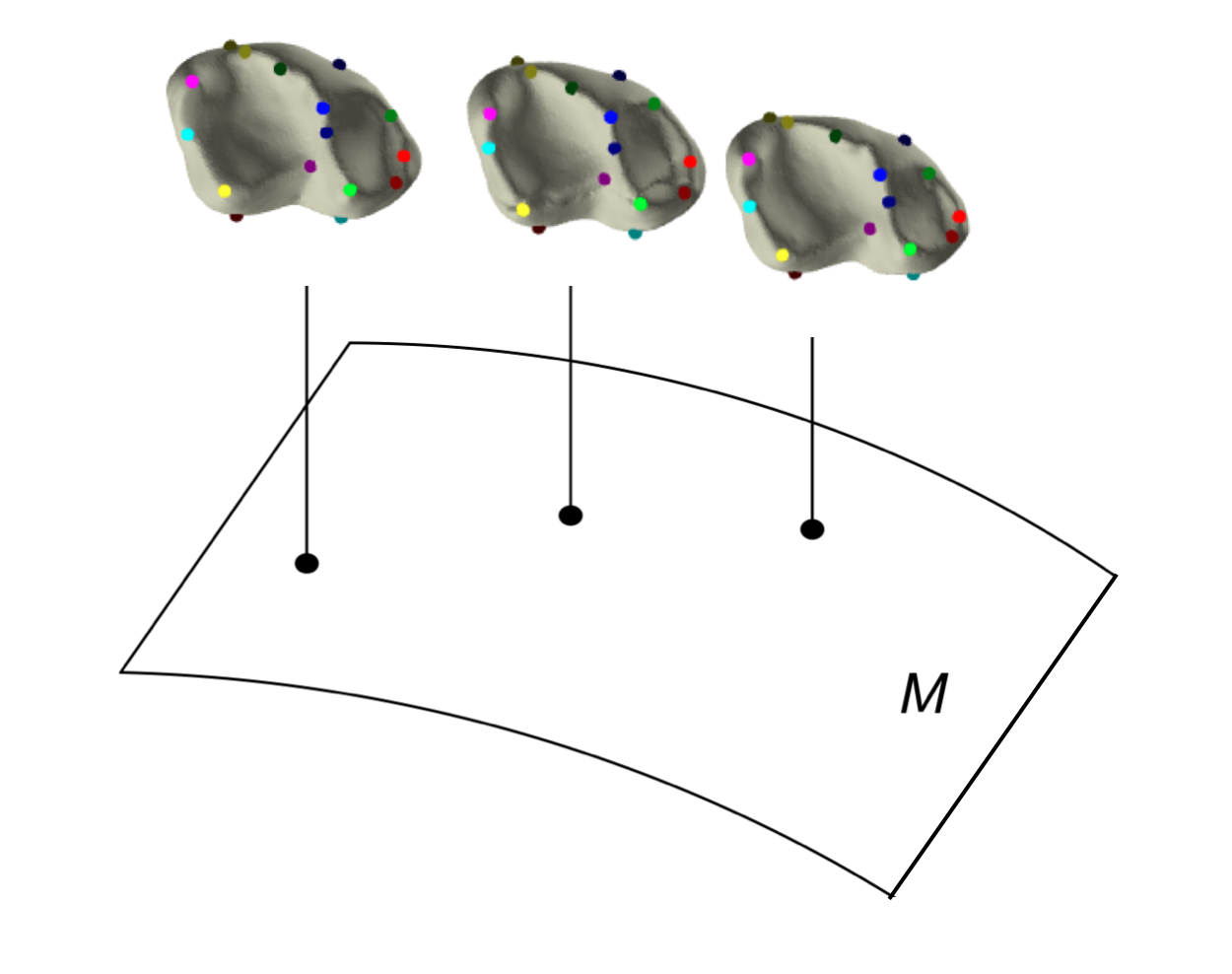}
         \caption{Landmark-based}
         \label{fig:kendall}
     \end{subfigure}
     \hfill
     \begin{subfigure}[b]{0.3\textwidth}
         \centering
         \includegraphics[width=\textwidth]{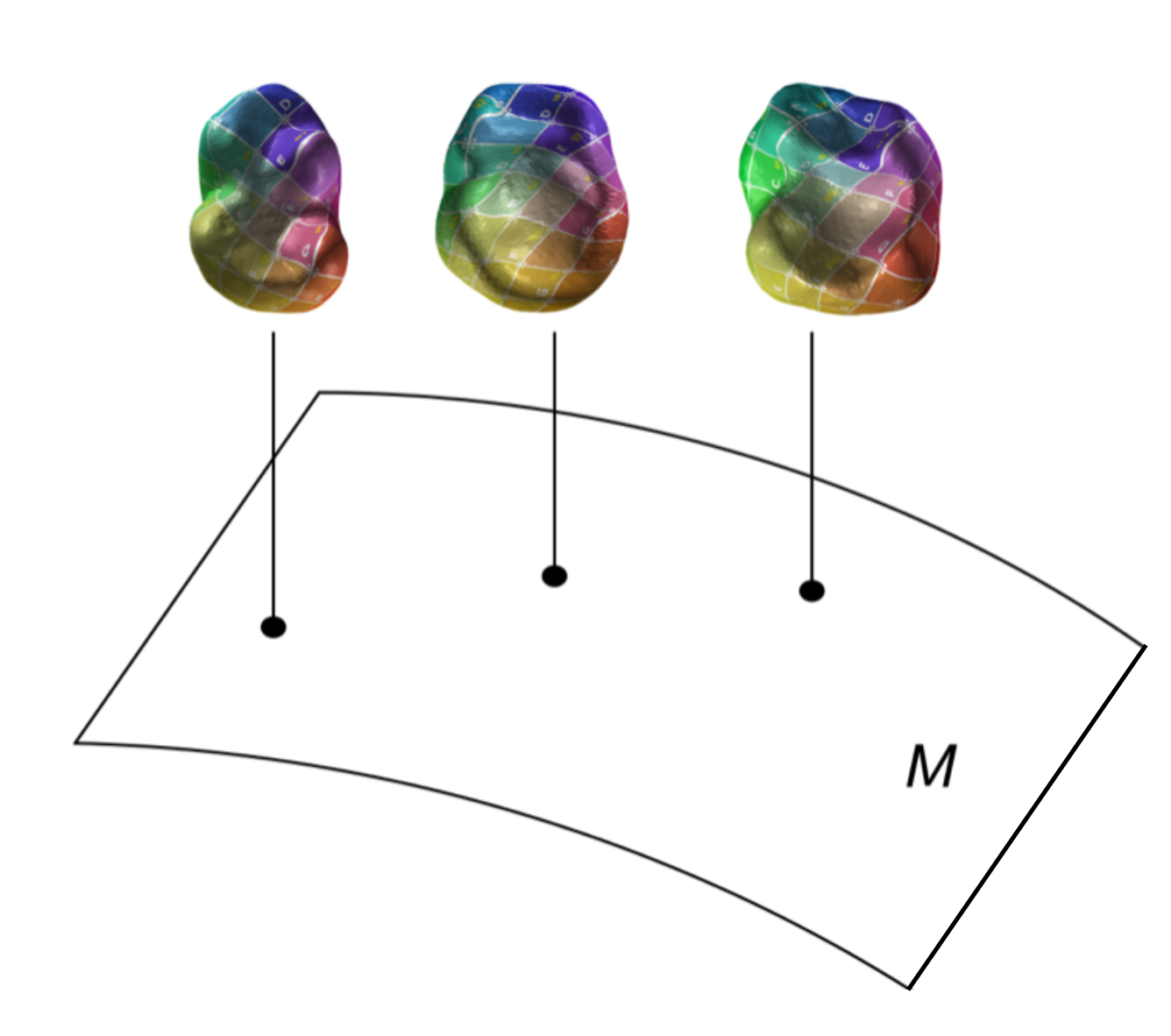}
         \caption{Diffeomorphism-based}
         \label{fig:fiberbundle}
     \end{subfigure}
     \hfill
     \begin{subfigure}[b]{0.3\textwidth}
         \centering
         \includegraphics[width=\textwidth]{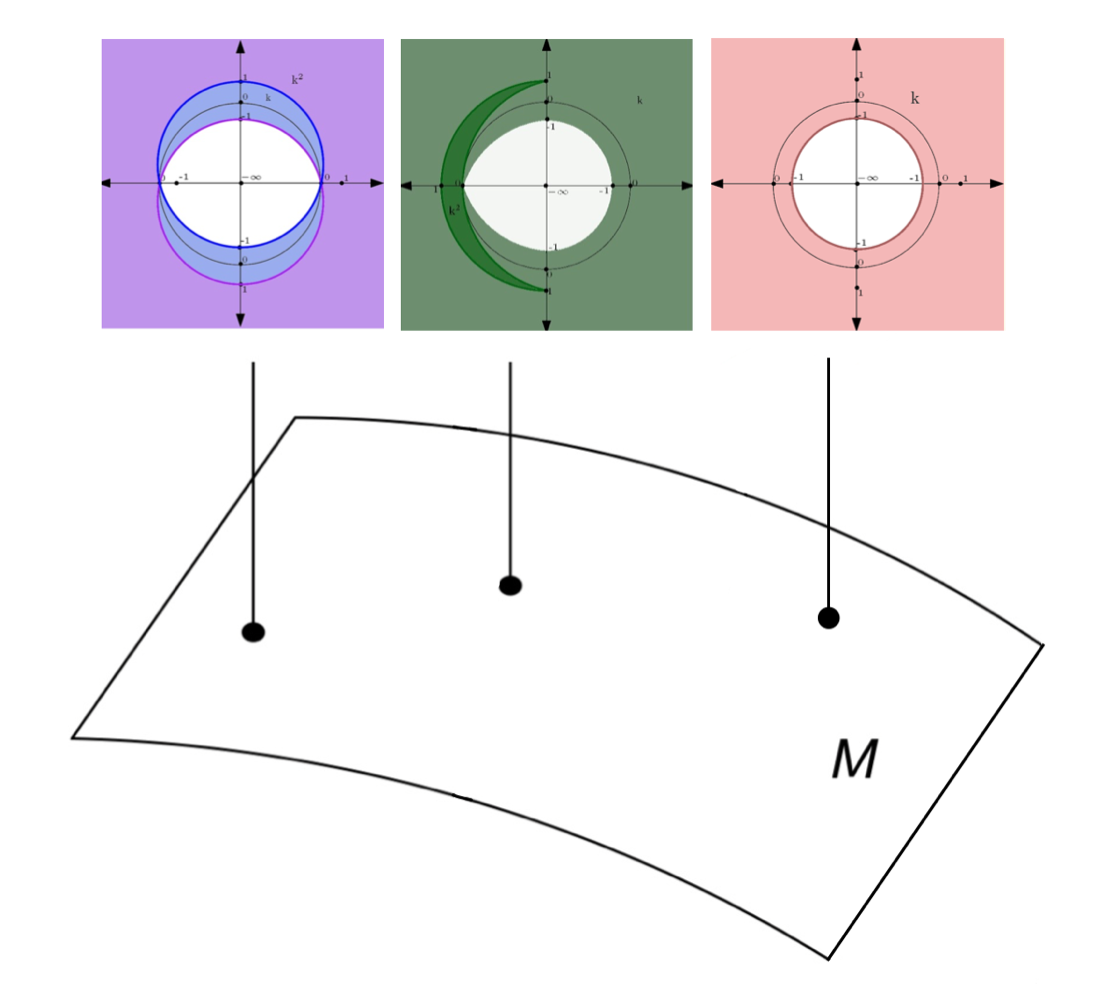}
         \caption{PHT-based}
         \label{fig:stalk}
     \end{subfigure}
     \caption{Previous constructions of shape space imply a fiber bundle perspective on shape space. The fibers encode similarity transformations or reparameterizations of a shape and the base space records the shape as an equivalence class. The PHT-based shape space introduced here uses a more algebraic construction where the base space is replaced by a base poset and the ``fibers'' are unique sheaf-theoretic representations of the shape.}\label{fig:Shape Spaces} 
\end{figure}

% \remove{It should be noted that positive scalar curvature of a constructible set $M$ (when defined) obstructs Theorem \ref{thm:nerve-lemma} from being directly applied, as cover elements may necessarily have higher homology when viewed in a direction normal to that point. See Figure~\ref{fig:PHTGOODCOVER} for an example.}

The importance of Theorem \ref{thm:nerve-lemma} (Corollary \ref{cor:nerve-lemma}) is that it allows us to calculate the higher homology PHTs using only $\PHT^0$ of elements in a PL covering.
Since finding connected components in a filtration is computationally simpler---$\mathcal{O}(n\log n)$---than calculating higher homology, this result illustrates how the theory of spectral sequences provides potential guarantees on computing the full PHT of a shape in all degrees.
However, the ability to cover a shape by PL shapes, and thus approximate a shape using PL ones, is a necessary first step.
% This in turn motivates the question of how much an approximation distorts the resulting PHTs.
% It is well known that the PHT is not stable under Hausdorff perturbations of the shape.
% This motivates us to define several novel metrics on PHTs and prove stability theorems for them.

As such it is desirable to have an approximation result that is provably stable under the persistent homology transform. 
We do this by first proving a general stability theorem for the PHT---after introducing several new metrics in Section \ref{sec: distance-on-pht}---under controlled homotopies.
The following theorem appears as Theorem \ref{thm::stability} below.

\begin{thm}[Stability of the PHT]
    Any two constructible sets $M,N \in \CS(\bR^{d})$ that are homotopy equivalent via homotopies that move no point more than $\e$ have $\e$-close PHTs of $M$ and $N$.
\end{thm}

Finally, using the sampling methods of \cite{niyogi}, we show how to construct with high confidence a PL approximation of any nice enough submanifold of $\R^d$.

\begin{cor}[Approximation of the PHT]\label{thm:approx}
    For any compact submanifold $M$ and $\e>0$ sufficiently small we can construct a polyhedron $N$ so that with high probability $\PHT(M)$ and $\PHT(N)$ are $\e$-close. $\PHT(N)$ can then be computed using Theorem \ref{thm:nerve-lemma}.
\end{cor}

\subsection{Prior Work on Shape Space}\label{sec:prior-shape-space}

We now outline in more detail prior approaches to shape space, in order to better situate the contributions of this paper.
There is a rich history of using differential geometry for modeling shapes, dating back to Riemann \cite{Jost-Riemann}, and the reader is encouraged to consult the survey articles \cite{Bauer-Shape-Survey} and \cite{fowlkes} for more context there.
%\shreya{\cite{Bauer-Shape-Survey} papers survey only one method. Maybe can add comp. anatomy survey ref \cite{miller2015} }
%\remove{However, at a high-level, there are three main approaches to shape space that have been worked out in some detail: the landmark approach; the diffeomorphism and optimal control approach; and the persistent homology transform and Euler characteristic transform approach endorsed here.}
However, what follows is a woefully incomplete survey of the literature that tries to balance the theoretical study of shape space as well as the development of algorithms that are used in practice.

As already mentioned, the landmark-based approach to shape space was pioneered in the works of Kendall \cite{Kendall77,Kendall84} and was given a textbook-length treatment in Bookstein \cite{Bookstein97}. 
In this approach, a pre-shape is defined by a set of $k$ landmark points in $\bR^d$, where typically $d = 2,3$. 
Comparing pre-shapes operates under the assumption that the $i$-th point in one shape corresponds to the $i$-th point in every other pre-shape; this introduces the central notion of correspondences in shape space and leads naturally to Procrustean distances, see~\cite{Kendall84} and more generally~\cite{procrustes}. 
Shapes are then regarded as elements of the quotient space 
\[
\Sigma_d^k := \{(\bR^d)^{k-1}\setminus 0\}/\mbox{Sim},
\]
where Sim is the group of rotations and dilations. 
Note that in $\Sigma_d^k$ a shape is deduced from a $d \times k$ matrix, which is a very convenient representation. 
However, the downside of this approach is that a user will need to decide on landmarks before analysis can be carried out, and reducing modern databases of 3-dimensional micro-computed tomography (CT) scans \cite{Goswami:2015aa,Boyer:2016aa} to landmarks can result in a great deal of information loss. 

% The second commonly accepted shape space was pioneered by Grenander \cite{DepuisGrenander98}, although some aspects were anticipated by \cite{cervera}. In these works, a shape space is specified for each manifold $M$ and dimension $d$. One then considers all possible immersions modulo the group of reparameterizations of $M$, i.e.
% \[
%     \text{Shape}(M) := \text{Imm}(M,\R^d)/\text{Diff}(M).
% \]
% Variation in shape is then modeled by the action of the Lie group of diffeomorphisms on $\R^d$. 
% The advantage of Grenander's approach is that it bypasses the need for landmarks, but the resulting spaces of interest are infinite-dimensional and shapes with different topology cannot be compared. 
% \remove{However, many tools have been developed that efficiently compare the similarity between shapes in large databases via algorithms that continuously deform one shape into another \cite{Boyer:2011aa,Ovsjanikov:2012aa,Boyer:2015aa,Gao:2018ab}.}

The second approach to shape space mentioned above really is a broader umbrella of techniques, but they are all connected by the use of larger, infinite-dimensional groups of deformations to compare and match shapes using an underlying template.
Pulling on ideas from analysis, differential geometry, and probability theory, these approaches are sometimes gathered under the heading of ``Pattern Theory'', see \cite{Grenander:1996,Grenander:1998} and \cite{mumford-pattern-theory}.
In many of these approaches, the collection of shapes is regarded as a homogeneous space under the group of diffeomorphisms of $\mathbb{R}^d$.
The difference between two shapes under this paradigm is encoded by the deformation that matches one shape to another. 
Theses differences can be measured algorithmically by solving a registration problem, with the large deformation diffeomorphic metric mapping (LDDMM) \cite{beg:2005,joshi,argui2015} being one such class of methods. 
%Roughly speaking these algorithms measure the amount the ambient space has to be deformed so that the desired shape is formed.
These algorithms typically find the optimal one-parameter family of diffeomorphisms, parametrized by time, that smoothly transforms the initial shape to the target; the optimization problem considered is over paths with the least time integrated kinetic energy. 
The theory of these shape spaces typically involves the geometry (and Riemannian structure) on the infinite-dimensional space of immersions of a fixed template into $\R^d$, written $\mathrm{Imm}(M,\bR^d)$, and the resulting, induced $\mathrm{Diff}(M)$-invariant metrics \cite{mumford,michor-mum}.  
%which led to the study of the geometry of shape spaces and diffeomorphims \cite{Bauer-Shape-Survey}).
Although this viewpoint bypasses the need for landmarks, the resulting spaces of interest are infinite-dimensional. 

There are many more approaches to shape comparison worth mentioning. In similar spirit to the above, there is elastic shape analysis \cite{younes:1998, mio:2007,Srivastava2010} as well as approaches based on conformal geometry \cite{wang, Sharon}, currents \cite{valliant}, and varifolds \cite{charon}. 
%Not all of these approaches require the topology of the shapes to be the same, for example, the metric geometry   shaped based signatures perspective and the currents perspective works for any subsets of Euclidean space.}}
%\question{A lot of the methods do not require that the shapes have the same topology, but rather require that there exists a diffeomorphism of the ambient space that takes one shape to another.  When the shapes are embeddings, then this implies that the shapes are diffeomorphic but this is not true when shapes are immersions. }
Although each of these make varying assumptions on the geometric rigidity of the underlying shapes, the approaches of \cite{memoli} and \cite{bronstein:2010} are notable for working with probabilistic variations on the Gromov-Hausdorff model for shapes, which is the most general and works for arbitrary metric spaces.

Finally, the shape space model that we propose builds on fundamental work of Schapira \cite{schapira-op,schapira-tom}, who implicitly developed the Euler characteristic transform (ECT), which was later developed alongside the
persistent homology transform (PHT) in \cite{tmb} to compare triangulable shapes.
The ECT and PHT have two useful properties: standard statistical methods can be applied to the transformed shape and the transforms are injective \cite{Ghrist:2018aa,cmt}, so no information about the shape is lost via the transform. 
The utility of the transforms for applied problems in evolutionary anthropology, biomedical
applications and plant biology were demonstrated in \cite{Crawford:2017aa,10.1214/20-AOAS1430,ProSinatra,Amezquita-2022}.
%The shape space we construct in this paper is a dramatic generalization of the sheaf-theoretic formulation of the PHT found in \cite{cmt}. 

\subsection{Our Contribution to Shape Space Theory}
% \color{blue}Respond to referee about how the PHT performs better than other shape space analysis techniques. Focus on distances between PHTs.
% \color{black}
Although \cite{tmb} pioneered the use of the ECT and PHT for shape classification and machine learning, the algebraic relation between the transforms of varying shapes was not realized.
In this paper we leverage the functoriality properties of homology to show how the PHT admits a finer organizing structure for the space of all triangulable shapes, ordered by inclusion.
In this sense, we move beyond the particular mathematical properties proved in \cite{cmt}, to a broader outlook of how to move between shapes using a sheaf structure.
Pinning down this structure occupies the bulk of the paper as it requires an unprecedented amount of algebra for shape space theory, including the first known contact with $\infty$-categories, currently isolated to Appendix \ref{app:infinity}.
Although the current paper is primarily directed towards developing the full algebraic picture of a PHT-based shape space, we have also made several innovations on the metrics and analysis side, with some preliminary comparison with Procrustes-type distances detailed in Section \ref{sec:compare-distances}.
A future research program for more systematically comparing our approach with the approaches in Section \ref{sec:prior-shape-space} is outlined in Section \ref{sec:future-work}.

\section{Background on Constructibility, Persistent Homology and Sheaves}\label{sec:background}

In this section we recall background material from \cite{cmt}.
We proceed by defining the class of shapes that we want to work with---constructible sets (Definition \ref{defn:o-minimal})---then identify an increasingly refined set of topological transforms for studying these.
The Euler characteristic transform (ECT), the simplest of them all, paves the way for the Betti curve transform (BCT) and Persistent Homology Transform (PHT). 

%\justin{Will rewrite to include more material on topological transforms, including ECT/BCT/PHT.}

\subsection{O-Minimality}
Although shapes in the real world can exhibit wonderful complexity, we impose a fairly weak tameness hypothesis that prohibits us from considering infinitely constructed shapes such as fractals and Cantor sets.
This tameness hypothesis is best expressed using the language of o-minimal structures \cite{tametopology}. 

\begin{defn}\label{defn:o-minimal}
An \define{o-minimal structure} $\mathcal{O}= \{\mathcal{O}_d\}$, is a specification of a boolean algebra of subsets $\mathcal{O}_d$ of $\bR^d$ for each natural number $d\geq 0$.
In particular, we assume that $\mathcal{O}_1$ contains only finite unions of points and intervals.
We further require that $\mathcal{O}$ be closed under certain product and projection operators, i.e.~if $A \in \mathcal{O}_d$, then $A\times \bR$ and $\bR\times A$ are in $\mathcal{O}_{d+1}$; 
and if $A \in \mathcal{O}_{d+1}$, then $\pi(A)\in \mathcal{O}_d$ where $\pi: \bR^{d+1}\to \bR^d$ is axis aligned projection. 
It is a fact that $\mathcal{O}$ contains all semi-algebraic sets, but may contain certain regular expansions of these sets.
Elements of $\mathcal{O}$ are called \define{definable sets} and \emph{compact} definable sets are called \define{constructible sets}.
The subcollection of constructible subsets in $\mathcal{O}_d$ is denoted $\CS(\R^d)$.
\end{defn}

Intuitively, any shape that can be faithfully represented via a mesh on a computer is a constructible set.
This is because every constructible set is triangulable \cite{tametopology}.
This property also implies that certain algebraic topological signatures, such as Euler characteristic and homology, are well-defined for any constructible set.

\begin{figure}
\centering
\includegraphics[width=\textwidth]{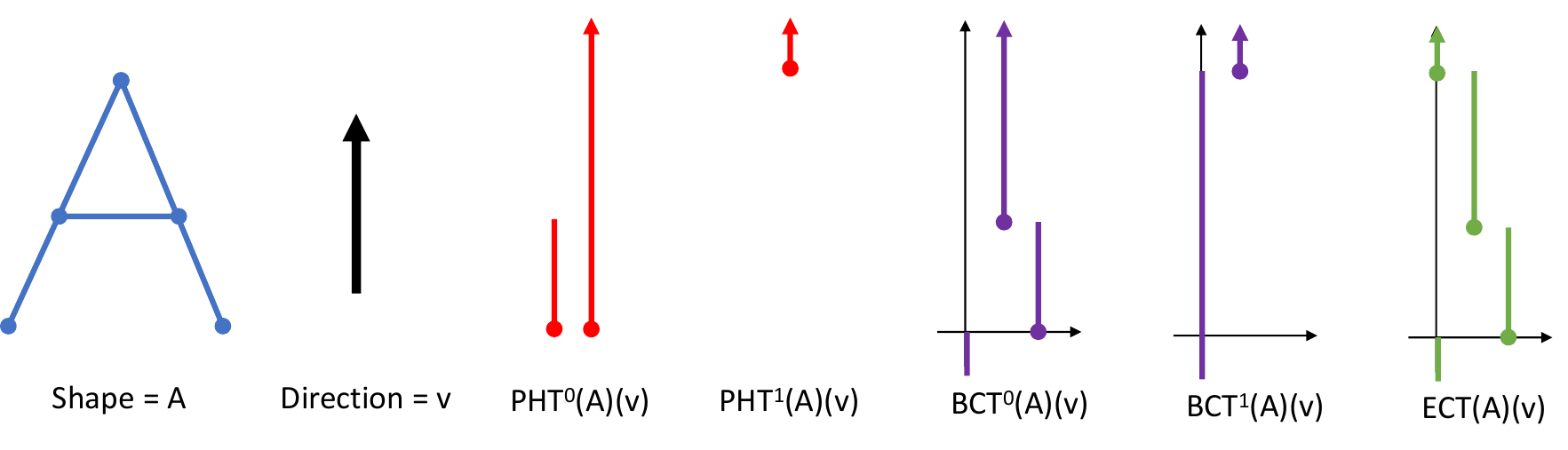}
\caption{The letter `A' summarized using three, increasingly coarse, topological transforms. The Persistent Homology Transform (PHT) and the Betti Curve Transform (BCT) are both graded by homological degree, but the Euler Characteristic Transform (ECT) produces a single integer-valued function for each direction $v$.}\label{fig:top-transforms-A}
\end{figure}

\subsection{Euler Characteristic and the Radon Transform}\label{sec:euler-char-intro}

Any triangulable space $M$ has a well-defined Euler characteristic, which is given by the formula
\[
\chi(M)=\sum_{i=0}^{\infty} (-1)^i \# i-\text{dimensional simplices}.
\]
The fact that this is a topological invariant---meaning it is independent of the choice of triangulation---is a surprising, if not routine, exercise in homology theory, which stands as the deeper invariant that we will work towards.
However, even this simple alternating count of cells can prove to be a powerful invariant for an embedded subset $M\in \CS(\R^d)$ when filtered in a given direction $v$; see the rightmost part of Figure \ref{fig:top-transforms-A} for an example.
This provides our first definition of a topological transform, which we will refine with other topological invariants in due course.

  \begin{defn}[ECT: Map Version]\label{defn:ECT}
    The \define{Euler Characteristic Transform (ECT)} of a constructible set $M\in \CS(\R^d)$ is the map that assigns to each direction $v\in \bS^{d-1}$ the piece-wise constant integer-valued function on $\R$ that records the Euler characteristic of the sublevel set of $M$ in direction $v$, i.e.,
    \[ \ECT(M):\bS^{d-1} \to \Fun(\R,\Z) \qquad \ECT(M)(v,t) = \chi\big(M_{v,t}\big), \]
    where $M_{v,t}:=\{x\in M \mid x\cdot v \leq t\}$ is the intersection of $M$ with the half-space $x\cdot v\leq t$.

    We note that since $M$ is constructible and the equation defining a sub-level set is semi-algebraic---and hence definable---the intersection $M_{v,t}$ is constructible as well, and thus has a well-defined Euler characteristic.
  \end{defn}

%Alternatively, one can view the Euler Characteristic Transform of $M$ as the Radon transform of its indicator function on $\bR^d$ \cite{cmt,Ghrist:2018aa}.
%This perspective requires using the operations of Euler Calculus to their full effect \cite{schapira-op}.

As our notation suggests, one can also view $\ECT(M)$ as a function from $\bS^{d-1}\times \R$ to $\Z$ that assigns to each pair $(v,t)\in \bS^{d-1}\times \R$ the Euler characteristic $\chi(M_{v,t})$. 
This perspective is important because it allows us to view the ECT as a type of integral transform, which takes a shape (viewed as an indicator function on $\R^d$) and produces a function on $\bS^{d-1}\times \R$---a coordinate system built for tomographic comparison.
To make this precise, we briefly review the key ingredients of Euler calculus \cite{curry2012euler}, which treats Euler characteristic as a type of measure\footnote{More accurately, $\chi$ defines a \emph{valuation} \cite{alesker2018introduction} on constructible sets, as it can take negative values.}.

% %The most important aspect of Schapira's work is that it provided a dictionary between certain operations on constructible functions with operations on constructible sheaves, which collectively determine what is called the Euler calculus \cite{euler}
\begin{defn}[Constructible Functions, Integration, Pushforward and Pullback, cf.~\cite{schapira-op}]
If $X\in \CS(\R^d)$, then a \define{constructible function} is an integer valued function $f:X \to \Z$ with only finitely many non-empty level sets, each of which are definable, and hence triangulable.
We denote the subgroup of constructible functions on $X$ by $\text{CF}(X)\subset \Fun(X,\Z)$.

If $\varphi:X\to Y$ is a (tame) mapping between constructible sets, then we have \define{pushforward} $\varphi_*:\text{CF}(X) \to \text{CF}(Y)$ and \define{pullback} $\varphi^*:\text{CF}(Y) \to \text{CF}(X)$ operations via
\[
    \varphi_* f (y) :=\int_{\varphi^{-1}(y)} f d\chi \quad \text{and} \quad \varphi^*g(x):=g(\varphi(x)) \quad \text{where} \quad \int_X f d\chi := \sum_{n=-\infty}^{\infty} n \cdot \chi(f^{-1}(n))
\]
is integration with respect to compactly-supported Euler characteristic.
\end{defn}

% % A definable mapping $\varphi:X\to Y$ then specifies two group homomorphisms $\varphi_*:\CF(X) \to \CF(Y)$ and $\varphi^*:\CF(Y) \to \CF(X)$ via the formulas
% % \[
% % \varphi_* f (y) :=\int_{\varphi^{-1}(y)} f d\chi \qquad \text{and} \qquad \varphi^*g(x):=g(\varphi(x)).
% % \]
 
  \begin{defn}[Radon Transform and ECT: Version 2]\label{defn:radon-ECT}
    Let $S\subset X\times Y$ be a closed constructible subset of the product of two constructible sets. 
    Let $\pi_X$ and $\pi_Y$ be the projections onto the indicated factors. 
    The \define{Radon Transform} with respect to $S$ is a group homomorphism $\mathcal{R}_S:\text{CF}(X)\to\text{CF}(Y)$ defined by
    \[ 
    \mathcal{R}_S(\phi) := (\pi_Y)_* \big[ ((\pi_X)^*\phi) \mathbb{1}_S \big]
    \quad \text{where} \quad \mathcal{R}_S(\phi)(y)= \int_{\pi_Y^{-1}(y)} (\phi \circ \pi_X)\mathbb{1}_S d\chi. 
    \]
    Notice that by taking 
    $\phi= \mathbb{1}_M$ and $S= \{(x,v,t)\in \R^d\times \bS^{d-1}\times \bR \mid x\cdot v \leq t\}$ 
    the Euler characteristic transform of $M$ coincides with the Radon transform of its indicator function. 
  \end{defn} 

A celebrated theorem of Schapira~\cite{schapira-tom} gives a criterion for determining the invertibility of the Radon transform $\mathcal{R}_S$ in terms of the Euler characteristic of the fibers of $S\subset X \times Y$, when projected to each of these two factors:

\begin{thm}[\cite{schapira-tom} Theorem 3.1]\label{thm:inversion}
If $S\subset X\times Y$ and $S'\subset Y\times X$ have fibers $S_x$ and $S'_x$ in $Y$ satisfying
\begin{enumerate}
	\item $\chi(S_x\cap S'_x)=\chi_1$ for all $x\in X$, and
	\item $\chi(S_x\cap S'_{x'})=\chi_2$ for all $x'\neq x \in X$,
\end{enumerate}
then for all $\phi\in \CF(X)$,
\[
(\Rad_{S'} \circ \Rad_{S})\phi = (\chi_1-\chi_2)\phi +\chi_2 \left(\int_X \phi d\chi\right) 1_X .
\]
\end{thm}

In \cite{cmt} this was used to prove that the Euler Characteristic Transform is injective.

\begin{thm}[\cite{cmt} Theorem 3.5]\label{thm:ECT-injects}
The Euler Characteristic Transform $\ECT: \CS(\R^d) \to \CF(S^{d-1} \times \R)$ is injective, i.e.~if $\ECT(M)=\ECT(M')$, then $M=M'$.
% \[
% \text{If } \ECT(M)=\ECT(M') \text{ then } M=M'.
% \]
\end{thm}

\subsection{Homology and the Betti Curve Transform}\label{sec:homology-and-BCT}

A useful algebraic summary of a constructible set $M\in \CS(\R^d)$ is homology in degree $n$ with coefficients in a field $\Bbbk$, written $H_n(M;\Bbbk)$. 
Although there are many flavors of homology---simplicial, cellular, and singular, to name a few---for constructible sets they all agree.

As a reminder of the simplest version of homology---simplicial homology---one uses a triangulation of $M$ to define the group of $n$-chains, written $C_n(M;\Bbbk)$, which is the $\Bbbk$-vector space generated by all the $n$-simplices in the triangulation of $M$.
Since an $n$-simplex can be viewed as a formal combination of $n+1$ vertices $\sigma=\{v_{i_0},\ldots,v_{i_n}\}$, one has a natural map from $n$-chains to $n-1$-chains given by deleting one vertex at a time.
To make this precise, and to operate over fields $\Bbbk\neq \mathbb{F}_2$, we choose a local orientation of each simplex $\sigma$. 
This amounts to specifying an ordering of the vertices so that we can present $\sigma=[v_{i_0},\ldots,v_{i_n}]$ as an ordered list rather than a set of vertices.
The $n$-boundary operator $\partial_n$ takes a basis vector $\sigma$ to
\[
\partial_n \sigma = \sum_{j=0}^{n} (-1)^j [v_{i_0},\ldots,\hat{v}_{i_j},\ldots,v_{i_n}],
\]
where $\hat{v}_{i_j}$ indicates the vertex being deleted.
Extending this definition linearly, one obtains a linear map $\partial_n:C_n(M;\Bbbk) \to C_{n-1}(M;\Bbbk)$. 
It is a tedious, but routine, check to verify that $\partial_{n-1} \circ \partial_n=0$ for every $n\geq 0$.
This then guarantees that $\im \partial_{n+1} \subseteq \ker \partial_n$, which allows us to define the $n^{th}$ homology group\footnote{Actually a vector space, since we're using field coefficients.} of $M$ as
\[
H_n(M;\Bbbk) := \frac{\ker \partial_n}{\im \partial_{n+1}}.
\]
The rank/dimension of the $n^{th}$ homology group of $M$ is also called the $n^{th}$ Betti number $\beta_n(M)$, and it serves as another topological invariant of $M$, which can be used to distinguish shapes such as the torus $\mathbb{T}$ and the 2-sphere $\bS^2$, because $\beta_1(\mathbb{T})=2$ and $\beta_1(\bS^2)=0$.
Intuitively, the Betti numbers count the number of ``holes'' in each dimension. 
The circle $\bS^1$, or the letter `A' in Figure \ref{fig:top-transforms-A}, each have one hole in dimension 1 so $\beta_1(\bS^1)=\beta_1(\text{`A'})=1$.
Because the sphere $\bS^2$ and the torus $\mathbb{T}$ have a hollow interior which disconnects the complement when viewed as a subset of $\R^3$, they each have a single two-dimensional ``hole'', reflected in the calculation $\beta_2(\bS^2)=\beta_2(\mathbb{T})=1$.
In general, a compact subset of $\R^d$ cannot have any ``holes'' in dimension $d$ or higher, thus guaranteeing that $\beta_{i}(M)=0$ for $i\geq d$.

The Betti numbers also provide a ``lift'' of the Euler characteristic in the sense that it both refines the Euler characteristic via the formula
\[
\chi(M):=\sum_{i=0}^d (-1)^i\beta_i(M),
\]
but it also distinguishes spaces that the Euler characteristic cannot, e.g.~$\chi(\mathbb{T})=\chi(\bS^1)=0$, but $\beta_1(\mathbb{T})\neq \beta_1(\bS^1)$.
Filtering in each direction also lifts the Euler characteristic transform.

  \begin{defn}[BCT: Map Version]\label{defn:ECT}
    The \define{Betti Curve Transform (BCT)} of a constructible set $M\in \CS(\R^d)$ associates to each direction $v\in \bS^{d-1}$ the vector of piece-wise constant integer-valued functions on $\R$ that record the Betti numbers of the sublevel set of $M$ in direction $v$, i.e.,
    \[ \text{BCT}(M):\bS^{d-1} \to \CF(\R)^d \qquad \BCT(M)(v,t) = \big(\beta_0\big(M_{v,t}\big),\ldots,\beta_{d-1}\big(M_{v,t}\big)\big). \]
    %where $M_{v,t}:=\{x\in M \mid x\cdot v \leq t\}$ is the intersection of $M$ with the half-space $x\cdot v\leq t$.
  \end{defn}

\subsection{Functoriality and the Persistent Homology Transform}

Implicit in the definition of both the ECT and the BCT is the study of topological invariants of a 1-parameter family of spaces.
Studying how these invariants vary as a function of the parameter $t\in \R$ is the central idea of persistent topology.
Of all the topological invariants discussed---Euler characteristic, homology, and Betti numbers---homology is far more structured than the other two.
This is due to the \emph{functoriality} of homology, which states that maps on spaces $f:X\to Y$ induce maps on homology groups $f_*:H_*(X)\to H_*(Y)$ and these maps compose correctly, i.e.,~if $g:Y\to Z$ is another map, then the induced map on homology of the composition $(g\circ f)_*$ is the same as the composition of the induced maps $g_*\circ f_*$.
This is summarized by saying that homology in each degree $i$ defines a functor $H_i:\Top \to \Vect$ from the category\footnote{See \cite{riehl2017category} for a good review of category theory.} of topological spaces and continuous maps to the category of vector spaces and linear maps.

\begin{figure}
\centering
\includegraphics[width=.9\textwidth]{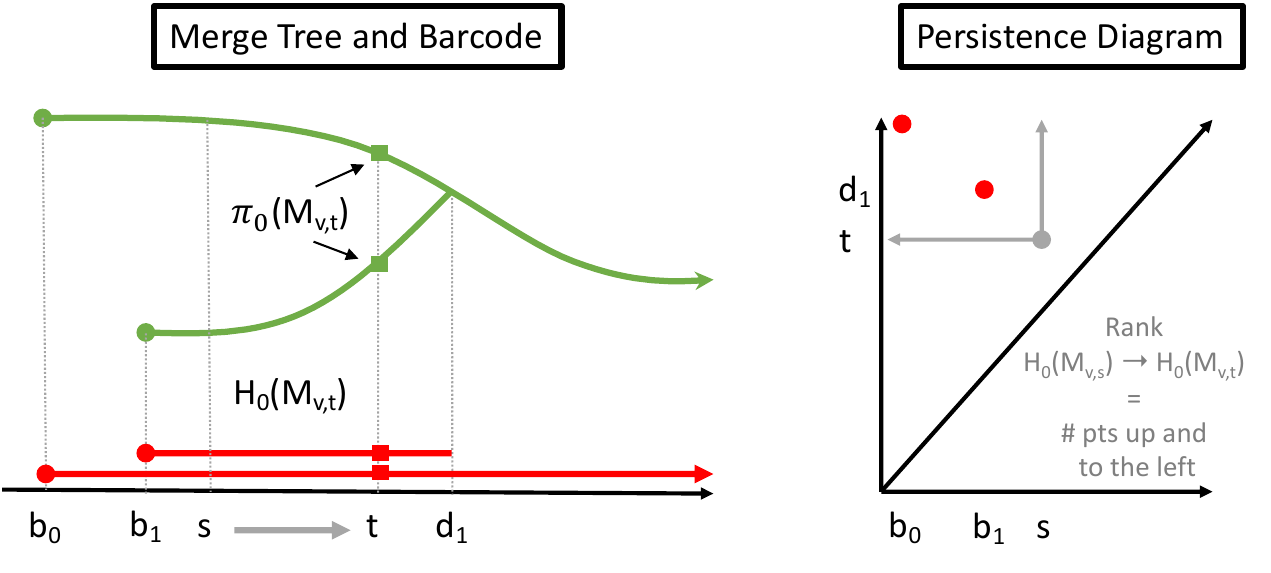}
\caption{When filtering a shape $M$ in the direction $v$, one tracks connected components using the functor $\pi_0$. The merge tree (shown above left) summarizes the 1-parameter family of sets $\pi_0(M_{v,t})$ and one can read off the maps by tracing components along the merge tree. Zeroth homology $H_0$ linearizes this family of sets and maps to produce vector spaces and linear maps. The barcode decomposition of $H_0(M_{v,t})$ is a canonical representation of this family of vector spaces and maps using a multi-set of intervals. The number of bars overlapping a point $t$ indicates the dimension of the vector space $H_0(M_{v,t})$ and the number of intervals spanning an interval encodes the rank of the map $H_0(M_{v,s})\to H_0(M_{v,t})$. The persistence diagram---our preferred representation---encodes these intervals as points above the diagonal. The rank of each map is encoded in this second representation by counting the number of points up and to the left.}\label{fig:mt-bc-pd-explanation}
\end{figure}

In traditional (sublevel set) persistent homology, one organizes the filtration of a shape $M$ in direction $v$ as a functor $F_v:(\R,\leq)\to \Top$, where $F(t):=M_{v,t}$ is the subset of $M$ cut out by $x\cdot v\leq t$ and $F(s\leq t):M_s \subseteq M_t$ is the induced inclusion of these subsets.
Post-composing the filtration functor with the homology functor defines the \define{persistent homology} of the filtration, written $\PH_i(M_v):=H_i\circ F_v : (\R,\leq) \to \Vect$.
It is a remarkable fact of representation theory \cite{crawley-boevey} that whenever this functor lands in the sub-category of \emph{finite-dimensional} vector spaces and maps, written $\PH_i(M_v): (\R,\leq) \to \vect$, then persistent homology decomposes into a collection of indecomposable ``building block'' functors called interval modules, i.e.,~$PH_i(M_v)\cong \bigoplus_I I^{n_I}_{\Bbbk}$, where each $I^{n_I}_{\Bbbk}:(\R,\leq)\to\vect$ assigns $\Bbbk^{n_I}$ to points in the interval\footnote{For a sublevel set filtration of a constructible set, the intervals in this decomposition are always half-open (closed on the left and open on the right). We will ignore all other types of intervals.} $I=[b,d)$, the identity map to $I^{n_I}_{\Bbbk}(s\leq t)$ whenever $s,t\in [b,d)$, and outside of this interval the vector spaces and maps are 0. 
This multi-set of intervals $\mathcal{B}=\{(I,n_I)\}$ defines the \define{barcode} of $\PH_i(M)$, which is alternatively represented as a \define{persistence diagram}; see Figure \ref{fig:mt-bc-pd-explanation}.
It should be noted that historically the first paper to introduce persistent homology and an efficient algorithm for its computation \cite{ELZ-2000} did not utilize the decomposition result of \cite{crawley-boevey}, but rather used the barcode/persistence diagram representation as a convenient device for representing the ranks of each map on homology $\PH_i(s\leq t)$ for arbitrary $s\leq t$.
This perspective was later formalized and generalized by \cite{patel2018generalized}, which observed that the persistence diagram can be defined as the M\"{o}bius inversion of the rank function.

Regardless of its theoretical grounding, the persistence diagram allows to metrize functors from $(\R,\leq)$ to $\vect$ using any one of many Wasserstein-type metrics, as explained in the next two definitions below.

\begin{defn}[Persistence Diagram Space]\label{def:pd-space}
    The set of all possible persistence diagrams, written $\Dgm$, is the set of all countable multi-sets (sets with repetition allowed) of 
    \[
    \R^{2+}:=\{I=(b,d) \in (\{ - \infty\} \cup \R)\times (\R\cup\{\infty\}) \mid b \leq d\}
    \]
    where the number of points of the form $(b, \infty)$ and $(-\infty, d)$ are finite and $\sum_{d-b<\infty} d-b <\infty$.
    Points of the form $(b,\infty)$ or $(-\infty,d)$ are called \define{essential classes} and points of the form $(b,d)$ where neither coordinate is $\infty$ are called \define{inessential classes}.
    For sublevel set filtrations of a constructible set $M$, there are only inessential classes and essential classes of the form $[b,\infty)$.
    We can regard any persistence diagram $\mathcal{B}$ as a set rather than a multi-set by using a second coordinate to enumerate copies of each interval $I$, i.e., $\mathcal{B}=\{(I;j) \mid (I;j) \in \R^{2+}\times \mathbb{N}\}$ where $j$ indicates the $j^{th}$ copy of $I$.
\end{defn}

\begin{defn}[Matchings and the $p$-Wasserstein Distances]\label{def:Wasserstein-distances}
    Suppose $\mathcal{B}=\{(I;j)\}$ and $\mathcal{B}'=\{(I';j)\}$
    are two persistence diagrams.
    A \define{matching} of these is a partial bijection $\sigma:\mathcal{B}\to\mathcal{B'}$, i.e., a choice of subset $\mathcal{M}\subseteq \calB$, called the domain $\dom(\sigma)$, and an injection $\sigma:\mathcal{M}\to\calB'$.
    The complement of the domain of $\sigma$ as $\dom^c(\sigma):=\calB \setminus \dom(\sigma)$ and the complement of the image of $\sigma$ as $\im^c(\sigma):=\calB'\setminus \im(\sigma)$ together define the \define{unmatched points} of $\sigma$.
    We then promote a partial bijection $\sigma$ to an actual bijection via the introduction of diagonal images;
    a point $I=(b,d)\in \R^{2+}$ where neither coordinate is $\infty$ has a \define{diagonal image} $\Delta(I)=(\frac{b+d}{2},\frac{b+d}{2})$.
    With this convention, a partial bijection $\sigma:\calB\to\calB'$ becomes an actual bijection of augmented persistence diagrams $\tilde{\sigma}:\calB(\sigma)\to \calB'(\sigma)$ where
    \[
        \calB(\sigma):=\dom(\sigma) \, \cup \, \dom^c(\sigma) \, \cup \bigcup_{(I',j')\in\im^c(\sigma)} (\Delta(I');j')
    \]
    and
    \[
        \calB'(\sigma):=\im(\sigma)\, \cup\, \im^c(\sigma)\, \cup \bigcup_{(I,j)\in\dom^c(\sigma)} (\Delta(I);j). 
    \]
    The map $\tilde{\sigma}$ now matches points that were previously unmatched by $\sigma$ with their corresponding diagonal images.
    By abuse of notation, we simply write $\sigma$ for the extended map $\calB(\sigma)\to\calB'(\sigma)$. 
    If $\sigma_j(I)$ denotes the $\R^{2+}$ coordinates of $\sigma(I;j)$ and $||I-\sigma_j(I)||_p$ is the $\ell^p$ metric on $\R^{2+}$, then the $p$-\define{cost} of this extended matching $\sigma$ is
    \[
        \cost_p(\sigma)=\left(\sum_{(I,j)\in\calB(\sigma)} ||I-\sigma_j(I)||_p^p\right)^{1/p}.
    \]
    %The $\ell^{\infty}$ distance is $||I-\sigma_j(I)||_{\infty}=\max\{|b_j-b_k|,|d_j-d_k|\}$.

    For every $p\in[1,\infty]$ we define the \define{Wasserstein $p$-distance} between two diagrams $\mathcal{B}$ and $\mathcal{B}'$ is then the infimum of this matching cost over all matchings, i.e.
    \[
        W_{p}(\mathcal{B},\mathcal{B}'):=\inf_{\sigma:\mathcal{B}\to\mathcal{B}'} \cost_p(\sigma).
    \]
    We note that the Wasserstein $\infty$-distance is also called the \define{bottleneck distance}, for which we reserve the special notation
    \[
        d_B(\calB,\calB'):=W_{\infty}(\calB,\calB') = \inf_{\sigma:\calB\to\calB'} \max_{(I,j)\in\calB(\sigma)}||I-\sigma_j(I)||_{\infty}.
    \]
\end{defn}

\begin{figure}
    \centering
    \includegraphics[width=0.6\textwidth]{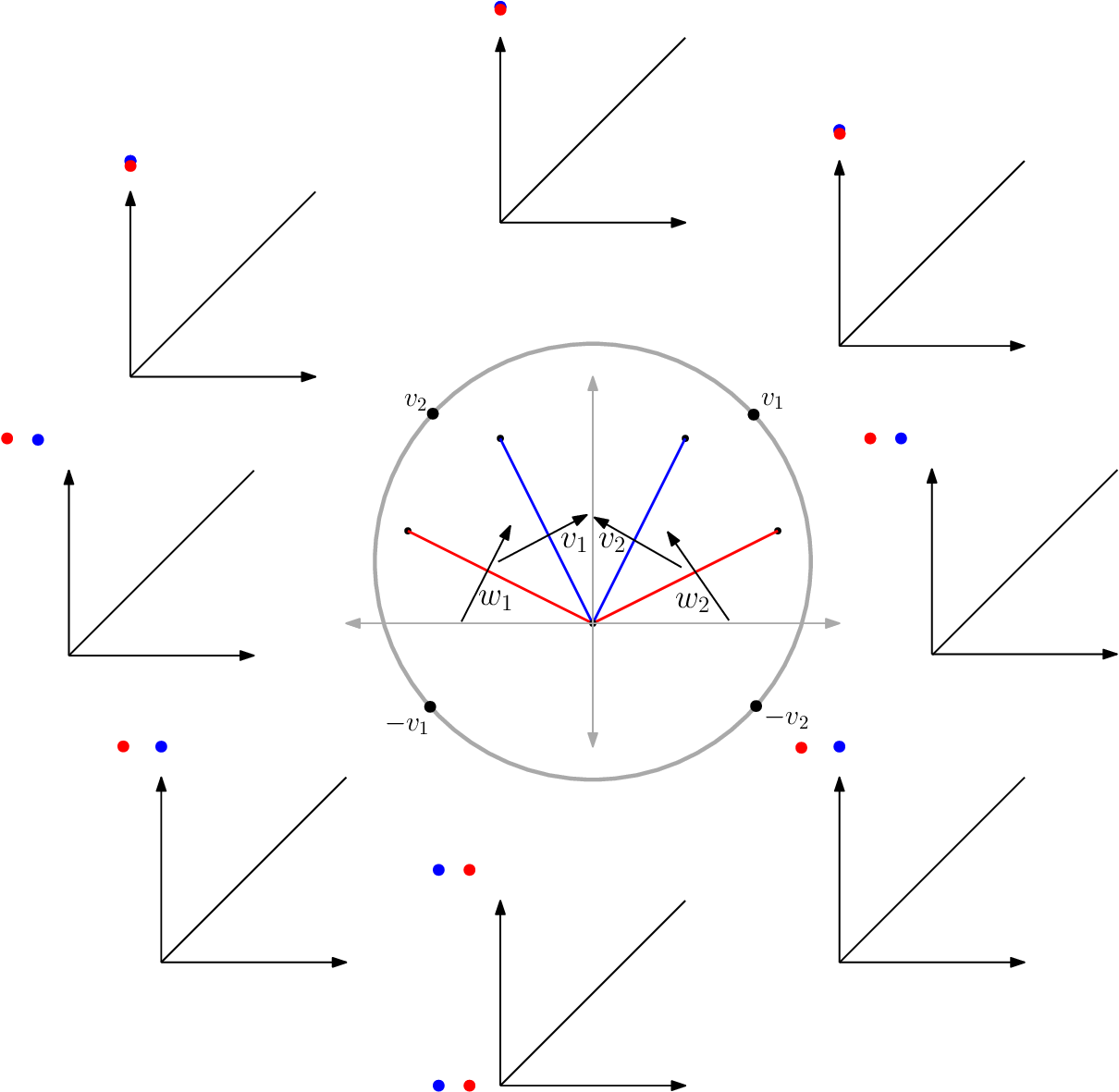}
    \caption{For an embedded shape in $\R^2$, the PHT (Definition \ref{defn:map version}) assigns to each direction $v\in \bS^1$ a persistence diagram. In this example we super-impose the PHT of two different embeddings of the letter `V', one in blue with normal directions $v_1$ and $v_2$ to the sides of the letter `V' and one in red with flatter sides with steeper normal directions $w_1$ and $w_2$. Around the circle are eight persistence diagrams corresponding to the directions $\pm e_i$ and $\pm v_i$ for $i=1,2$; $e_1=(1,0)$ points in the positive $x$ direction and $e_2=(0,1)$ points in the positive $y$ direction. Notice that the blue persistence diagram associated to direction $-e_2$ has an earlier birth time than the red persistence diagram.}
    \label{fig:pht-letter-V}
\end{figure}

We now have enough preliminaries to provide the first, classical definition of the persistent homology transform.

\begin{defn}[PHT: Map Version]\label{defn:map version}
	Let $M\in \CS(\bR^d)$ be a constructible set. The \define{persistent homology transform} of $M$ is defined as the continuous map
	\[
	\pht(M): \bS^{d-1} \to \Dgm^d \qquad v \mapsto (\PH_0(M_v),\dots,\PH_{d-1}(M_v))
	\]
    where $\Dgm$ is equipped with some Wasserstein $p$-distance. See Figure \ref{fig:top-transforms-A} and \ref{fig:pht-letter-V} for examples.
\end{defn}

The Betti Curve Transform (BCT) in degree $i$ of a shape $M$ is determined by the PHT of $M$ in degree $i$ by associating to each vector $v$ and filtration time $t$, the number of points up and to the left of $(t,t)$ in the persistence diagram $\PHT_i(M_v)$.
The Euler Characteristic Transform (ECT) of $M$ is then determined from the BCT of $M$ via the usual formula:
\[
    \ECT(M)(v,t)=\sum_{i=0}^{d-1}(-1)^i \BCT_i(M)(v,t)
\]
Given the injectivity of the ECT (Theorem \ref{thm:ECT-injects}) the following is an immediate corollary.

\begin{cor}[PHT: Injectivity, cf. \cite{cmt,Ghrist:2018aa}]
    If $\PHT(M)=\PHT(N)$, then $\BCT(M)=\BCT(N)$ and $\ECT(M)=\ECT(N)$ and hence, by Theorem \ref{thm:ECT-injects}, $M=N$.
\end{cor}

\subsection{Cohomology and Sheaf Theory}

Alongside homology is an even more powerful topological invariant known as \emph{cohomology}, which was formally introduced in 1935 along with its ring structure \cite{alexander1935chains}.
We will not go through a systematic development of cohomology, but will outline some of the key ideas and notation, as well as its historical context.
This is important because cohomology had a strong influence on the development of sheaf theory \cite{gray2006fragments}, which was initiated by Leray as a prisoner of war (POW) in Austria from 1940-1945 \cite{miller2000Leray}.

For us, the distinction between homology and cohomology is immaterial because we typically consider finite-dimensional homology over a field $\Bbbk$ and so homology and cohomology are isomorphic.
Under this perspective, the difference between homology and cohomology is purely notational: we will now write $\PHT^i(M)$ for the persistent homology transform in degree $i$, instead of $\PHT_i(M)$, to reflect the switch from homology to cohomology, even though this has no effect on the persistence diagrams assigned to each vector $v$.
As a reminder, the concepts introduced in Section \ref{sec:homology-and-BCT} have vector space duals, where one associates to a triangulation of $M$ the group of $n$-cochains $C^n(M;\Bbbk)$, which consists of linear functionals on $C_n(M;\Bbbk)$. 
The sequence of boundary operators $\partial_n:C_n\to C_{n-1}$ then dualize to define (after re-indexing) a sequence of coboundary operators $\delta^n:C^n(M;\Bbbk) \to C^{n+1}(M;\Bbbk)$, that together define the cochain complex $(C^{\bullet},\delta^{\bullet})$. 
The homology of this cochain complex then defines the cohomology groups of $M$:
\[
H^n(M;\Bbbk) := \frac{\ker \delta^n}{\im \delta^{n-1}} \qquad \text{and} \qquad H^n(M)\cong H_n^*(M)
\]
A subtle consequence of this dualization procedure is that functoriality of cohomology is reversed: a map $f:X\to Y$ induces a map $f^*:H^*(Y) \to H^*(X)$ in the opposite direction.

Leray was aware of these recent developments in cohomology theory and was particularly interested in developing a version of cohomology that did not depend on a triangulation or simplicial approximation.
He was also motivated \cite[\S2]{miller2000Leray} to develop a version of de Rham cohomology---based on differential forms---that could work for arbitrary topological spaces---and not just manifolds.
In his POW papers \cite{leray1945forms,leray1945position,leray1945equations}, Leray developed the idea of assigning to each closed subset $V$ of a topological space $X$ a cochain complex $\mathcal{F}^{\bullet}(V)$.
By considering $V=\{x\}$ for each point $x\in X$, one can then consider how the cohomology groups $H^i(\mathcal{F}^{\bullet}(x))$ vary across the space $X$. 
%in a spirit very similar to persistent homology.

It is now understood \cite{curry-thesis} that traditional persistent homology can be viewed as a special case of sheaf theory, in the sense that the filtration parameter space $\R$ indexes the (co)chain complexes of the spaces $M_{v,t}$.
Moreover, as first outlined in \cite{cmt}, sheaf theory also allows us to view the persistent homology transform (PHT) as a collection of (co)chain complexes parameterized by the entire parameter space $\bS^{d-1}\times\R$.
In this section we review this latter construction in more detail using the modern convention of sheaves defined on open sets, before moving onto the more general versions on sheaf theory---developed over the 80 years since Leray---that make our current contributions possible.
%Understanding how (co)homology of a space varies with reference to a map to a base space was Leray's motivation behind the development of sheaf theory, but it quickly became adapted for more general purposes over the intervening 80 years.

\begin{defn}[Pre-Sheaves and Stalks]
    Let $\bX$ be a topological space and let $\Open(\bX)$ be the poset of open sets in $\bX$.
    A \define{pre-sheaf} of vector spaces on $\bX$ is a functor $\F:\Open(\bX)^{op}\to \Vect$.
    We sometimes write $\rho_{U,V}:\F(U)\to \F(V)$ for the restriction map associated to the inclusion $V\subseteq U$.
    The \define{stalk} of a pre-sheaf at a point $x\in \bX$ is then defined as the direct limit\footnote{The direct limit or injective limit are all examples of colimits. See \cite[\S3]{riehl2017category} for a good introduction to limits and colimits.} of $\F$ over all open sets containing $x$, i.e.,
    \[
        \F_x := \varinjlim_{U\ni x} \F(U) 
    \]
    One can think of the stalk at $x$ as the ``local picture'' of $\F$ at $x$. This can be constructed rigorously as the vector space of equivalence classes under restriction, i.e., $u\in F(U)$ and $v\in F(V)$ are equivalent if $\exists W\subseteq U\cap V$ with $x\in W$ and $\rho_{U,W}(u)=\rho_{V,W}(v)$. 
\end{defn}

\begin{defn}[\v{C}ech Cohomology of a Cover]
    Given a pre-sheaf $\F$ and an open cover $\cU=\{U_i\}_{i\in \Lambda}$ of an open set $U$, one has the \define{\v{C}ech cochain complex} associated to $\cU$ where the $n$-cochains are given by the product over intersections of $n+1$ cover elements, i.e.,
    \[
        \check{C}^n(\cU;\F)=\prod_{i_0,\ldots,i_n \in \Lambda} \F(U_{i_0,\dots,i_n})
    \]
    where we always assume $\F(\varnothing)=0$.
    The $n^{th}$-coboundary operator is defined by specifying the contribution of a general element $s\in \check{C}^n(\cU;\F)$ to the factor $i_0,\ldots,i_{n+1}$ in $\check{C}^{n+1}(\cU;\F)$. This is given by the formula
    \[
        (\delta^n s)_{i_0,\ldots,i_{n+1}} = \sum_{j=0}^{n+1} (-1)^j s_{i_0,\dots,\hat{i_j},\dots,i_{n+1}} |_{U_{i_0,\ldots,i_{n+1}}},
    \]
    where $\hat{i_j}$ denotes removal of that entry and the vertical line is the commonly accepted short hand for the application of the restriction map from $U_{i_0,\dots,\hat{i_j},\dots,i_{n+1}}$ to $U_{i_0,\ldots,i_{n+1}}$.
\end{defn}

\begin{defn}[Sheaves]\label{defn:sheaf}
    A pre-sheaf of vector spaces $\F$ is a \define{sheaf} if for every open set $U$ and open cover $\cU=\{U_i\}_{i\in \Lambda}$ we have that the value of $\F$ on $U$ can be computed using the \v{C}ech cohomology of the cover:
    \[
        \F(U) \cong H^0\bigg[ \check{C}^0(\cU;\F) \to \check{C}^1(\cU;\F) \to \check{C}^2(\cU;\F) \to \cdots \bigg]
    \]
    This is a \emph{local-to-global principle}, because it guarantees that a sheaf is always determined locally by a cover.
    More generally, one can define sheaves valued in categories that are not necessarily abelian, such as $\Set$, but where the categorical notion of an (inverse) limit makes sense. 
    As a reminder, the limit\footnote{Kernels, equalizers, and pullbacks are all examples of limits. See \cite[\S3]{riehl2017category} for more background.} is a categorical operation that takes an entire diagram of objects---in this case the objects $\F(U_i)$---and produces a single object that maps to each of the objects and commutes with the morphisms connecting the objects---in this case the restriction maps $\F(U_i) \rightarrow \F(U_{ij})\leftarrow \F(U_j)$.
    One then modifies the sheaf axiom to instead require that
    \[
    \mathcal{F}(U) \cong \varprojlim  \bigg[ \prod_{i}\mathcal{F}(U_i) \rightrightarrows \prod_{i,j}\mathcal{F}(U_{ij}) \rightthreearrow{} \cdots   \bigg].
    \]
    It is well-known that the limit only depends on the objects $\{\F(U_i)\}$ associated to the cover elements along with the objects associated to the pairwise-intersections $\{\F(U_{ij})\}$, in similar spirit to how $H^0$ only depends on the nodes and edges of a triangulated space. Higher order information, such as intersections of three or more elements, is ignored by the usual sheaf axiom.
\end{defn}

\begin{defn}
Let $\Dat$ be a complete ``data'' category, i.e., all limits in $\Dat$ exist. 
Denote the category of pre-sheaves and sheaves on $\bX$ valued in $\Dat$ by $\PShv(\bX;\Dat)$ and $\Shv(\bX;\Dat)$, respectively.
Since every sheaf is a pre-sheaf there is a natural inclusion of categories 
\[
\pshf: \Shv(\bX;\Dat) \hookrightarrow \PShv(\bX;\Dat).
\]
Under suitable hypotheses\footnote{See condition (17.4.1) in \cite{kashiwara}.} on $\Dat$, the $\pshf$ has an ``inverse''\footnote{i.e.~a left adjoint, see \cite[\S4]{riehl2017category}.} called \define{sheafification} 
\[
\shff:\PShv(\bX;\Dat) \to \Shv(\bX;\Dat).
\]
Typically we set $\Dat=\Vect$ and let $\Shv(\bX)$ denote the category of sheaves of vector spaces.
\end{defn}

\begin{exmp}[Constant and Locally Constant Sheaves]\label{exmp:constant-sheaf}
    The pre-sheaf that assigns the vector space $\Bbbk$ to every open set is not a sheaf because the value that should be associated to an open set $U=U_1 \cup U_2$ with two components should be $\Bbbk^2$.
    The sheafification of this constant pre-sheaf is called the \define{constant sheaf} and is usually written $\Bbbk$ as well.
    One can replace $\Bbbk$ with any vector space $\mathbb{V}$ and obtain an analogous constant sheaf.
    Moreover, if a sheaf $\F\in \Shv(\bX)$ has the property that for each point $x\in \bX$ there is a neighborhood $U_x$ so that $\F$ restricts to a constant sheaf on $U_x$, then one calls $\F$ a \define{locally constant sheaf} or \define{local system}.
\end{exmp}

\begin{rmk}[Preservation of Stalks]\label{rmk:stalks-shff}
    It is a fact that sheafification preserves stalks, so the ``local picture'' of a pre-sheaf $\F$ is unchanged by this process and only the failure of the local-to-global principle is repaired.
\end{rmk}

\begin{defn}[Pushforward and Pullback of Sheaves]\label{defn:pushforward-pullback}
    Suppose $f:\bX \to \bY$ is a continuous map of spaces, $\F$ is a sheaf on $\bX$, and $\calG$ is a sheaf on $\bY$.
    The \define{pushforward} (or direct image) of $\F$ along $f$, written $f_* \F$ is the sheaf that assigns to each open set $V\in \Open(\bY)$ the value $\F(f^{-1}V)$.
    Dually, the \define{pullback} (or inverse image) of $\calG$ along $f$, written $f^* \calG$, is the sheaf associated to the pre-sheaf that assigns to every open set $U\in \Open(X)$ the ``stalk at $f(U)$'':
    \[
        \shff \big[U \mapsto \varinjlim_{V\supseteq f(U)} \calG(V) \big]
    \]
    These define functors
    \[
    f_*: \Shv(\bX)\to \Shv(\bY) \qquad \text{and} \qquad f^*:\Shv(\bY)\to\Shv(\bX).
    \]
\end{defn}

We now isolate a very important class of pushforward sheaves, which we call Leray sheaves.

\begin{defn}[Leray Sheaves]\label{defn:Leray-sheaves}
Suppose $f:\bY \to \bX$ is a proper continuous map of spaces, then the $i^{th}$ \define{Leray sheaf of $f$}, written $R^if_*\Bbbk$, is the sheaf associated to the pre-sheaf
\[
\shff \big[U \mapsto H^i(f^{-1}U;\Bbbk)\big].
\]
For $i=0$ this is just the pushforward of the constant sheaf, but for $i\geq 1$ this definition arises naturally from the ``derived'' perspective considered in Section \ref{sec:derived-sheaf-theory}.
\end{defn}

\begin{rmk}[Cohomology of the Fiber]\label{rmk:fiber}
    As Remark \ref{rmk:stalks-shff} indicates, the stalk of the Leray pre-sheaf is preserved under sheafification. Consequently, if $\varinjlim_{U\ni x} H^{i}(f^{-1}U))\cong H^i(f^{-1}(x))$, then this sheaf records the cohomology of the fiber.
    This last consequence follows from properness of $f$, which plays a key role in the proper base change theorem \cite{iverson}.
\end{rmk}

We now have enough language to describe the PHT sheaf-theoretically.

\begin{defn}[PHT: Sheaf Version, cf.~\cite{cmt}]\label{defn:pht-sheaf}
Let $M\in \CS(\R^d)$ be a constructible set.
Associated to $M$ is the \define{auxiliary total space}
\[
Z_M := \{(x,v,t) \in M\times \bS^{d-1}\times\R \mid x\cdot v \leq t\}.
\]
The $i^{th}$ \define{persistent homology transform sheaf} of $M$, written $\PHT^i(M)$, is the $i^{th}$ Leray sheaf of the map $f_M:Z_M \to \bS^{d-1}\times\R$ that projects onto the last two factors.
Since $M$ is compact and $f_M$ is a projection, we see that $f_M$ is proper. By Remark \ref{rmk:fiber}, the stalk of the $i^{th}$ Leray sheaf at $(v,t)\in \bS^{d-1}\times \bR $ is isomorphic to the $i^{th}$ cohomology of the fiber i.e., $H^i({f_M}^{-1}(v,t))$. 
See Figure~\ref{fig:vshape} for a stalk-wise picture of $\PHT^0$ of the shape `V', which was considered from the map perspective in Figure \ref{fig:pht-letter-V}.
\end{defn}

\begin{figure}
     \centering
     \begin{subfigure}[b]{0.4\textwidth}
         \centering
         \includegraphics[width=0.6\textwidth]{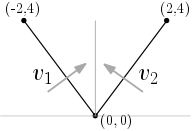}
         \caption{$M$ is an embedded `V' shape.}
         \label{fig:V}
     \end{subfigure}
     
     \bigskip
     % \hspace{-1cm}
     \begin{subfigure}[b]{0.45\textwidth}
         \centering
        \includegraphics[width=1.1\textwidth]{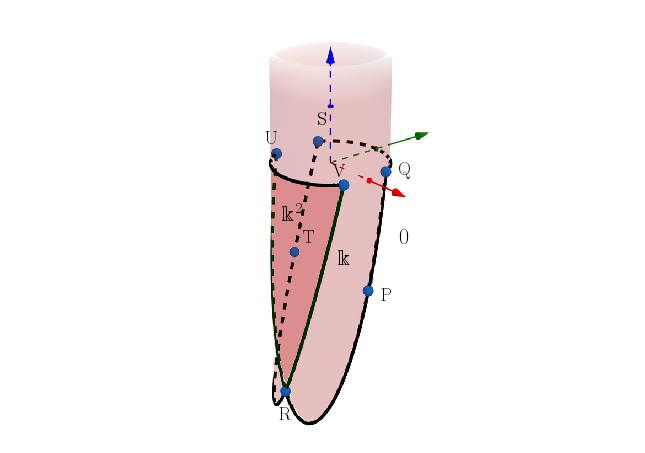}
         \caption{$\PHT^0(M)$}
         \label{fig:cylV}
     \end{subfigure}
    %\hfill
     \begin{subfigure}[b]{0.45\textwidth}
         \centering
         \includegraphics[width=.8\textwidth]{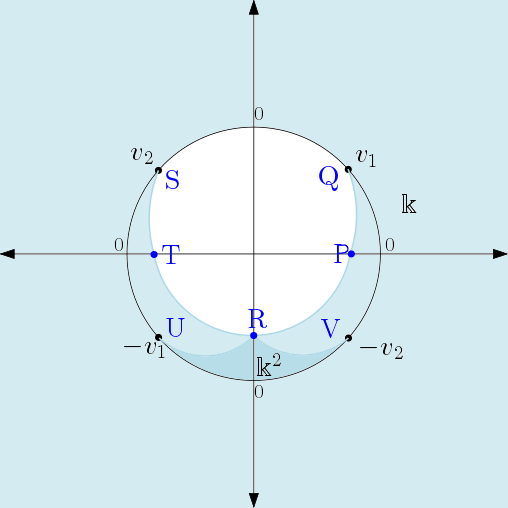}
         \caption{$\pht^0(M)$ viewed on the plane.}
         \label{fig:phtv}
     \end{subfigure}
        \caption{Figure~(B) is a visualization of the PHT sheaf in degree 0 of $M=$ `V' on the cylinder $\bS^1\times \bR$. The dark pink region represents where there are two connected components and the light pink region represents one connected components. The green arrow represents the direction $e_2=(0,1)$ and the the red axis corresponds looking to the right $e_1=(1,0)$.
        To visualize the sheaf on the plane, as in Figure~(C), we map the cylinder $\bS^1\times \bR$ to the plane $\bR^2$ by sending the circle at $t=-\infty$ to $(0,0)$ and the circle at $t=+\infty$ goes to infinity in each direction.} 
        %Mores specifically: fix direction $v \in \bS^1$ and then map $\{v\}\times \bR$ onto $(0,\infty)$. So every direction $v$ has a ray attached to it.}  
        %For example, the direction $e_2 = (0,1)$ at value $t=0$ then in Figure (B) we see $H^0(M_{e_2,0},\Bbbk)= \Bbbk$ since $M_{e_2,0}= \{ x\in M \mid x\cdot e_2 \le 0 \} =\{(0,0)\} $.}
    \label{fig:vshape}
\end{figure}

\begin{rmk}
If we restrict the sheaf $\pht^i(M)$ to the subspace $\{v\}\times \bR$, then one obtains a constructible sheaf that is equivalent to the persistent (co)homology of the filtration of $M$ viewed in the direction of $v$. The persistence diagram in degree $i$ is simply the expression of this restricted sheaf in terms of a direct sum of indecomposable sheaves.
\end{rmk}

\begin{rmk}[Iterated Pushforwards and the Equivariance Property]\label{rmk:equivariance}
    An advantage of the sheaf-theoretic definition of the PHT is that there are further operations that can be performed on it. 
    For example, one can iterate the Leray sheaf construction for any change of coordinates of $\bS^{d-1}\times \R$.
    For example, a rigid motion $g\in \mathrm{SO}(d)$, induces an action on the corresponding PHT sheaves $\PHT^i(M)$ via pushforward along $g$. In particular, we have the following equivariance formula:  
    \begin{equation*}
        \PHT^i(g \cdot M) = g_* \PHT^i(M)
    \end{equation*}
    To prove the above formula, it suffices to show for any test open $U\subseteq\mathbb{S}^{d-1}\times \mathbb{R}$, 
    $$ f^{-1}_{gM}(U) = f_M^{-1}\left(g^{-1}(U)\right). $$ The equality of the sets follows from the definition of the projection map $f_M$ (Definition~\ref{defn:pht-sheaf}) and the fact that $gx\cdot v = x\cdot g^{-1}v $ for any $x\in M$ and $v\in S^{d-1}$.  In other words, filtering a shape rotated by $g$ in direction $v$ is equivalent to filtering the original shape in direction $g\cdot v$.
    A similar expression can be shown for translations. In particular, if $g$ represents translation by $T\in \bR^{d}$. Then it induces a map $g:\bS^{d-1}\times \bR\to \bS^{d-1}\times \bR$ by $(v,t) \mapsto (v,t+T\cdot v)$.
\end{rmk}

\subsection{Derived Sheaf Theory}\label{sec:derived-sheaf-theory}

There is a third way of describing the persistent homology transform that requires the language of derived categories. 
In many ways this third perspective follows naturally the history of cohomology and sheaf theory.
Cochain complexes of vector spaces generalize naturally to cochain complexes of sheaves, which motivated \cite{weibel-history} the development of a unified theory that could handle cohomology and the related algebra of all these different situations.
This led to the development of homological algebra, which was systematically laid out by Cartan and Eilenberg in 1956 \cite{cartan-eilenberg}.
This was quickly followed by Grothendieck's famous \emph{Tohoku} paper \cite{tohoku}, which introduced the general notion of an abelian category, where the concepts of (co)homology make sense.
Grothendieck and his student, Verdier, then laid the groundwork for derived category theory, which permit certain dualities and where quasi-isomorphisms (isomorphisms on cohomology) are genuine isomorphisms.
Verdier's 1967 thesis, published later in two installments \cite{verdier1977} and \cite{verdier1996}, crystallized the algebraic aspects of sheaf theory for nearly 50 years as it became the standard treatment used by textbooks, e.g., \cite{kashiwara}, to this day.
However, our efforts to ``glue'' PHTs---carried out in Section \ref{sec:htpy-sheaf}---also illustrate the deficiencies of the derived category, cf. \cite[p.~1261]{hormann2017}, and motivate \cite{hormann2019} the $\infty$-category perspective outlined in our Appendix \ref{app:infinity}.

We now recall one definition of the derived category, as there are many. 

\begin{defn}
Let $\Acat$ be an abelian category, in particular every morphism has a kernel and cokernel.
Consider the category of bounded chain complexes of objects in $\Acat$, written $\cC^b(\Acat)$.
Associated to this category is the \define{homotopy category} $\cK^b(\Acat)$ of chain complexes, which has the same objects as $\cC^b(\Acat)$, but where morphisms are homotopy classes of chain maps.
Recall that a chain map $\varphi:(A^{\bullet},d_A) \to (B^{\bullet},d_B)$ is a \define{quasi-isomorphism} if it induces isomorphisms on all cohomology groups
\[
    H^i(\varphi):H^i(A^{\bullet}) \to H^i(B^{\bullet}).
\]
Let $\cQ$ denote the class of quasi-isomorphisms.
The \define{bounded derived category} of $\Acat$ is the localization of $\cK^b(\Acat)$ at the collection of morphisms $\cQ$, i.e.
\[
\cD^b(\Acat):=\cK^b(\Acat)[\cQ^{-1}].
\]
\end{defn}

\begin{rmk}\label{rmk:htpy-inj}
An alternative definition of the derived category makes use of the assumption that $\Acat$ has enough injectives, i.e. every object in $\Acat$ has an injective resolution or, said differently, every object in $\Acat$ is quasi-isomorphic to a complex of injective objects. 
Under this assumption, the derived category of $\Acat$ is equivalently defined as the homotopy category of injective objects in $\Acat$, i.e.
\[
\cD^b(\Acat) \simeq \cK^b(\text{Inj}-\Acat).
\]
\end{rmk}

Remark \ref{rmk:htpy-inj} provides an easier to understand prescription for working with the derived category. One simply takes an object, e.g.~a sheaf, replaces it with it's injective resolution and works with the resolution instead.

\begin{defn}
Suppose $F:\Acat \to \Bcat$ is an additive and left-exact functor, i.e.~it commutes with direct sums and preserves kernels, then the \define{total right derived functor} of $F$, written $RF:\cD^b(\Acat) \to \cD^b(\Bcat)$ is defined by
\[
    RF(A^{\bullet}):=F(I^{\bullet})
\]
for $I^{\bullet}$ an injective resolution of $A^{\bullet}$. In general, one can substitute $I^{\bullet}$ with any $F$-acyclic resolution of $A^{\bullet}$.
Such resolutions are said to be \define{adapted} to $F$.
\end{defn}

We can now define the persistent homology transform as a derived sheaf.

\begin{defn}[PHT: Derived Version, cf.~\cite{cmt}]\label{defn:derived-PHT-sheaf}
Let $M\in\CS(\R^d)$ be a constructible set. 
Let $Z_M$ be the auxiliary space construction from Definition \ref{defn:pht-sheaf}.
Let $f_{M*}:\Shv(Z_M) \to \Shv(\bS^{d-1}\times\R)$ be the pushforward (or direct image) functor along the projection map $f_M:Z_M \to\bS^{d-1}\times\R $.
The \define{derived PHT sheaf} is
\[
    \PHT(M):= Rf_{M*}\Bbbk_{Z_M} \in \cD^b(\Shv(\bS^{d-1}\times\R)).
\]
More explicitly we can describe this right-derived pushforward as follows:
For a topological space $\bX$ we let $\mathcal{S}^p(U;\bk)$ denote the group of singular $p$-cochains of $U\subset \bX$ with coefficients in $ \bk $. 
Define $\mathscr{S}^p(X,\bk)= \shff(U\mapsto \mathcal{S}^p(U;\bk))$ where $\shff$ stands for sheafification.
The constant sheaf $\bk_{Z_{M}}$ admits a flabby resolution by singular cochains:
	\begin{equation*}
	0\to \bk_{Z_{M}}\to \mathscr{S}^0(Z_{M};\bk)\to \mathscr{S}^1(Z_{M};\bk) \to \mathscr{S}^2(Z_{M};\bk)\to \cdots
	\end{equation*} 
Because flabby resolutions form an adapted class for the pushforward functor \cite{bredon} we can describe $\PHT(M)$ as the pushforward of the complex of sheaves of singular cochains:
	\begin{equation*} %\label{flabby resolution of constant sheaf}
	Rf_{M*}\bk_{Z_{M}} := f_{M*}\mathscr{S}^0(Z_{M_i};\bk)\to f_{M*}\mathscr{S}^1(Z_{M_i};\bk) \to \cdots
	\end{equation*}
 Taking $i^{th}$ cohomology of this complex---written $\mathcal{H}^i$ in the category of sheaves---produces the $i^{th}$ PHT sheaf of Definition~\ref{defn:pht-sheaf}.
\end{defn}

\subsection{Constructible Sheaves and their Functions}\label{sec:sheaf-to-function}

We finish our review of preliminary material by showing how our first two topological transforms---the Euler characteristic and Betti curve transforms---are recovered by the derived perspective.
This is accomplished via the sheaf-to-function correspondence, which is best studied for constructible sheaves.

\begin{defn}[Constructible and Cellular Sheaves]
A sheaf $\F\in \Shv(\bX)$ is said to be \define{constructible} if there is a decomposition of $\bX$ into definable subsets $\{\bX_{\alpha}\}$ so that for each $\alpha$ the pullback of $\F$ along the inclusion $i_{\alpha}:\bX_{\alpha}\hookrightarrow \bX$ produces a locally constant sheaf $i^*_{\alpha} \F\in \Shv(\bX_{\alpha})$, cf. Example \ref{exmp:constant-sheaf}.
When the subsets $\{\bX_{\alpha}\}$ are cells in a triangulation of $\bX$, we can equivalently express $\F$ as a \define{cellular sheaf}, which simply assigns a vector space to each cell $\bX_{\alpha}$ and a linear map to each pair $\bX_{\alpha}\subseteq \overline{\bX_{\beta}}$; see \cite{shepard1985cellular} for the first description of cellular sheaves and \cite{curry-thesis} for a modern treatment.

In similar spirit, if $\mathcal{F}^\bullet$ is a complex of sheaves, e.g., a derived sheaf, then it is said to be \define{cohomologically constructible} if the cohomology sheaves $\mathcal{H}^i\mathcal{F}^\bullet$ are constructible for every $i$.
A derived cellular sheaf then associates to each cell in a cellulation a complex of vector spaces and a chain map for every face-relation pair $\bX_{\alpha}\subseteq \overline{\bX_{\beta}}$. 
\end{defn}

We now show how to associate to a constructible function to a constructible sheaf.

\begin{defn}[Euler-Poincar\'e and Hilbert Functions]
Let $\mathcal{F}^\bullet$ be a complex of cohomologically constructible sheaves on $\bX$. 
The \define{local Euler-Poincar\'e index} is the piecewise constant integer-valued function defined by
  \begin{equation*}
      h(x):= \chi(\mathcal{F}^\bullet)(x) = \sum_i (-1)^{i} \dim \mathcal{H}^i(\mathcal{F}^{\bullet})_x = \sum_i (-1)^{i} \dim (H^i\mathcal{F}^{\bullet}_x)
  \end{equation*}
The first equality considers the stalks of the cohomology sheaves of $\F^{\bullet}$ and the second equality considers the cohomology of the stalk complex. By Remark \ref{rmk:stalks-shff} these are isomorphic and thus have the same dimension and Euler characteristic.
In the simplest setting---constructible sheaves that are concentrated in a single cohomological degree---the local Euler-Poincar\'e index is just the \define{Hilbert function}---it records the dimension of the stalks of a sheaf.
\end{defn}

% For complexes of sheaves, the dimension function is replaced by the point-wise Euler characteristic of the complex.
% This index was used by Kashiwara to prove that the Grothendieck group of constructible sheaves is isomorphic to the group of constructible functions on $X$ \cite[Thm 9.7.1]{kashiwara}.
 
 % \begin{exmp}
 %  The function associated to a single sheaf $\mathcal{F}$ is $ h(x) = \dim(\mathcal{F}_x)$. This can be seen by the Euler-Poincar\'e index of \[\mathcal{F} \to 0 \to 0 \to \cdots\]
 % \end{exmp}
 % A sheaf $\mathcal{F}$ can thought of a complex of sheaves, 
  %\[\mathcal{F} \to 0 \to 0 \to \cdots\] and so the constructible function turns out to be $ h(x) = \dim(\mathcal{F}_x)$.
  
For a constructible set $M$, the derived PHT sheaf $\pht(M)= R(f_M)_*\Bbbk_{Z_M}$ is constructible \cite{cmt} and so by the above correspondence there is a constructible function associated to it. 
%This function is the Euler Characteristic Transform (ECT) \cite{cmt}.
   
   % Recall the derived PHT sheaf, (Equation~(\ref{flabby resolution of constant sheaf}))
   % \[	Rf_{M*}\bk_{Z_{M}} := f_{M*}\mathscr{S}^0(Z_{M_i};\bk)\to f_{M*}\mathscr{S}^1(Z_{M_i};\bk) \to \cdots
   % \]
   % The function given by the local Euler-Poincar\'e index of the above equation is
   % \[  \chi(\pht(M))(v,t) = \sum_i (-1)^i \dim H^i\big(f_M^{-1}(v,t);\Bbbk\big)= \chi(f_M^{-1}(v,t)).  \]
    
  \begin{defn}[ECT and BCT: Sheaf-to-Function Version]
    The ECT of a constructible set $M$ is the local Euler-Poincar\'e index of the derived PHT sheaf of Definition \ref{defn:derived-PHT-sheaf}:
    \[ \text{ECT}(M)(v,t) = \chi\big(R^{\bullet} f_{M*} \Bbbk_{Z_M}\big)(v,t) 
    \]
    Similarly, the BCT in each degree $i$ is just the Hilbert function of the $i^{th}$ PHT sheaf of Definition \ref{defn:pht-sheaf}.
  \end{defn}

\section{A Homotopy Sheaf on Shape Space}\label{sec:htpy-sheaf}

As mentioned in the introduction, we want to build a shape space using a sheaf-theoretic construction on the poset of constructible sets $\CS(\R^d)$.
Naively one would like to prove that the association
    \[
	\F:\CS(\bR^d)^{op} \to \D^b(\Shv(\bS^{d-1}\times \bR)) \qquad
	M \mapsto  \PHT(M)
	\]
is a sheaf, but there are two main obstacles.

The first obstacle is that a topology on $\CS(\R^d)$ needs to be specified.
Although sheaves on posets are well-defined via the Alexandrov topology---see \cite{curry-thesis} for a modern treatment---the poset under consideration is infinite and using the Alexandrov topology here would imply that a shape can be determined via a cover by its points; this is clearly impossible as there is not enough of an interface between points to determine homology.
This first obstacle is neatly handled by restricting the types of covers we're willing to consider---finite closed covers, to be precise---and is formalized using Grothendieck topologies and sites (Definition \ref{defn:site}).
However, the second obstacle is far more subtle and requires a departure from the usual sheaf axiom of Definition \ref{defn:sheaf}.
To see why, we show that a candidate ``local-to-global'' principle for the Euler characteristic transform is simply the inclusion-exclusion principle (Theorem \ref{thm:decat}).
This in turn motivates the need for a sheaf axiom that considers higher-order intersections, which we call the \emph{homotopy} sheaf axiom or \v{C}ech descent (Definition \ref{def:homotopysheaf}).
This then implies our main gluing result for the PHT (Theorem \ref{thm:main-hpty-sheaf}).

\subsection{Inclusion-Exclusion for the ECT}

As a reminder, Definition \ref{defn:radon-ECT} presented the Euler characteristic transform as the Radon transform $\mathcal{R}_S$ associated to a particular relation $S\subseteq X\times Y$, where $X=\R^d$ and $Y=\bS^{d-1}\times \R$.
One of the stated properties of the Radon transform is that it defines a group homomorphism
\[
\mathcal{R}_S : \CF(X) \to \CF(Y) \qquad \text{i.e.} \qquad \mathcal{R}_S(\phi +\psi)=\mathcal{R}_S(\phi)+\mathcal{R}_S(\psi).
\]
This allows us to prove an immediate inclusion-exclusion principle for the ECT.

  \begin{thm}\label{thm:decat}
      For a finite cover $\mathcal{M}=\{M_i\}_{i\in \Lambda}$ of $M \subset \mathbb{R}^d$ by constructible subsets 
       \[ 
       \ECT(M) = \sum_{I\subset \Lambda}(-1)^{|I|+1}\ECT(M_I)  
       \]
       where each $M_I$ denotes the intersection $M_{i_1}\cap M_{i_2} \cap \cdots \cap M_{i_k}$ for $I=(i_1,...,i_k)$. 
  \end{thm}
  \begin{proof}
    The inclusion-exclusion principle allows us to write the indicator function as
    \begin{equation*}
        \mathbb{1}_M = \sum_{I\subset \Lambda}(-1)^{|I|+1}\mathbb{1}_{M_I}.
    \end{equation*}
    %which is exactly akin to the local Euler-Poincar\'e index of Godemont's resolution from Equation \ref{godemont-resolution}.
    Linearity of the Radon transform then implies that
    \begin{equation*}
        \mathcal{R}_{S}\mathbb{1}_M = \sum_{I\subset \Lambda}(-1)^{|I|+1}\mathcal{R}_{S}\mathbb{1}_{M_I},
    \end{equation*}
    which is the expression using ECTs written above.
    %This is exactly the local Euler-Poincar\'e index of the pushforward of the resolution in Equation \ref{godemont-resolution}.
    %Checking on stalks reveals that for any $(v,t)\in \bS^{d-1}\times \bR$
    %\[ \chi\big(f^{-1}_M(v,t)\big) = \sum_{I\subset \Lambda}(-1)^{|I|+1}\chi\big(f^{-1}_{M_I}(v,t)\big).  \]
  \end{proof}

  \begin{rmk}\label{rmk:decat}
      If we take the sheaf-to-function correspondence of Section \ref{sec:sheaf-to-function} seriously, then Theorem \ref{thm:decat} should be viewed as a decategorification of a deeper sheaf-theoretic result:
      \begin{equation*}
          \PHT(M) \cong \text{``summary of'' } \bigg[ \prod \PHT(M_i) \to \prod \PHT(M_{ij}) \to \prod \PHT(M_{ijk}) \to \cdots \bigg]
      \end{equation*}
      However, unlike the sheaf axiom of Definition \ref{defn:sheaf}, this summary operation cannot only use the $M_i$ and their pairwise intersections $M_{ij}$.
      Indeed the inclusion-exclusion result of Theorem \ref{thm:decat} says that \emph{all higher order terms must be considered}.
      This summary operation is accomplished by the homotopy limits of Definition \ref{defn:homotopy-limit} or via limits in the derived $\infty$-category (Appendix \ref{app:infinity}).
  \end{rmk}

%This obstacle is neatly handled by imposing the condition that only finite closed covers can participate in a valid covering of a shape.
% The second obstacle is fatal for a naive sheaf-theoretic approach: the pre-sheaf $U\mapsto H^i(U)$ is not a sheaf for $i\geq 1$.
% Indeed, the connecting homomorphism in the Mayer-Vietoris long-exact sequence quantifies precisely the failure of the sheaf axiom.
% Both of these obstacles are addressed via tools from ``higher'' sheaf theory: Grothendieck topologies, homotopy sheaves, and $\infty$-sheaves (presented in Appendix \ref{app:infinity}).
% We develop the first two parts of this machinery now.

\subsection{Sites and Homotopy Sheaves}
Grothendieck topologies provide a way of generalizing sheaves to contravariant functors on a general category $\cC$, like the poset $\CS(\R^d)$. 
Covers of an open set are replaced with collections of morphisms that have certain ``cover-like'' properties.

\begin{defn}[Grothendieck Pre-topology, cf.\cite{artin}]\label{defn:site}
		Let $\cC$ be a category with pullbacks.
		A \define{basis for a Grothendieck topology} (or a \define{pre-topology}) on $\mathcal{C}$ requires specifying for each object $U\in\cC$ a collections of \define{admissible covers} of $U$.
		This collection of covers must be closed under the following operations:
		\begin{enumerate}
			\item (Isomorphism) If $f:U'\to U$ is an isomorphism then $\{f:U' \to U\}$ is a cover.
			\item (Composition) If $\{f_i: U_i\to U\}$ is a cover of $U$ and if for each $i$ we have a cover $\{g_{i,j}:U_{i,j} \to U_{i}\}$ then the composition $\{f_i\circ g_{i,j}: U_{i,j}\to U\}$ is also a cover.
			\item (Base Change) If $\{f_i:U_i \to U\}$ is a cover and $V\to U$ is any morphism then the pullback $\{\pi_i: V\times_U U_i \to V \} $ is a cover as well 
		\end{enumerate}
		As the name suggests, the above data specifies a genuine Grothendieck topology on $\cC$.
		A category equipped with a Grothendieck topology is known as a \define{site}. 
\end{defn}

\begin{rmk}[Sheaves on Sites]\label{rmk:sheaves-on-sites}
The classical definition of a pre-sheaf and sheaf can now be generalized to a site. 
A functor $\cF:\cC^{op}\to \Dat$ is a \define{pre-sheaf}.
If $\Dat$ has all limits, we say a pre-sheaf is a \define{sheaf} if for every object $U\in \cC$ and cover $\cU=\{f_i: U_{i} \to U\}$
 \[
 \cF(U) = \varprojlim \bigg[ \prod_{i}\mathcal{F}(U_i) \rightrightarrows \prod_{i,j}\mathcal{F}(U_{ij})\bigg].
 \]
Here equality means isomorphic up to a unique isomorphism and $U_{ij}$ is the pullback of $f_j: U_j\to U$ along $f_i:U_i \to U$ for any pair of morphisms $f_i$ and $f_j$ that participate in the cover $\cU$.
\end{rmk}

Unfortunately the functor $\cF$ specified in Theorem \ref{thm:htpy-sheaf} is valued in the derived category of sheaves on $\bS^{d-1}\times \R$.
It is well-known among experts that the derived category $\cD^b(\Acat)$ of an abelian category $\Acat$ is not abelian.
Candidate kernels and co-kernels do not have canonical inclusion and projection maps, but one can work with so-called \emph{distinguished triangles} instead.
More generally, we can describe a sheaf axiom whenever the notion of a \emph{homotopy limit} makes sense in the target category $\Dat$.
We recall a special case of this construction for $\Dat=\cD^b(\Acat)$.

\begin{defn}[Homotopy Limits]\label{defn:homotopy-limit}
Given an inverse system of objects $(K_n,f_n)$ in $\cD^b(\Acat)$
\[
\cdots \xrightarrow{f_{n-1}} K_{n-1} \xrightarrow{f_n} K_{n} \xrightarrow{f_{n+1}} \cdots
\]
an object $K$ is a \define{homotopy limit} if there is a distingushed triangle in the derived category
\[
K \to \prod_{n}K_n \xrightarrow{\text{shift}} \prod_n K_n \to K[1].  
\]
The shift map being given by $(k_n) \mapsto (k_n - f_{n-1}(k_{n-1}))$.
We note that the homotopy limit is not necessarily unique and so we say that $S$ is \emph{a} homotopy limit rather than it is \emph{the} homotopy limit.
\end{defn}

We can now define sheaves valued in the derived category.

\begin{defn}[Homotopy Sheaf]\label{def:homotopysheaf}
        A pre-sheaf $\cF:\cC^{op}\to\cD^b(\Acat)$ is a \define{homotopy sheaf} (or satisfies \define{\v{C}ech descent} \cite{dugger-simp}) if for every object $U\in\cC$ and cover $\cU=\{U_i \to U\}$ the following map is a quasi-isomorphism:
        \[
        \mathcal{F}(U) \xrightarrow{\simeq} \holim  \bigg[ \prod_{i}\mathcal{F}(U_i) \rightrightarrows \prod_{i,j}\mathcal{F}(U_{ij}) \rightthreearrow{} \cdots   \bigg]
        \]
\end{defn}

\subsection{Gluing Results for the PHT}

We can now prove our main results.

\begin{lem}\label{lem:site}
The poset $\CS(\R^d)$ admits the structure of a site.
\end{lem}
\begin{proof}
For every object $M\in \CS(\R^d)$ we say that $\{M_i \hookrightarrow M\}$ is a covering if it is a finite closed cover of $M$ in the usual sense, i.e.~$\cup M_i = M$.
Pullbacks exist by virtue of the fact that o-minimal sets are closed under intersection.
\end{proof}

\begin{lem}\label{lem:pre-sheaf}
	The following assignment is a pre-sheaf    
	\[
	\F:\CS(\bR^d)^{op} \to \D^b(\Shv(\bS^{d-1}\times \bR)) \qquad
	M \mapsto  \PHT(M)
	\]
	where $\PHT(M)$ is the derived sheaf version of the PHT; see Definition \ref{defn:derived-PHT-sheaf}.
\end{lem}
\begin{proof}
	We want to show that $\mathcal{F}$ is a contravariant functor. Let $\iota: M_1\xhookrightarrow{} M_2  $ be an inclusion of constructible sets of $\bR^d$. 
	Note that $ M_1 $ is a closed subspace of $ M_2 $. 
	This induces an inclusion of the auxiliary total spaces $\iota: Z_{M_1}\xhookrightarrow{} Z_{M_2} $ of Definition \ref{defn:pht-sheaf}. 
	This in turn determines a morphism of pre-sheaves $f_{M_2 *}\mathcal{S}^j(Z_{M_2};\bk)\to f_{M_1 *}\mathcal{S}^j(Z_{M_1};\bk)$ for all $j$. 
	To see this, take an open $U\subset \bS^{d-1}\times \bR$ and observe that $f_{M_1}^{-1}(U)$ is open in $ f_{M_2}^{-1}(U)$. 
    More generally, for $U\subset V$ in $\bS^{d-1}\times\bR$ we have a commutative diagram of cochain groups:
	\begin{center}
		\begin{tikzcd}
		\mathcal{S}^j(f_{M_1}^{-1}(U);\bk)& \mathcal{S}^j(f_{M_2}^{-1}(U);\bk) \arrow[l, ""]                \\
		\mathcal{S}^j(f_{M_1}^{-1}(V);\bk) \arrow[u, ""] & \mathcal{S}^j(f_{M_2}^{-1}(V);\bk) \arrow[u, ""] \arrow[l, ""]
		\end{tikzcd}
	\end{center}
    Since sheafification is a functor, we get a morphism $f_{M_2 *} \mathscr{S}^j(Z_{M_2};\bk)\to f_{M_1 *}\mathscr{S}^j(Z_{M_1};\bk)$ for all $j$. 
    These fit together into a morphism between complexes of sheaves:
		\begin{center}
		\begin{tikzcd}
		f_{M_2 *}\mathscr{S}^0(Z_{M_2};\bk) \arrow[r] \arrow[d] & f_{M_2 *}\mathscr{S}^1(Z_{M_2};\bk)\arrow[r] \arrow[d] & f_{M_2 *}\mathscr{S}^2(Z_{M_2};\bk)\arrow[d] \arrow[r] & \cdots \\
		f_{M_1 *}\mathscr{S}^0(Z_{M_1};\bk) \arrow[r]           & f_{M_1 *}\mathscr{S}^1(Z_{M_1};\bk) \arrow[r]           & f_{M_1 *}\mathscr{S}^2(Z_{M_1};\bk) \arrow[r]           & \cdots
		\end{tikzcd}
	\end{center}
The canonical functor from $\cC^b(\textbf{Shv}(\bS^{d-1}\times \bR))\to \cD^b(\textbf{Shv}(\bS^{d-1}\times \bR))$  then induces the desired restriction morphism between derived PHT sheaves:
\[
    \cF(M_2):= Rf_{M_2 *}\bk_{Z_{M_2}} \to Rf_{M_1 *}\bk_{Z_{M_1}} =:\cF(M_1)
\]
\end{proof}

The following is the main result of the paper, which was stated as Theorem \ref{thm:htpy-sheaf} in the introduction. 
We give a direct proof below, but Remark \ref{rem:ss-proof} gives a more intuitive and computationally flavored proof using spectral sequences.

\begin{thm}\label{thm:main-hpty-sheaf}
    The pre-sheaf $\cF$ of Lemma \ref{lem:pre-sheaf} is a homotopy sheaf; see Definition \ref{def:homotopysheaf}.
\end{thm}
\begin{proof}
We have already specified a Grothendieck topology on $\CS(\R^d)$ in Lemma \ref{lem:site}.
Let $\mathcal{M}=\{M_i\}_{i\in \Lambda}$ be a finite closed cover of $M$. 
Since $\mathcal{F}$ is a pre-sheaf we have an inverse system of derived sheaves: 
\[ 
\bigoplus_{I\subset \Lambda \,s.t.\, |I|=1}Rf_{M_I *}\bk_{Z_{M_I}}\rightrightarrows \bigoplus_{J\subset \Lambda \,s.t.\,|J|=2} Rf_{M_{J} *}\bk_{Z_{M_J}} \rightthreearrow{} \bigoplus_{K\subset \Lambda \,s.t.\,|K|=3} Rf_{M_{K} *}\bk_{Z_{M_K}} \to  \cdots  
\]
%Here each $M_I$ with $|I|=k$ represents a $k$-fold intersection of cover elements, i.e. a set of the form $M_{i_1,\dots,i_k}$ for $i_1,\dots,i_k \in \Lambda^k$ with all index elements distinct. 

We wish to show that  $Rf_{M *}\bk_{Z_M}$ is the homotopy limit of the above inverse system of derived sheaves, i.e. we want to show that  
\[
Rf_{M *}\bk_{Z_M} \simeq \holim \bigg[\bigoplus_{|I|=1}Rf_{M_{I} *}\bk_{Z_{M_I}}\rightrightarrows \bigoplus_{|J|=2} Rf_{M_{J} *}\bk_{Z_{M_J}} \rightthreearrow{} \cdots \bigg].
\]
By replacing each $Rf_{M_I *}\bk_{Z_{M_I}}$ with its flabby resolution by singular cochains it suffices to prove that the following is a distinguished triangle:
\begin{center}
	\begin{tikzcd}
    f_{M*}\scrS^{\dotr}(Z_M;\bk) \ar[r] &  \prod_{n}\bigoplus_{|I|=n}f_{M_{I} *} \scrS^{\dotr}(Z_{M_I};\bk) \ar[d,"\text{shift}"] \\
    & \prod_{n}\bigoplus_{|I|=n}f_{M_{I} *}\scrS^{\dotr}(Z_{M_I};\bk) \ar[ul,"+1"] & 
	\end{tikzcd}
\end{center}
%To show this we want to show that the three complexes above form a short exact sequence in the category of chain complexes of sheaves:
To show this we consider the following maps of complexes of sheaves:
% $$  0 \to f_{M*}\scrS^{\dotr}(Z_M) \to \prod_{n}\bigoplus_{|I|=n}f_{M_{I} *} \scrS^{\dotr}(Z_{M_I}) \xrightarrow{\text{shift}}
% \prod_{n}\bigoplus_{|I|=n}f_{M_{I} *}\scrS^{\dotr}(Z_{M_I}) \to 0  $$
$$  f_{M*}\scrS^{\dotr}(Z_M) \to \prod_{n}\bigoplus_{|I|=n}f_{M_{I} *} \scrS^{\dotr}(Z_{M_I}) \xrightarrow{\text{shift}}
\prod_{n}\bigoplus_{|I|=n}f_{M_{I} *}\scrS^{\dotr}(Z_{M_I})  $$
Where we drop the coefficient $\bk$ for convenience.
For every $(v,t)\in \bS^{d-1}\times \bR$, these morphisms induce a  sequence on stalks, which give rise to a sequence of cochain complexes
\begin{equation}\label{homotopysheafonstalks}
S^{\dotr}(M_{v,t}) \to \prod_{n}\bigoplus_{|I|=n}S^{\dotr}((M_I)_{v,t})\xrightarrow{\text{shift}}\prod_{n}\bigoplus_{|I|=n}S^{\dotr}((M_I)_{v,t})
\end{equation}
where $(M_I)_{v,t}$ is the intersection of $M_I$ with the half-space $\{x\mid x\cdot v \leq t\}$.
The kernel of the shift map at each stalk is clearly the cochain complex of small co-chains
$S^{{\dotr}}_{\mathcal{M}_{v,t}}(M_{v,t})$; these are cochains supported on singular simplices that are individually contained in some cover element $ (M_i)_{v,t} $ of the fiber $M_{v,t}$.
Consequently, we have the following distinguished triangle
\begin{center}
	\begin{tikzcd}
    f_{M*}\scrS_{\mathcal{M}}^{\dotr}(Z_M) \ar[r] &  \prod_{n}\bigoplus_{|I|=n}f_{M_{I} *} \scrS^{\dotr}(Z_{M_I})  \ar[d,"\text{shift}"] \\
    & \prod_{n}\bigoplus_{|I|=n}f_{M_{I} *}\scrS^{\dotr}(Z_{M_I})  \ar[ul,"+1"]& 
	\end{tikzcd}
\end{center}
where $\scrS_\mathcal{M}^{\dotr} (Z_M)$ is the sheaf of cochains supported in the cover.

We now show we can replace $ f_{M*}\scrS_{\mathcal{M}}^{\dotr}(Z_M)$ with $ f_{M*}\scrS^{\dotr}(Z_M)$ above. For any $(v,t) \in \mathbb{S}^{d-1}\times \bR$, we show the inclusion of cochain complexes
\begin{equation}\label{smallcochains}
    S^{{\dotr}}_{\mathcal{M}_{v,t}}(M_{v,t}) \hookrightarrow S^{\dotr}(M_{v,t})
\end{equation}
is a quasi-isomorphism. 
Since a map of (derived) sheaves is an isomorphism if and only if it induces an isomorphism on stalks \cite[Prop.~2.2.2]{kashiwara}, we can conclude $ f_{M*}\scrS_{\mathcal{M}}^{\dotr}(Z_M)$ and $ f_{M*}\scrS^{\dotr}(Z_M)$ are isomorphic.

% \remove{Consequently, on the level of stalks we have distinguished triangles}
% \begin{center}\color{pink}
% 	\begin{tikzcd}
%     {S^{{\dotr}}_{\mathcal{M}_{v,t}}(M_{v,t}) \ar[r] &  \prod_{n}\bigoplus_{|I|=n}S^{\dotr}((M_I)_{v,t}) \ar[d,"\text{shift}"] \\
%     & \prod_{n}\bigoplus_{|I|=n}S^{\dotr}((M_I)_{v,t}) \ar[ul,"+1"] & 
% 	\end{tikzcd}
% \end{center}\color{black}
% \remove{We now show that we can replace $S^{{\dotr}}_{\mathcal{M}_{v,t}}(M_{v,t})$ with $S^{\dotr}(M_{v,t})$ above because the inclusion
% \[
% S^{{\dotr}}_{\mathcal{M}_{v,t}}(M_{v,t}) \hookrightarrow S^{\dotr}(M_{v,t})
% \]
% is a quasi-isomorphism, and hence an isomorphism in the derived category.}

To prove inclusion~(\ref{smallcochains}) is a quasi-isomorphism, we appeal to simplicial cohomology.
By the Triangulation Theorem (Theorem 2.9 in \cite{tametopology}) we can triangulate $M_{v,t}$ in a way that is subordinate to the closed cover $\{(M_I)_{v,t}\}$ for arbitrary (yet finite) intersections $M_I$.
Simplicial cochains for this triangulation form a sub-cochain complex of $S^{{\dotr}}_{\mathcal{M}_{v,t}}(M_{v,t})$, but the triangulation can be used to compute cohomology of $M_{v,t}$.
This completes the proof.
\end{proof}

\begin{rmk}[Proof via infinity categories.]
%\justin{Make sure this integrates will with reference at the beginning of the section. Maybe even promote this remark to there.}
    There is yet another perspective to the homotopy sheaf axiom. This can be seen by promoting the derived category of sheaves on $\bS^{d-1}\times \bR$ to the derived \emph{infinity} category. 
    In Appendix~\ref{app:infinity} we show that our desired sheaf axiom for finite covers is automatically satisfied.  %The background on derived infinity categories and derived functors is given in Appendix~\ref{app:infinity}.  The sheaf axiom follows from the fact that right derived functors are exact in infinity categories and so preserves finite limits and  colimits. }
\end{rmk}

We now show how Theorem \ref{thm:main-hpty-sheaf} can be inferred from a simple spectral sequence argument.

\begin{rmk}[Proof via Spectral Sequences]\label{rem:ss-proof}
By Theorem 4.4.1 of Godemont \cite{godement} there is a resolution of $\bk_{Z_M}$ using the cover of $Z_M$ by $\{Z_{M_I}\}$.
As such there is a weak equivalence (quasi-isomorphism)
\begin{equation}\label{godemont-resolution}
    \bk_{Z_M} \to \big[\oplus_{|I|=1}\bk_{Z_{M_I}} \to \oplus_{|J|=2}\bk_{Z_{M_{J}}} \to \cdots \big]
\end{equation}

Applying the right derived pushforward functor preserves this weak-equivalence. This already proves, in essence, \v{C}ech descent for the PHT.
More specifically, the homotopy sheaf axiom is witnessed via a first quadrant spectral sequence. To see this, observe that the sequence of chain complexes in (\ref{homotopysheafonstalks}) can be unpacked to be the double complex given by complex of stalks of the pushforward of singular cochain resolution to sequence~(\ref{godemont-resolution}). 

\begin{center}
\[ \begin{tikzcd}
	{} & {} & {} \\
	{S}^1(M_{v,t};\bk) & \oplus_{|I|=1}{S}^1((M_{I})_{v,t};\bk) & \oplus_{|I|=2}{S}^1({(M_{I})_{v,t}};\bk) & {}\\
    {S}^0(M_{v,t};\bk) & \oplus_{|I|=1}{S}^0((M_{I})_{v,t};\bk) & \oplus_{|I|=2}{S}^0((M_{I})_{v,t};\bk) & {}\\
	(f_*\bk_{Z_M})_{v,t} & \oplus_{|I|=1} (f_*\bk_{Z_{M_{I}}})_{v,t} & \oplus_{|I|=2} (f_*\bk_{Z_{M_{I}}})_{v,t} & {} 
	\arrow[from=4-1, to=4-2]
	\arrow[from=4-2, to=4-3]
	\arrow[from=4-1, to=3-1]
	\arrow[from=3-1, to=2-1]
	\arrow[from=4-2, to=3-2]
	\arrow[from=4-3, to=3-3]
	\arrow[from=3-3, to=2-3]
	\arrow[from=2-1, to=2-2]
	\arrow[from=2-2, to=2-3]
	\arrow[from=3-1, to=3-2]
	\arrow[from=3-2, to=3-3]
	\arrow[from=3-2, to=2-2]
\end{tikzcd}\]
\end{center}

In practice, the spectral sequence gives a method of computing the PHT of $M$ at a point $(v,t)$.
Passing to stalks the first quadrant of the $E_1$ page reads
\begin{center}
\begin{tikzcd}
    {\oplus_{|I|=1}H^2(M_{I,v,t};\bk)} \ar[r] & {\oplus_{|J|=2}H^2(M_{J,v,t};\bk)}  \\
	 {\oplus_{|I|=1}H^1(M_{I,v,t};\bk)} \ar[r] & {\oplus_{|J|=2}H^1(M_{J,v,t};\bk)}  \\
	 {\oplus_{|I|=1}H^0(M_{I,v,t};\bk)} \ar[r] & {\oplus_{|J|=2}H^0(M_{J,v,t};\bk)}

\end{tikzcd}
\end{center}
 Call this complex $ C^n(\mathcal{M}_{v,t};H^q(\bk_\mathcal{M})) $ where $ H^q(\bk_\mathcal{M}):I\to H^q(M_{I,v,t}) $ is a system of coefficients, i.e. a cellular sheaf on the nerve of $\mathcal{M}$. Taking cohomology of this complex horizontally gives us the $ E_2 $ page of the spectral sequence. 
 Theorem 5.2.4 of \cite{godement} guarantees that we converge to the cohomology of $M_{v,t}$. 
\end{rmk}

We now reconsider Theorem \ref{thm:decat} from this spectral sequence perspective.

\begin{rmk}[Redux of the Inclusion-Exclusion Principle for the ECT]
    The inclusion-exclusion principle expression for the indicator function
    \begin{equation*}
        \mathbb{1}_M = \sum_{I\subset \Lambda}(-1)^{|I|+1}\mathbb{1}_{M_I}.
    \end{equation*}
    is exactly the local Euler-Poincar\'e index of Godemont's resolution from Equation \ref{godemont-resolution}.
    The Radon transform expression
    \begin{equation*}
        \mathcal{R}_{S}\mathbb{1}_M = \sum_{I\subset \Lambda}(-1)^{|I|+1}\mathcal{R}_{S}\mathbb{1}_{M_I},
    \end{equation*}
    is exactly the local Euler-Poincar\'e index of the pushforward of the resolution in Equation \ref{godemont-resolution}.
    Checking on stalks reveals that for any $(v,t)\in \bS^{d-1}\times \bR$
    \[ \chi\big(f^{-1}_M(v,t)\big) = \sum_{I\subset \Lambda}(-1)^{|I|+1}\chi\big(f^{-1}_{M_I}(v,t)\big).  
    \]
\end{rmk}

We now illustrate the power of the spectral sequence approach in the following corollary, which was previously stated as Theorem \ref{thm:nerve-lemma} in the introduction.

\begin{cor}\label{cor:nerve-lemma}
	Suppose $ M\in \CS(\bR^d)$ is a polyhedron and suppose $ \mathcal{M} = \{M_i\}_{i\in I}$ is a cover of $ M $ by PL subspaces, then $\pht^n(M) $ is the $ n $-th cohomology of the complex,
	\begin{equation}
	0\to \oplus_{|I|=1}\pht^0(M_I) \to \oplus_{|I|=2}\pht^0(M_I) \to \cdots
	\end{equation}
	where the $\cdots$ represents $\pht^0$ of higher intersection terms. 
\end{cor}
\begin{proof}
By examining the $E_1$ page of the spectral sequence in Remark \ref{rem:ss-proof} one can see that for a PL cover, the higher homologies, i.e. the higher PHTs, all vanish. Consequently the spectral sequence collapses after the $E_1$ page.
\end{proof}

\begin{rmk}
It should be noted that positive scalar curvature of a constructible set $M$ (when defined) obstructs Theorem \ref{cor:nerve-lemma} from being directly applied. The cover elements may necessarily have higher homology when viewed in a direction along the normal vector to the point with positive scalar curvature. See Figure~\ref{fig:PHTGOODCOVER} for an example.
\end{rmk}

\begin{figure}\label{blah}
\centering
\includegraphics[scale=0.17]{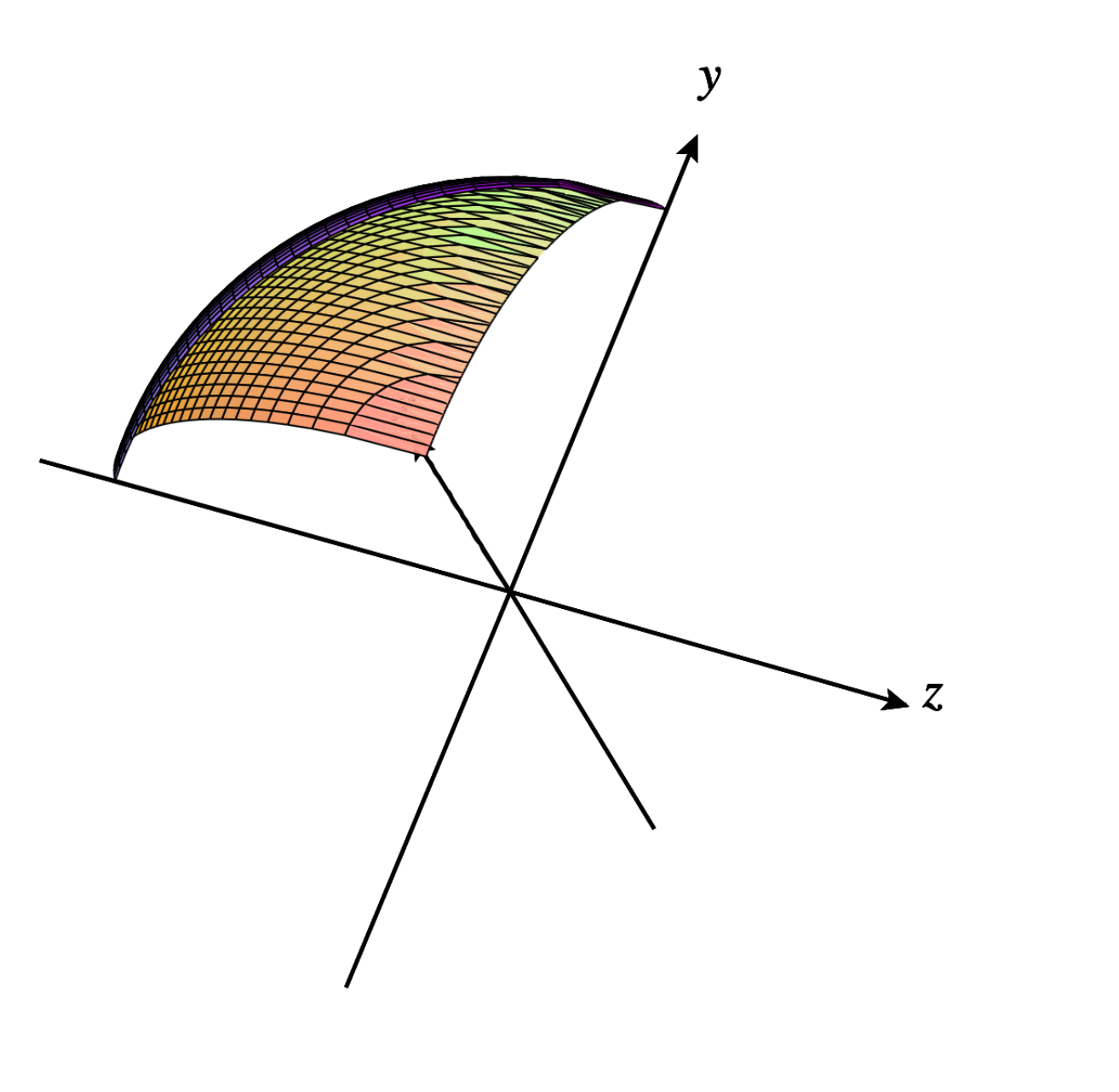}
\includegraphics[scale=0.17]{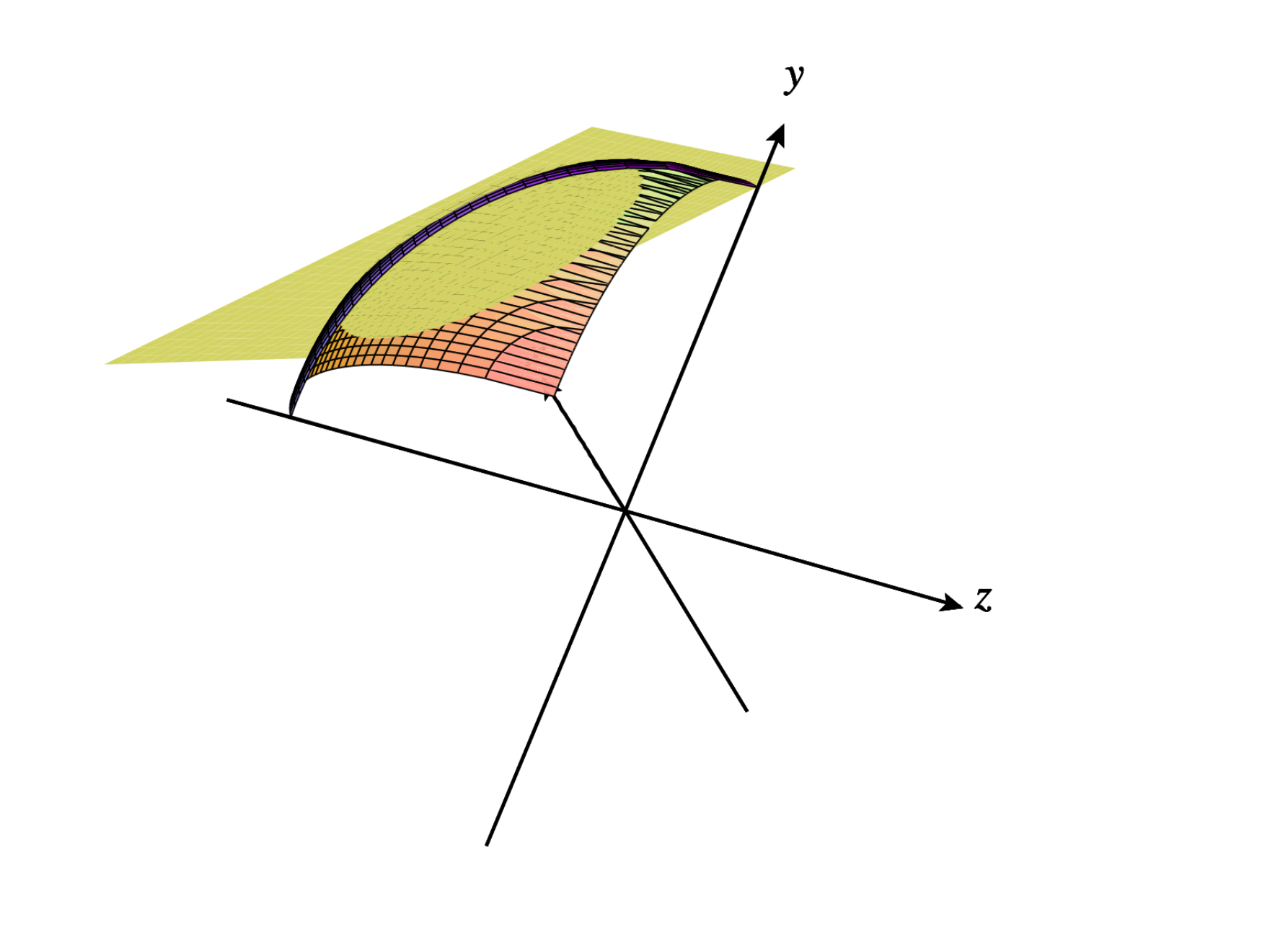}
\includegraphics[scale=0.17]{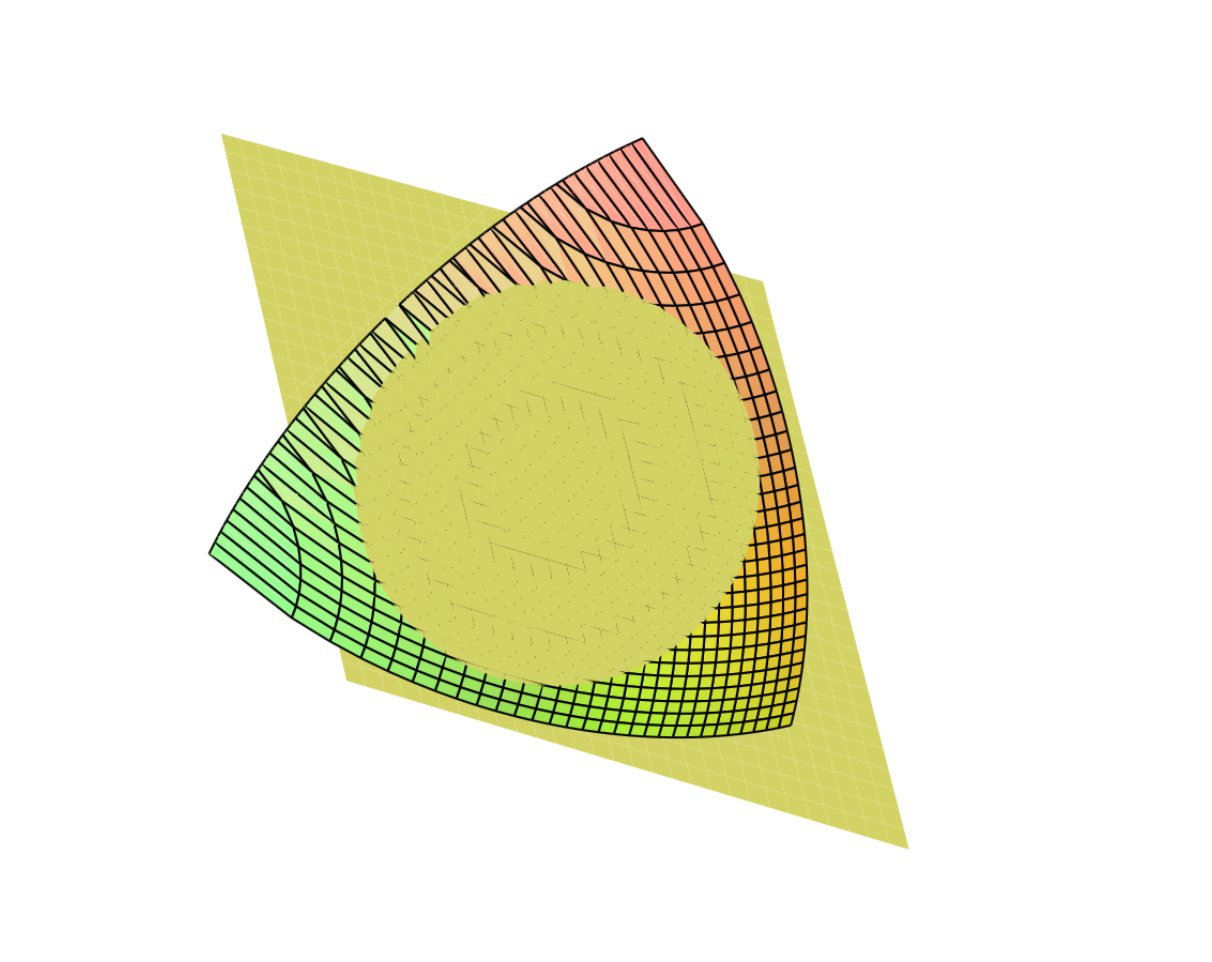}
\caption{For a cover of $M=\bS^{2}$ using the intersection with each orthant in $\R^3$, a cover element $M_i$ can have non-trivial $\PHT^1(M_i)$ even though $\PHT^1(M)$ is trivial.}\label{fig:PHTGOODCOVER}
\end{figure}

%\subsection{Example Calculation}
% So far we have showed that we can construct a homotopy sheaf on shape space where each shape is assigned its persistent homology transform. 
% In this section, we leverage the spectral sequence argument of Remark \ref{rem:ss-proof} to illustrate an explicit calculation of the gluing process. 

We conclude this section with a detailed example that leverages the spectral sequence argument of Remark \ref{rem:ss-proof}.

\begin{figure}[h]
    \centering
    \includegraphics[width=0.25\textwidth]{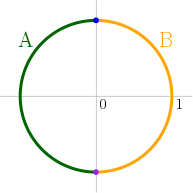}
    \caption{$M=S^1$ with cover elements $A$ and $B$.}
    \label{figure:circle}
\end{figure}

Let $ M=\bS^1 $ be the unit circle in $ \bR^2 $. 
Define a covering $ \mathcal{M} =\{A,B\} $ by two closed half-circles, as indicated in Figure~\ref{figure:circle}. 
First, we compute the PHT of each of the cover elements and their intersection.
Because our PHT sheaves are on $\bS^1\times\R$ we can project this cylinder onto the plane $\R^2$ by following the instructions in the caption of Figure~\ref{figure:PHT of circle}.

\begin{figure}
     \centering
     \begin{subfigure}[b]{0.3\textwidth}
         \centering
         \includegraphics[width=0.95\textwidth]{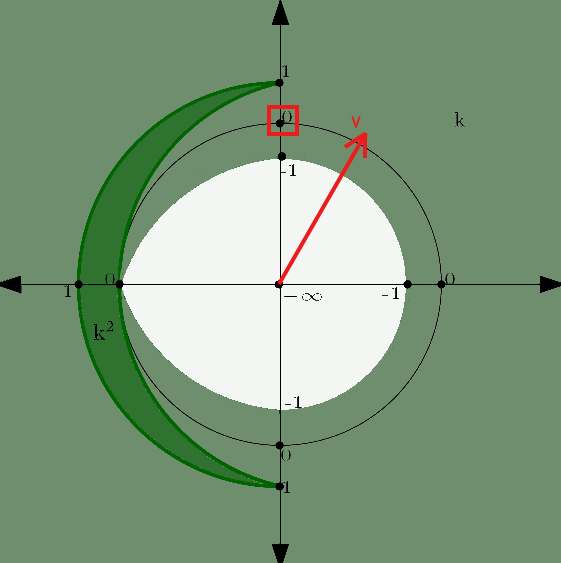}
         \caption{$\pht^0(A)$}
         \label{fig:y equals x}
     \end{subfigure}
     \hfill
     \begin{subfigure}[b]{0.3\textwidth}
         \centering
         \includegraphics[width=0.95\textwidth]{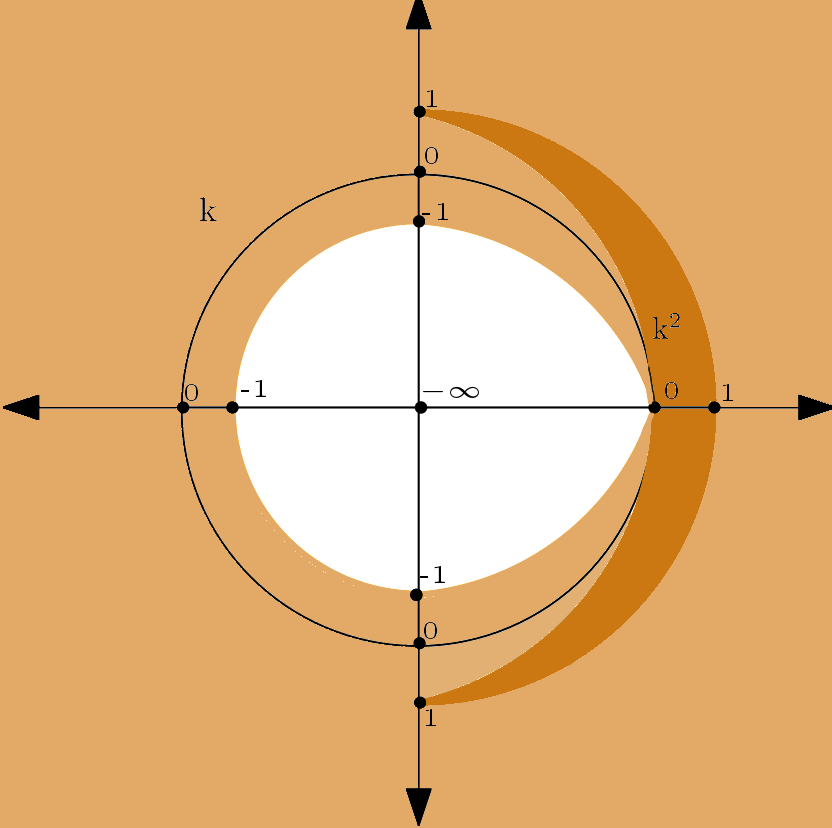}
         \caption{$\pht^0(B)$}
         \label{fig:three sin x}
     \end{subfigure}
     \hfill
     \begin{subfigure}[b]{0.3\textwidth}
         \centering
         \includegraphics[width=0.95\textwidth]{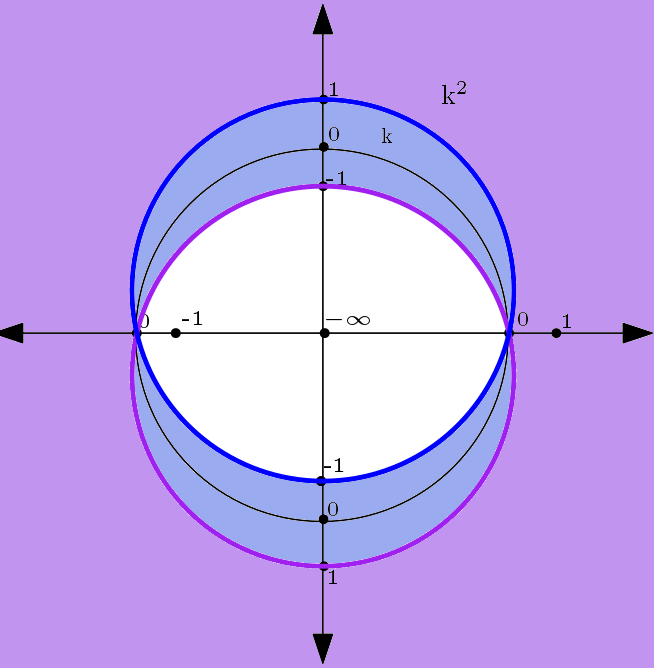}
         \caption{$\pht^0(A\cap B)$}
         \label{fig:five over x}
     \end{subfigure}
        \caption{To visualise the sheaf we have mapped the cylinder $\bS^1\times \bR$ to the plane $\bR^2$ by squashing down a tapering cylinder onto a plane in the following way: fix direction $v \in \bS^1$ and then map $\{v\}\times \bR$ onto $(0,\infty)$. So every direction $v$ has a ray attached to it.  For example, consider direction $w=\uparrow = (0,1)$ and $t=0$ then in figure (A) we see $H^0(A_{w,t},\Bbbk)= \Bbbk$ (see the red square) since $A_{w,t}= \{ x\in A | x\cdot w \le t \}$ has one connected component.}
        \label{figure:PHT of circle}
\end{figure}

Now for every point $ (v,t)\in \bS^1\times\bR $ we write out the spectral sequence in Remark~\ref{rem:ss-proof}. 
For example, let $ (v,t)= (\uparrow,0) $, then the $ E_1 $ page of the spectral sequence works out to be:

\begin{center}
\[\begin{tikzcd}
	& 0 & 0 & 0 \\
	& 0 & 0 & 0 & {} \\
	{} & \bk\oplus\bk & \bk & 0
	\arrow[from=3-2, to=3-3]
	\arrow[from=3-3, to=3-4]
	\arrow[from=1-2, to=1-3]
	\arrow[from=1-3, to=1-4]
	\arrow[from=2-2, to=2-3]
	\arrow[from=2-3, to=2-4]
\end{tikzcd}\]
\end{center}
This spectral sequence collapses after the $E_2$ page and converges to $ H^*(M_{v,t};\bk) $. And so for this example taking cohomology horizontally gives us that $ H^0(M_{v,t};\bk)=\bk $ and $ H^1(M_{v,t};\bk)=0 $.
Since the PHT is a sheaf we can can do this at all $ (v,t) $ to find $ PHT^{*}(M) $. 
Figure~\ref{fig:PHT of M} shows the PHT of $ M $.

\begin{figure}[h]
	\centering
	\includegraphics[width=0.25\textwidth]{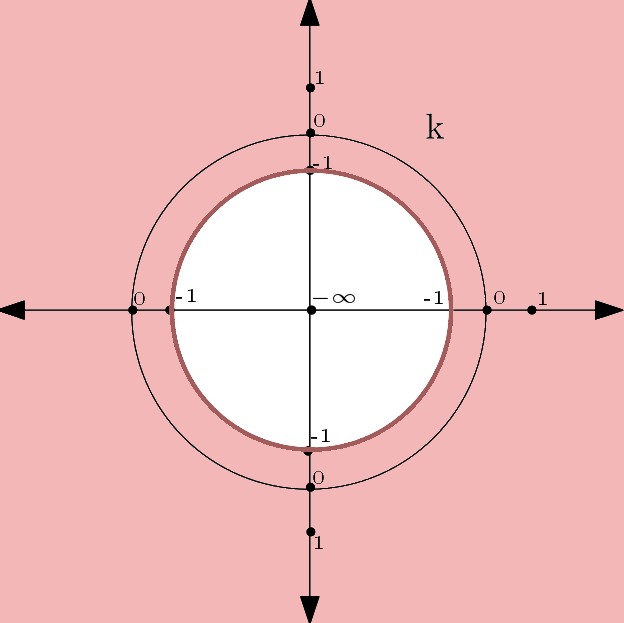}

	\caption{$\pht^0(M)$}
	\label{fig:PHT of M}
\end{figure}

\subsection{Relative PHT and ECT}
We showed how to construct the PHT of a shape by gluing PHTs that from a cover.
Intuitively this corresponds to ``adding'' several PHTs together in a precise way.
A natural question to consider is if there is a process for ``subtracting'' one PHT from another.
This is accomplished by using relative cohomology.

\begin{defn}[Relative PHT]\label{Def:rel PHT}
Let $ M\in \CS(\bR^d)$ be a constructible set and suppose $N\subset M$ is a closed constructible subset of $M$ 
The \define{relative PHT} is defined to be the sheaf over $\bS^{d-1}\times \bR$ defined by
    \[ 
        \pht^i(M,N) = \shff\Big[ U \to H^i\big(f^{-1}_M(U), f^{-1}_N(U); \Bbbk_{Z_M}\big) \Big]. 
    \]
The stalk at $(v,t)\in \bS^{d-1}\times \bR$ is the relative cohomology of the pair $(M_{v,t},N_{v,t})$.
\end{defn}

To prove that this definition suitably ``subtracts'' one PHT from another, consider the long exact sequence of pairs: 
\begin{equation}\label{relativePHT}
    \cdots \to \pht^i(M,N)_{v,t} \to \pht^i(M)_{v,t} \to \pht^i(N)_{v,t} \to  \pht^{i+1}(M,N)_{v,t} \to  \cdots
\end{equation}

Exactness at stalks implies exactness of sheaves and so we have the following long exact sequence of PHT sheaves:
\[ \cdots \to \pht^i(M,N) \to \pht^i(M) \to \pht^i(N) \to \pht^{i+1}(M,N) \to  \cdots \]

Similarly, we can also interpret the long exact sequence of pairs given in Equation~\ref{relativePHT} from the point of view of Euler characteristic. 
 \begin{defn}(Relative ECT)
    The relative ECT of a pair $(M,N)$ where $N$ is a constructible subset of $M$ is the function associated to the relative PHT sheaf under the function-to-sheaf correspondence. 
    That is for all $(v,t)\in \bS^{d-1}\times\bR$, 
    \[ \ECT(M,N)(v,t) = \chi(\pht(M,N))(v,t) = \sum_i (-1)^i \dim H^i\big(f_M^{-1}(v,t),f_N^{-1}(v,t)\big)\]
 \end{defn}

Subtraction of ECTs characterizes the relative Euler Characteristic Transform.
 \begin{lem}
    For closed and constructible $N\subset M$
      \[  \ECT(M,N) = \ECT(M)-\ECT(N). \]
 \end{lem}
 \begin{proof}
    Recall the LES of a pair from Equation~\ref{relativePHT},
 \[ \cdots \to \pht^i(M,N)_{v,t} \to \pht^i(M)_{v,t} \to \pht^i(N)_{v,t} \to  \pht^{i+1}(M,N)_{v,t} \to  \cdots \]
The long exact sequence implies that
 \[ \chi\big(\pht(M)\big)(v,t) = \chi\big(\pht(M,N)\big)(v,t) + \chi\big(\pht(N)\big)(v,t), \]
  which can be rewritten as
  \[ \ECT(M)(v,t) = \ECT(M,N)(v,t) + \ECT(N)(v,t) .\]
 \end{proof}

 %To summarise, the PHT sheaf for a constructible set $M$ corresponds to the ECT of $M$ and the homotopy sheaf axiom for the PHT corresponds to the inclusion-exclusion of the ECT.  

% \begin{center}
%\begin{tabular}{ |c|c| } 
% \hline
 %   Constructible sheaves & Constructible functions  \\ 
%    \hline
 % PHT sheaf & Euler Characteristic Transform \\ 
  %Homotopy sheaf axiom & Inclusion-exclusion %Principle for ECT  \\ 
 % Long exact sequence for PHT of a pair $(M.N)$ & $\ECT(M,N) = \ECT(M)-\ECT(N).$   \\ 
 %\hline
%\end{tabular}
%\end{center}

\section{Metrics, Stability and Approximation Theory for the PHT}\label{sec:stability}

Aside from the intrinsic theoretical interest in a gluing result for the PHT, a practical motivation is to parallelize PHT computations over a cover.
This parallelization inevitably becomes more complex if our cover elements have higher homology when viewed in certain directions and at certain filtration values. See Figure \ref{fig:PHTGOODCOVER} for an example.
This motivates our need to replace our shapes with piecewise-linear (PL) approximations.
The main result of this section proves that, up to some tolerance, we can always approximate a submanifold $M\in \CS(\R^d)$ via a PL shape $K$ so that the PHT's of $M$ and $K$ are arbitrarily close.
This requires several preparatory steps. 
Firstly, we introduce several novel distances\footnote{By which we mean an extended---the value $\infty$ is allowed---metric.} on the PHT and prove that the PHT is stable under small perturbations of the underlying shape. 
This stability property reaffirms our belief that the PHT is a good summary statistic for shapes. 
We also carry out some calculations for especially simple shapes (finite point clouds in $\R^d$) and compare these calculations with the Procrustes distance.
Secondly, we use the sampling procedure from Niyogi-Smale-Weinberger \cite{niyogi} to approximate a submanifold of $\bR^{d}$ by a (PL) polyhedron. 
Finally, we conclude from the stability theorem that the PHT of the polyhedron is close to the PHT of the submanifold.  

%\question{The next section is revised. }

\subsection{Distances on PHTs}\label{sec: distance-on-pht}
In this section, we define distances on PHTs using both the sheaf perspective (Definition \ref{defn:derived-PHT-sheaf}) and the map to persistence diagram space (Definition \ref{defn:map version}) perspective. 
We start with the sheaf perspective as the interleaving distance we introduce is more involved algebraically, but it is also simplest for proving our main stability result. Bounds on the interleaving distance then imply bounds on certain Wasserstein-type distances, but not others.

\subsubsection{Interleaving-type Distances on the PHT}
As first introduced in \cite{chazal2009proximity}, the interleaving distance provides a powerful generalization of the Hausdorff distance to functors from $(\R,\leq)$ to an arbitrary data category $\Dat$.
When $\Dat=\vect$, the celebrated isometry theorem of \cite{lesnick} proves that this distance is equivalent to the bottleneck distance (cf.~Definition \ref{def:Wasserstein-distances}), which due to its combinatorial structure as a matching problem can be computed in polynomial time.
However, the interleaving distance is far more general and permits the construction of distances in more general categorical settings, where no easy distance is to be found.
Following a suggestion of Patel, the interleaving distance for sheaves first appeared in \cite[\S 15]{curry-thesis}, then for cosheaves of sets over $\R$ (equivalent to Reeb graphs) in \cite{CRG}, and finally for derived sheaves over vector spaces in \cite{kashiwara2018persistent} as the \emph{convolution distance}.

In parallel to these developments was the thesis work of Stefanou \cite{stefanou-thesis}, which generalized the interleaving distance to any category equipped with an action of $[0,\infty)$ on it---also called a category with a flow \cite{cat-with-flow}.
This action is used to send any object $F$ to its forward evolution $F^{\epsilon}$.
With this in hand, one defines an $\epsilon$-interleaving (or $\epsilon$-isomorphism \cite[Def.~2.2]{kashiwara2018persistent}) between two objects $F$ and $G$ to be any commutative diagram of the form 
  \begin{equation*}
      \begin{tikzcd}
	F & F^{\e} & F^{2\e} \\
		G & G^{\e} & G^{2\e}
	\arrow[from=1-1, to=1-2]
	\arrow[from=1-2, to=1-3]	\arrow[from=2-1, to=2-2]
	\arrow[from=2-2, to=2-3]
	\arrow[from=2-1, to=1-2]
	\arrow[from=1-1, to=2-2]
	\arrow[from=1-2, to=2-3]
	\arrow[from=2-2, to=1-3]	
\end{tikzcd}
  \end{equation*}
although \cite{cat-with-flow,stefanou-thesis} consider more general diagrams then this one.
Regardless of the particulars, one defines the \define{interleaving distance} $d_I(F,G)$ to be the infimum over all $\epsilon\in [0,\infty)$ where such diagrams exist.
If there is no such diagram relating two objects, then $d_I(F,G)=\infty$.

Returning to interleavings of pre-sheaves, the suggestion of Patel was to define the $\e$-thickening/smoothing of a presheaf on a metric space $X$ to be $F^\e(U)=F(U^\e)$, where $U^\e$ is the thickening of an open set by $\epsilon\geq 0$.
This has a very similar effect as Kashiwara and Schapira's \cite{kashiwara2018persistent} convolution with the constant sheaf supported on the closed ball of radius $\e$.
However, for our applications we leverage the more general perspective of \cite{cat-with-flow} to work with a specialized shift operation in order to define our interleavings, which differs from both \cite{CRG} and \cite{kashiwara2018persistent}.

\begin{defn}[$\e$-Shift of the Derived PHT]
    The \define{$\e$-shift} of the derived PHT sheaf of a shape $M\in \CS(\mathbb{R}^d)$ is 
    $$\PHT(M)^{\e}: =R({f_{M^\e}})_*\Bbbk_{Z_{M^{\e}}}$$ 
    where $Z_{M^{\e}}$ is an $\e$-shift of the set $Z_M$ in the filtration parameter, i.e.,  
$$Z_{M^{\e}} = \{ (x,v,t)\in M\times\bS^{d-1}\times \bR \mid x\cdot v\le t+\e   \}.$$ 
The map $f_{M^\e}$ is the usual projection of $Z_{M^{\e}}$ onto its last two factors. 
Notice that if $(x,v,t)\in Z_M$, then it certainly is contained in $Z_{M^{\e}}$, which implies that $Z_M \subseteq Z_{M^{\e}}$. 
By functoriality of cohomology, there is a restriction map of constant sheaves $\Bbbk_{Z_{M^{\e}}} \to \Bbbk_{Z_{M}}$ and thus a map of sheaves $\PHT(M)^{\e}\to \PHT(M)$.
Further details that this defines an $\epsilon$-shift functor, starting on the image of $\CS(\R^d)^{op} \to \D^b(\Shv(\bS^{d-1}\times \R))$, which further satisfies the axioms of \cite{cat-with-flow} is left to the reader.
Note that an $\e$-shift of the derived PHT sheaf is still a derived sheaf although it might not correspond to the PHT of any particular shape.
\end{defn}

Since our sheaves are defined used cohomology, our interleaving diagram goes in the opposite direction of the one stated above, thus closer in spirit to the interleaving diagrams of \cite[\S15]{curry-thesis} and \cite{kashiwara2018persistent}.

\begin{defn}[Interleaving Distance between PHTs]
Let $M, N \in \CS(\bR^{d})$.  
An $\e$-interleaving of $\PHT(M)$ and $\PHT(N)$ is a pair of morphisms $\varphi : \PHT(M)^\e \to \pht(N)$ and $\psi: \pht(N)^\e\to \pht(M) $ such that the following diagram is commutative:
  \begin{equation}\label{dgm: interleaving}
      \begin{tikzcd}
	{\PHT(M)^{2\e}} & {\PHT(M)^{\e}} & {\PHT(M)^{}} \\
		{\PHT(N)^{2\e}} & {\PHT(N)^{\e}} & {\PHT(N)^{}}
	\arrow[from=1-1, to=1-2]
	\arrow[from=1-2, to=1-3]	\arrow[from=2-1, to=2-2]
	\arrow[from=2-2, to=2-3]
	\arrow[from=2-1, to=1-2]
	\arrow[from=1-1, to=2-2]
	\arrow[from=1-2, to=2-3]
	\arrow[from=2-2, to=1-3]	
\end{tikzcd}
  \end{equation}
 The arrows $\PHT(M)^{2\e} \to \pht(N)^{\e}$ and $\pht(N)^{2\e}\to \pht(M)^{\e}$ being given by the image of $\varphi$ and $\psi$ under the $\epsilon$-shift functor.
The interleaving distance between PHT sheaves is then
\[ d_I(\pht(M),\pht(N)) :=\inf \{\e \ge 0 \mid \exists \, \e \text{-interleaving}\}. \]
If no such interleaving exists, then $d_I(\pht(M),\pht(N))=\infty$.
\end{defn}

\subsubsection{Wasserstein-type Distances on the PHT}

We can also define a metric on the $\PHT$s viewed as map (Definition~\ref{defn:map version}). This is a generalisation of the $p$-PHT distance in \cite[Definition 5.4]{skraba}.  %\question{should we mention Kate and Skarbas paper}
Let $\PH^i(M,v)$ be the persistence diagram in degree $i$ associated to shape $M$ with sub-level set filtration given by the height function in direction $v\in \bS^{d-1}$, i.e. $h_v(x)=x\cdot v$.

\begin{defn}($(p,q)$-PHT distance)\label{defn:pqpht}
The $(p,q)$-PHT distance between $M,N\in\CS(\bR^{d})$ for $p,q\ge 1$ in degree $i$ for $i>0$ and is defined as: 
\[  d^{\pht^i}_{p,q}(M,N) = \left( \int_{v\in \bS^{d-1}} \mathcal{W}_p^q\left(\PH^i(M,v),\PH^i(N,v)\right) d\mu \right)^{\frac{1}{q}}\]
where $\mu$ is the Lebesgue measure on $\bS^{d-1}$.
%In particular, when $q=1$,
%\[  d^{\pht^i}_{p,1}(M,N) = \int_{\bS^{d-1}} \mathcal{W}_p\left(\PH^i(M,-),\PH^i(N,-)\right) \ dv  \]
When $q=\infty$,
\[  d^{\pht^i}_{p,\infty}(M,N) = \max_{v\in \bS^{d-1}} \mathcal{W}_p\left(\PH^i(M,v),\PH^i(N,v)\right). \]
Note that when $p=\infty$, we have the bottleneck distance between persistence diagrams. 
\end{defn}

Below, we consider the cases where $p=2$ or $p=\infty $, i.e., the Wasserstein 2-distance or the bottleneck distance on diagrams. 
Additionally, we will restrict ourselves to $q=2$, which computes the squared average diagram distance over all directions, or  $q=\infty$, which takes the biggest diagram distance over all directions. 
We refer to $d^{\pht^i}_{\infty,\infty}$ as the \define{PHT bottleneck distance} in degree $i$. 
The next lemma explains the relationship between the PHT sheaf interleaving distance and the PHT bottleneck distance. 

\begin{lem}\label{lem:relation-between-dist}
\[d_I(\pht(M),\pht(N))  \ge \max_{i\ge 0} d^{\pht^i}_{\infty,\infty}(M,N). \]
\end{lem}
\begin{proof}
    
Evaluating diagram~\ref{dgm: interleaving} on stalks $(v,t)\in \bS^{d-1}\times \bR$ gives 
  \[\begin{tikzcd}
	{H^i\left(f^{-1}_{M^{2\e}}(v,t);\Bbbk \right)} & {H^i\left(f^{-1}_{M^{\e}}(v,t);\Bbbk \right)} & {H^i\left(f^{-1}_{M^{}}(v,t);\Bbbk \right)} \\
		{H^i\left(f^{-1}_{N^{2\e}}(v,t);\Bbbk \right)} & {H^i\left(f^{-1}_{N^{\e}}(v,t);\Bbbk \right)} & {H^i\left(f^{-1}_{N^{}}(v,t);\Bbbk \right)}
	\arrow[from=1-1, to=1-2]
	\arrow[from=1-2, to=1-3]	\arrow[from=2-1, to=2-2]
	\arrow[from=2-2, to=2-3]
	\arrow[from=2-1, to=1-2]
	\arrow[from=1-1, to=2-2]
	\arrow[from=1-2, to=2-3]
	\arrow[from=2-2, to=1-3]	
\end{tikzcd}\] 
This can be interpreted as an $\e$-interleaving of persistence modules in degree $i$ obtained by filtering $M,N$ in direction $v$. The isometry theorem~\cite{chazal,lesnick} guarantees that the interleaving distance between persistence modules is equal to the the bottleneck distance between their corresponding persistence diagrams. In other words, 
\[d_I(\pht(M),\pht(N)) =\e \implies \forall i \ \forall v \in \bS^{d-1} \quad  \mathcal{W}_{\infty}^i\left(\PH(M,v),\PH(N,v)\right) \le \e, \]
where $\mathcal{W}_{\infty}^i$ is the bottleneck distance in degree $i$. 
\end{proof}

\subsection{Comparison with other Distances on PHTs and Shape Spaces}\label{sec:compare-distances}

We now proceeed to compute the above mentioned distances on some simple examples and attempt a comparison with other shape space metrics, primarily the Procrustes distance.
A more detailed computation of the PHT of the two embeddings of the letter `V' from Figure \ref{fig:pht-letter-V} is carried out in Example \ref{exmp:PHT-v-distance} in Section \ref{sec:stability-thm} after our main stability result is proved. 

\subsubsection{Distances between Point Clouds}
We begin with the simplest possible shapes: a finite collection of points in $\R^d$.
It should be noted that one quirk of the persistent homology transform is that it is very sensitive to the global homology of a shape.
Consequently, if two point clouds have differing numbers of points and no further construction is performed on them, then their PHTs are infinitely far apart.
To this end, let $A=\{a_1,\cdots, a_n\}$ and $B=\{b_1,\cdots, b_n\}$ be two point clouds in $\bR^d$.

% Since the sublevel sets of $A$ and $B$ are finite collection of points, we have that the sheaf interleaving distance is equal to the maximum bottleneck distance over all directions in $\bS^{d-1}$. 

% \justin{Will promote these next three results to Propositions.}

\begin{prop}\label{prop:point-cloud-distances}
    If $A= [a_1,\cdots a_n]$ and $B=[b_1,\cdots, b_n]$ are point clouds regarded as matrices where the vectors that coordinatize each point are stored as columns, then
    % \begin{enumerate}
    %     \item  $$d_I(\pht(A),\pht(B))= d^{\pht^0}_{\infty,\infty}(A,B) =\inf_{\substack{\phi:A\to B \\ \text{ matching}}} \max_{a_i\in A} \|a_i-\phi(a_i) \|_{\bR^d} $$

        \begin{equation}\label{eq:point-cloud-1}
            d_I(\pht(A),\pht(B))= d^{\pht^0}_{\infty,\infty}(A,B) =\inf_{\substack{\phi:A\to B \\ \text{ matching}}} \max_{a_i\in A} \|a_i-\phi(a_i) \|_{\bR^d} 
        \end{equation}
   
        % \item
        % \begin{align*}
        %     d^{\pht^0}_{2,2}(A,B) &= \left(\int_{\|v\|=1} \inf_{\substack{\phi:A\to B \\ \text{ matching}}} \sum_{i=1}^n|a_i\cdot v-\phi(a_i)\cdot v |^2 d\mu \right)^{1/2} \\&  \le \inf_{\substack{\phi:A\to B \\ \text{ matching}}}  \frac{\sqrt{\text{Area}(\bS^{d-1})}\times \|A-\phi(A)\|_F}{\sqrt{n}}.
        % \end{align*}

        \begin{equation}\label{eq:point-cloud-2}
        \begin{aligned}%\label{eq:point-cloud-2}
            d^{\pht^0}_{2,2}(A,B) &= \left(\int_{\|v\|=1} \inf_{\substack{\phi:A\to B \\ \text{ matching}}} \sum_{i=1}^n|a_i\cdot v-\phi(a_i)\cdot v |^2 d\mu \right)^{1/2} \\ 
            &  \le  \inf_{\substack{\phi:A\to B \\ \text{ matching}}}  \frac{\sqrt{\text{Area}(\bS^{d-1})}\times \|A-\phi(A)\|_F}{\sqrt{n}}.
        \end{aligned}
        \end{equation}
        
        %\item $$ d^{\pht^0}_{2,\infty}(A,B) =\inf_{\substack{\phi:A\to B \\ \text{ matching}}}  \|A-\phi(A)\|_\infty.$$
        %\item $d^{\pht^0}_{\infty,2}(A,B) = $$
        \begin{equation}\label{eq:point-cloud-3}
            d^{\pht^0}_{2,\infty}(A,B) =\inf_{\substack{\phi:A\to B \\ \text{ matching}}}  \|A-\phi(A)\|_\infty.
        \end{equation}
    %\end{enumerate}
    where  $\|\cdot\|_F$ and $\|\cdot\|_\infty$ stands for the Frobenius norm and Schatten $\infty$-norm of a matrix respectively.
\end{prop}

\begin{proof}
    To prove Equation \ref{eq:point-cloud-1}, we prove that the sheaf interleaving distance is equal to the max bottleneck distance over all directions. By Lemma~\ref{lem:relation-between-dist}, we have that an $\e$-interleaving of PHT sheaves implies that the $(\infty,\infty)$-PHT distance is less than $\e$. Suppose the latter distance is $\e$ in degree 0, since the sublevel sets are finite point sets, for any $v\in \bS^{d-1}$ and $t\in \bR$ the following diagram commutes,

% https://q.uiver.app/#q=WzAsNixbMCwwLCJhIl0sWzEsMCwiYiJdLFsyLDAsImMiXSxbMCwxLCJkIl0sWzEsMSwiZSJdLFsyLDEsImYiXSxbMSwwXSxbMiwxXSxbNCwzXSxbNSw0XSxbMiw0XSxbNSwxXSxbNCwwXSxbMSwzXV0=
\[\begin{tikzcd}
	{f^{-1}_{A^{2\e}}(v,t)} & {f^{-1}_{A^{\e}}(v,t)} & {f^{-1}_{A^{}}(v,t)} \\
		{f^{-1}_{B^{2\e}}(v,t)} & {f^{-1}_{B^{\e}}(v,t)} & {f^{-1}_{B^{}}(v,t)}
	\arrow[from=1-2, to=1-1]
	\arrow[from=1-3, to=1-2]
	\arrow[from=2-2, to=2-1]
	\arrow[from=2-3, to=2-2]
	\arrow[from=1-3, to=2-2]
	\arrow[from=2-3, to=1-2]
	\arrow[from=2-2, to=1-1]
	\arrow[from=1-2, to=2-1]
\end{tikzcd}\]

Since the sets in the above diagram are finite, it extends to a commutative diagram for every open $U\in\bS^{d-1}\times \bR$ and so, by the functoriality of the right derived functor there is an $\e$-interleaving of $\pht(A)$ and $\pht(B)$. So,  $d_I(\pht(A),\pht(B)) = d^{\pht^0}_{\infty,\infty}(A,B) $. For the case of points sets, the PHT Bottleneck distance turns out to be:
\[ d^{\pht^0}_{\infty,\infty}(A,B) = \max_{v\in\bS^{d-1}}  \inf_{\substack{\phi:A\to B \\ \text{ matching}}} \max_{a_i\in A} |a_i\cdot v-\phi(a_i)\cdot v|=\inf_{\substack{\phi:A\to B \\ \text{ matching}}} \max_{a_i\in A} \|a_i-\phi(a_i) \|_{\bR^d}   \]

To prove Equation \ref{eq:point-cloud-2}, we need to calculate,
$$  d^{\pht^0}_{2,2}(A,B) =  \left( \int_{v\in \bS^{d-1}}  \inf_{\substack{\phi:A\to B \\ \text{ matching}}} \sum_{i=1}^n|a_i\cdot v-\phi(a_i)\cdot v |^2 d\mu \right)^{1/2} $$ where $\mu$ is the Lebesgue measure on $\bS^{d-1}$. Let $X_i=a_i-\phi(a_i)$, and so rewriting the above expression in terms of matrices gives, 
\begin{align*}
     \left(d^{\pht^0}_{2,2}(A,B)\right)^2 \le&  \inf_{\substack{\phi:A\to B \\ \text{ matching}}} \int_{v\in\bS^{d-1}} v^T \left(\sum_{i=1}^nX_iX_i^T\right)v\ d\mu \\ =& \inf_{\substack{\phi:A\to B \\ \text{ matching}}} \int_{v\in\bS^{d-1}} v^T Yv\ d\mu \\ =& \inf_{\substack{\phi:A\to B \\ \text{ matching}}} \int_{v\in\bS^{d-1}} \mathrm{tr} (v^T Yv) \ d\mu \\=& \inf_{\substack{\phi:A\to B \\ \text{ matching}}} \int_{v\in\bS^{d-1}} \mathrm{tr} ( Yvv^T) \ d\mu\\=& \inf_{\substack{\phi:A\to B \\ \text{ matching}}}\mathrm{tr}\left( Y \int_{v\in\bS^{d-1}} vv^T \ d\mu  \right)
\end{align*}
where $Y$ is $\sum_{i=1}^nX_iX_i^T$.  The integral $A:=\int_{v\in\bS^{d-1}} vv^T \ d\mu $ is invariant under conjugation by orthogonal matrices, since the sphere is rotationally invariant. So, in particular for orthogonal matrix $U$ we have,
\[A= \int_{v\in\bS^{d-1}} vv^T \ d\mu = \int_{v\in\bS^{d-1}} Uvv^TU^T \ d\mu = UAU^T.\]
The matrix $A$ commutes with orthogonal matrices and so by Schur's Lemma, $A$ must be a scalar times the identity matrix. 
\[ A =  \int_{v\in\bS^{d-1}} vv^T = \lambda I  \] where $\lambda$
 can be found my taking trace on both sides i.e.,
 \[ \int_{v\in\bS^{d-1}} d\mu  = \lambda n.  \] 
 So, we have,
 \begin{align*}
     \inf_{\substack{\phi:A\to B \\ \text{ matching}}}\mathrm{tr}\left( Y \int_{v\in\bS^{d-1}} vv^T \ d\mu  \right) = \inf_{\substack{\phi:A\to B \\ \text{ matching}}}\frac{1}{n} \mathrm{tr}(Y) \mathrm{Area}(\bS^{d-1}) = \inf_{\substack{\phi:A\to B \\ \text{ matching}}}\frac{1}{n} \|A-\phi(A)\|_F^2 \mathrm{Area}(\bS^{d-1}) .
 \end{align*}

%Since $Y$ is a positive semi-definite symmetric matrix, it has a singular value decomposition,  $Y=U\Sigma U^T$ where $U$ is an orthogonal matrix and $\Sigma$ is a diagonal matrix given by $\diag(\sigma_1,\cdots,\sigma_n)$. The integral defined above is invariant under conjugation by orthogonal matrices, since the sphere is rotationally invariant. So,
%\[  \int_{v\in\bS^{d-1}} v^T Yv\ d\mu  =  \int_{v\in\bS^{d-1}} v^T \Sigma v\ d\mu = \sum_{i=1}^n \int_{v\in\bS^{d-1}} (v_i)^2\sigma_i d\mu =  \]

Finally, to prove Equation \ref{eq:point-cloud-3}, we note that
\begin{align*}
    \left(d^{\pht^0}_{2,\infty}(A,B)\right)^2 =& \max_{v\in \bS^{d-1}} \inf_{\substack{\phi:A\to B \\ \text{ matching}}}  \sum_{i=1}^n|a_i\cdot v-\phi(a_i)\cdot v |^2 \\ =& \inf_{\substack{\phi:A\to B \\ \text{ matching}}} \max_{v\in \bS^{d-1}} v^TYv \\ =& \inf_{\substack{\phi:A\to B \\ \text{ matching}}} \lambda_{\max}(Y) \\=& \inf_{\substack{\phi:A\to B \\ \text{ matching}}} \lambda_{\max}\left((A-\phi(A)(A-\phi(A))^T\right) \\=& \inf_{\substack{\phi:A\to B \\ \text{ matching}}}  \|A-\phi(A)\|_\infty  ,
\end{align*}
 where $\|\cdot \|_\infty$ stands for the $\infty$-Schatten norm of the matrix.                                                                       
\end{proof}

The appearance of matrix norm expressions for our PHT distances is a welcome development, as it permits a qualitative comparison with a certain class of Procrustes distances.

\begin{rmk}[Comparison with Procrustes Distances]
    The general \define{Procrustes distance} \cite{procrustes} between two \emph{ordered} point clouds in $\bR^d$, scaled by their centroid sizes\footnote{The centroid size is given by $\left(n^{-1}\sum_{i=1}^n\|a_i-\bar{a}\|^2\right)^{\frac{1}{2}}$ where $\bar{a}$ is the mean of $\{a_i\}_{i=1}^n$.}, $A=[a_1,\cdots,a_n]$ and $B=[b_1,\cdots,b_n]$ is 
    $$d_P(A,B) = \inf_{R\in \mathcal{R}} \left(\sum_{i=1}^n \|Ra_i-b_i\|^2 \right)^{1/2} =  \inf_{R\in \mathcal{R}} \|RA-B\|_F  $$
    where $\mathcal{R}$ is the group of rigid motions on $\bR^d$. 
    One advantage of the PHT distances considered here is that apriori no ordering needs to be put on the points in the point clouds.
    On the other hand, the PHT distances are sensitive to the embedding and consequently point clouds are not compared via their optimal alignments.
    
    However, there is a closer connection between the ECT/PHT and the \emph{orthogonal} Procrustes distance, which is given by
    \[
        d_{OP}(A,B) = \min_{R \text{ s.t. } R^TR=I} \| RA-B \|_F,   
    \]
    and whose solution $R=UV^T$ is determined by the singular value decomposition (SVD) of $B^TA=U\Sigma V^T$.
    By comparison, \cite[Thm.~6.7]{cmt} states that if two generic simplicial complexes $M$ and $N$ in $\R^d$ have identical pushforward measures on the space of Euler curves (or persistence diagrams), i.e., if $\ECT(M)_*\mu = \ECT(N)_*\mu$, then $M$ and $N$ are related by an $O(d)$ action.
    This, along with Remark \ref{rmk:equivariance}, suggests that one could modify the PHT distances here to also consider a minimization procedure along orthogonal transformations or rigid motions.
    \end{rmk}

\subsection{The Stability Theorem}\label{sec:stability-thm} 

In general the PHT distances are hard to compute, so often times we need to use other notions to bound the PHT distance.
In this section we prove that if two shapes are homotopic through an $\e$-controlled homotopy, then their derived PHTs (Definition~\ref{defn:derived-PHT-sheaf}) are also $\e$-close in the interleaving distance.

\begin{thm}[Stability of the PHT under Controlled Homotopies]\label{thm::stability}
     Let $M,N \subseteq \mathbb{R}^d$ be constructible sets and let $\varphi: M \to N$ and $\psi: N\to M$ be a homotopy equivalence of $M$ and $N$; that is, there are homotopies $H_M:M\times I\to M$ and $H_N:N\times I\to N$ connecting $Id_M$ to $\psi\circ \varphi$ and $Id_N$ to $\varphi\circ \psi$.
     If there is some $\e >0 $ such that $\|x-\varphi(x)\|_{\mathbb{R}^d}^2\le \e$ for all $x\in M$ and $\|y-\psi(y)\|_{\mathbb{R}^d}^2 \le \e$ for all $y\in N$, where further $\|x-H_M(x,s)\|_{\bR^{d}}^2 \le 2\e$ for all $x\in M$ and $\|y-H_N(y,s)\|_{\bR^{d}}^2 \le 2\e$ for all $y\in N$ and $s\in I$, then
     %Let $H_M:M\times I\to M$ (resp. $H_N:N\times I\to N$) be the homotopy connecting $Id_M$ and $\psi\circ \varphi$ (resp. $Id_N$ and $\varphi\circ \psi$). 
     %Further, for all $t\in I$ assume that $\|x-H_M(x,t)\|_{\bR^{d}}^2 \le 2\e$ for all $x\in M$ and $\|y-H_N(y,t)\|_{\bR^{d}}^2 \le 2\e$ for all $y\in N$.  Then 
     the PHT of $M$ and $N$ are $\e$-interleaved. %as sheaves. 
\end{thm}

    \begin{proof}
        We show the following diagram is commutative. 
    \begin{equation}\label{eqn::interleaving}
    \begin{tikzcd}
	{R(f_{M^{2\e}})_*\Bbbk_{Z_{M^{2\e}}}} & {R(f_{N^{2\e}})_*\Bbbk_{Z_{N^{2\e}}}} \\
	{R(f_{M^{\e}})_*\Bbbk_{Z_{M^{\e}}}} & {R(f_{N^{\e}})_*\Bbbk_{Z_{N^{\e}}}} \\
	{R(f_{M^{}})_*\Bbbk_{Z_M}} & {R(f_{N^{}})_*\Bbbk_{Z_N}^{}}
	\arrow[from=1-1, to=2-1]
	\arrow[from=2-1, to=3-1]
	\arrow[from=1-1, to=2-2]
	\arrow[from=2-2, to=3-1]
	\arrow[from=1-2, to=2-2]
	\arrow[from=2-2, to=3-2]
	\arrow[from=1-2, to=2-1]
	\arrow[from=2-1, to=3-2]
\end{tikzcd}\tag{*}\end{equation}

We prove this for the left triangle as the commutativity of the right triangle follows from a similar argument.
Let $U\subset\bS^{d-1}\times\bR$ be a test open set and so $Rf_*\Bbbk_{Z_{M^{\e}}}(U) =  \big[\mathscr{S}^\bullet(f_{M^\e}^{-1}(U);\Bbbk)\big] $ where $\big[\mathscr{S}^\bullet(f_{M^\e}^{-1}(U);\Bbbk)\big]$ represents the class of complexes quasi-isomorphic to the singular cochain complex on $f_{M^\e}^{-1}(U)$. 
We first prove that we have the following \emph{non-commutative} diagram of topological spaces 
\[\begin{tikzcd}
	{f_{M^{2\e}}^{-1}(U)} \\
	{f_{M^{\e}}^{-1}(U)} & {f_{N^{\e}}^{-1}(U)} \\
	{f_M^{-1}(U)}
	\arrow[hook, from=3-1, to=2-1]
	\arrow[hook, from=2-1, to=1-1]
	\arrow["g"', from=3-1, to=2-2]
	\arrow["h"', from=2-2, to=1-1]
\end{tikzcd}\]
that commutes up to homotopy; that is $h\circ g$ is homotopic to $\iota:f_M^{-1}(U)\xhookrightarrow{} f_{M^{2\e}}^{-1}(U)$. 
If we apply the singular cochain functor to this diagram and then take quotients by quasi-isomorphisms, we will get the desired commutative trianglea in Equation \ref{dgm: interleaving}. 

We now explicitly describe the maps $g$ and $h$. 
For $(x,v,t) \in f_M^{-1}(U)$ define $g$ such that  $g(x,v,t) = (\varphi(x),v,t+\e) $. 
It remains to verify that $(\varphi(x),v,t)$ is in $ f_{N^\e}^{-1}(U)$. 
Note that \[\varphi(x)\cdot v - x\cdot v = (\varphi(x)-x)\cdot v \le \|\varphi(x)-x\|^2 \le \e. \] Since $x\cdot v \le t $, we have that $\varphi(x)\cdot v - t \le \varphi(x)\cdot v - x\cdot v \le \e $ and so $\varphi(x)\cdot v \le t+\e$. 
Similarly for $(y,v,t)\in f_{N^\e}^{-1}(U)$, let $h(y,v,t)=(\psi(y),v,t+\e)$. 
After composing we get that $h\circ g (x,v,t) = (\psi\circ\varphi(x),v,t+2\e)$. 
To see that $h\circ g$ is homotopic to the inclusion $\iota:f_M^{-1}(U)\xhookrightarrow{} f_{M^{2\e}}^{-1}(U)$, define a map $K:f_M^{-1}(U) \times I \to f_{M^{2\e}}^{-1}(U) $ by 
\[ K((x,v,t),s) = (H_M(x,s),v,t+2\e). \]
By Cauchy-Schwarz, 
\[
H_M(x,s)\cdot v - x\cdot v = (H_M(x,s)-x)\cdot v \le \|H_M(x,s)-x\|^2\le 2\e
\]
for all $s\in I$. 
Since $x\cdot v\le t$, we have $H_M(x,s)\cdot v\le t+2\e$. 
This means that $K((x,v,t),s)=(H_M(x,s),v,t+2\e) \in f_{M^{2\e}}^{-1}(U)$ for $(x,v,t)\in f_M^{-1}(U)$. Further, $K((x,v,t),0)= (x,v,t+2\e)$ and $K((x,v,t),1)= (\psi\circ\varphi(x),v,t+2\e)= h\circ g(x,v,t)$. Continuity follows from continuity of $H_M$ and so $K$ is a homotopy between the inclusion map and $h\circ g$.
\end{proof}

Application of Lemma~\ref{lem:relation-between-dist} implies the following,
\begin{cor} (Stability of the PHT Bottleneck distance)
    Under the assumptions of Theorem~\ref{thm::stability} for all $i\ge 0$,
    $$d^{\pht^i}_{\infty,\infty}(M,N) \le \e$$ 
\end{cor}

\begin{rmk}
    One can also bound the $(\infty,q)$-PHT distance when $q\neq \infty$. 
    To see this, the $(\infty,q)$-PHT distance is the $q$-th integral norm of the bottleneck distance, integrated over all directions.
    Consequently, the $(\infty,q)$-PHT distance is bounded by $\e\cdot {\text{Area}(\bS^{d-1})}^{\frac{1}{q}}$, thereby establishing stability. 
\end{rmk}

\begin{exmp}\label{exmp:PHT-v-distance}
We now calculate and bound some PHT distances between the shapes $A$ (in blue) and $B$ (in red) in Figure~\ref{figure:v shapes dist}. 
The normals $v_1, v_2$ are $(\frac{2}{\sqrt{5}},\frac{1}{\sqrt{5}})$ and $(-\frac{2}{\sqrt{5}},\frac{1}{\sqrt{5}})$.
\begin{figure}[h]
    \centering
    \includegraphics[width=0.5\textwidth]{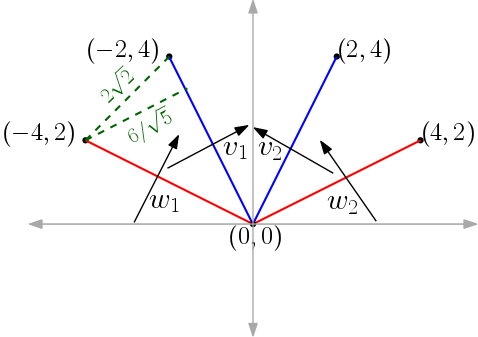}
    \caption{Distances between the PHTs of A (in blue) and B (in red)}
    \label{figure:v shapes dist}
\end{figure}
Since all the sublevel sets of the shapes are contractible, it suffices to only consider $\pht^0$ of the two shapes. 
In particular, the PHT bottleneck distance in degree 0 is
\begin{align*}
 d^{\pht^0}_{\infty,\infty}(A,B)=& \max_{\theta\in \bS^1} \mathcal{W}_\infty\left(\PH^0(A,\theta),\PH^0(B,\theta)\right) \\=& \left|(-4,2)\cdot \left(\dfrac{2}{\sqrt{5}},\dfrac{1}{\sqrt{5}}\right)\right|\\=& \frac{6}{\sqrt{5}}.     
\end{align*}
The direction at which the maximum occurs is $v_1$ and $v_2$. 

The PHT sheaf interleaving distance is hard to compute in practice. However, the application of Theorem~\ref{thm::stability} gives an upper bound on the PHT sheaf interleaving distance. In this example, we have that $A$ and $B$ are homeomorphic to each other. Explicitly, the linear map sending $(-4,2)$ to $(-2,4)$ can be extended to a homeomorphism $\varphi: B\to A $ where $(x,\frac{1}{2}|x|)\mapsto (\frac{1}{2}x,|x|)$. Further, the maximum movement of points under map $\varphi$ is, $\|(x,\frac{1}{2}|x|) - \varphi(x,\frac{1}{2}|x|)\|_{\mathbb{R}^2} \le 2\sqrt{2}$. So, by Theorem~\ref{thm::stability} and Lemma~\ref{lem:relation-between-dist} $$2\sqrt{2}\ge d_I\left(\pht(A)),\pht(B)\right)\ge \frac{6}{\sqrt{5}}.$$

\end{exmp}

\subsection{Point Samples for Approximating the PHT}\label{sec:sampling}

After having established various distances on the PHT, we are now in a position to describe how to approximate a shape with point samples so that the resulting PHTs are close.
We note that this is the only section where we require a manifold hypothesis on our shapes $M$.
This is because the problem of approximating a general constructible set (or stratified space) is not well understood.
Instead we rely on the following sampling and inference result, which makes implicit use of the injectivity radius of an embedded submanifold.
This is encoded via the condition number $\tau$, but we refer to \cite[\S2]{niyogi} for a more detailed description of this.

% In \cite{niyogi} the authors considered the problem of how to determine homology of a submanifold of $\R^d$ using a point sample.
% Their result reads as follows:

\begin{thm}[Theorem 3.1 in \cite{niyogi}]\label{thm:niyogi}
	Let $ M $ be a compact submanifold of $ \bR^d $ with condition number $ \tau $. Let $ \bar{x}=\{x_1,..,x_n\} $ be a set of $ n $ points drawn independently and identically from a uniform probability measure on $ M $. Let $ 0<\e <\tau/2 $. Let $ U=\cup_{x\in \bar{x}}B_\e(x) $ be the union of the open balls of radius $ \e $ around the sample points. Then for all
	\[ n>\beta_1\bigg(\log\beta_2 + \log\frac{1}{\delta} \bigg) \] the homology of $ U $ equals the homology of $ M $ with probability 
	$>1-\delta$.
	Here $\beta_1$ and $\beta_2$ are constants that depend on the condition number $\tau$, $\e$ and the volume of $ M $. 
The bound on $ n $ ensures that with high probability the sample is $ \frac{\e}{2} $-dense in $ M$.
\end{thm}

We let $\cU:=\{B_{\e}(x)\}$ be the balls produced by Theorem \ref{thm:niyogi}.
As the set of sampled points $X=\{x\}$ is embedded in $\R^d$, we can consider the Voronoi cell decomposition $\mathcal{V}=\{V_x\}$ of $\R^d$ generated by $X$---this is the decomposition of $\R^d$ into convex regions where every point $y\in \mathrm{int} V_x$ is closest to $x$ and no other point $x'\in X$.
The \define{alpha complex} \cite{edelsbrunner} is defined to be the (geometric) nerve of the cover $\{B_{\e}(x)\cap V_x\}$ of $U$; see Figure \ref{fig: deformation}.
By the nerve lemma, the alpha complex is homotopy equivalent to union of the balls $ U $ and so with high probability Theorem \ref{thm:niyogi} says that the homology of $ K $ is equal to homology of $ M $.
We now promote this to an observation about the PHT.

\begin{cor}(Approximation)\label{main=approx}
    	Let $ M $ be a compact submanifold of $ \bR^d $ with condition number $ \tau $. Let $ \bar{x}=\{x_1,..,x_n\} $ be a set of $ n $ points drawn independently and identically from a uniform probability measure on $ M $. Let $ 0<\e <\frac{\tau}{2} $. Let $ U=\cup_{x\in \bar{x}}B_\e(x) $ be the union of the open balls of radius $ \e $ around the sample points. Let $ K $ be the alpha complex of $ U $. Then for all
	\[ n>\beta_1\bigg(\log\beta_2 + \log\frac{1}{\delta} \bigg) \] we have that, 
%	\begin{equation}\label{pht-interleavingformula}
$	   d_I(\pht(M),\pht(K)) \le \e^2$
%	\end{equation}
with high confidence i.e. probability > 1- $\delta$.
\end{cor}

\begin{proof}
We show that the assumptions of Theorem~\ref{thm::stability} are satisfied with $ 0<\e < \tau/2$ and then apply Theorem~\ref{thm::stability} to conclude the result. 
We need to find a homotopy equivalence $\varphi:K\to M$ and $\psi:M\to K$ that satisfies the assumptions of Theorem~\ref{thm::stability}.
We do this by passing through the union of balls $U$ as an intermediary. 

\textit{Homotopy equivalence of $M$ and $U$:}  
Since the sample is $\e/2$-dense in $M$ with high probability, there is an inclusion of $M$ into $U$ with high probability. Let $\iota$ be this inclusion and let $f$ be the projection that sends $ x\mapsto \arg\min_{p\in M} \|p-x\| $. By the definition of condition number, the distance between any $q\in M$ to the medial axis is greater than $\tau > 2\e$. The well-definedness of $f$ is equivalent to $x\in U$ not contained in the medial axis. Suppose $x\in U$ is in the medial axis, so $\|p-x\| > 2\e$ for every $p\in M$. Since $\bar{x}$ is $\e/2$-dense in $M$ w.h.p and $U=\cup_{y\in\bar{x}}B_\e(y)$, it must be that for any $p\in M$, $\|p-x\|\le \|p-y\| + \|y-x\| < \frac{\e}{2} + \e < \frac{3}{2}\e$. This is a contradiction to $x$ in the medial axis. This proves that $f$ is well-defined. The map $f$ is a deformation retraction and can be seen by taking the homotopy $ H_U(x,t) = tx+ (1-t)f(x)$ for all $ x\in U $ and $ t\in [0,1] $. Further, $\|H_U(t,x)- x\| = \| tx+(1-t)f(x) -x \| = (1-t)\|x-f(x) \| < (1-t)\e/2 \le \e$. 

\textit{Homotopy equivalence of $U$ and $K$:} 
We have the inclusion map $j: K\to U$. 
There is a retract $g: U \to  K $ which follows the lines in $U$ connected to the nearest point in $K$; see Figure~\ref{fig: deformation}. 
We call the homotopy that follows these lines $G_U$. Since $U$ is a union of balls of radius $\e$, the homotopy $G_U$ does not move the points of $U$ more than $\e$. 

\begin{figure}[h]
    \centering
    \includegraphics[width=0.5\textwidth]{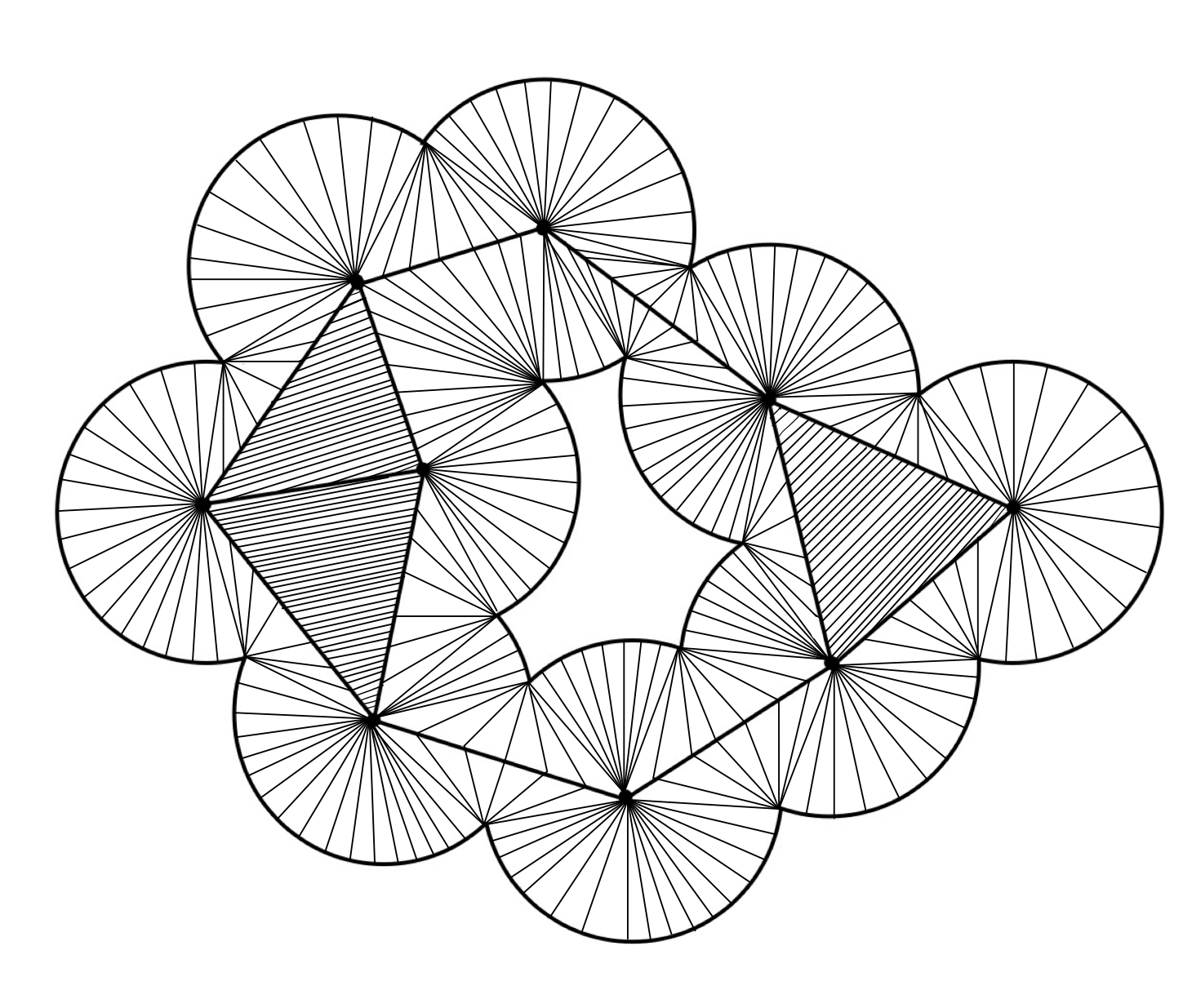}
    \caption{A collection of disks in the plane has a simpler summary given by its alpha complex, which is the geometric simplicial complex shown with shaded simplices above. The radial lines in each disk shows how points in the union of the disks deformation retract onto its alpha complex. Figure inspired by Fig.~3.1 of \cite{edelsbrunner}.}
    \label{fig: deformation}
\end{figure}
	
Let $\varphi:= f\circ j$ and $\psi:= g\circ \iota$. On composing, $\varphi \circ \psi = f\circ j \circ g \circ \iota \sim f \circ \text{Id}_U \circ \iota = f \circ \iota \sim \text{Id}_M  $ and similarly $\psi \circ \varphi \sim \text{Id}_K.$
The radius of balls are less than $\e$ and so for $x\in K$, $\|x-f\circ j(x)\| \le  \e$ and so $\|x - \varphi(x) \|^2 < \e^2$. Since the sample points are $\e/2$-dense, for $y\in M$, $\|y-g\circ \iota(y)\| \le \e/2 < \e$ and so $\|y-\psi(y)\|^2\le \e^2. $ 
 Since the homotopy maps $H_U$ and $G_U$ satisfy the conditions $\|H_U(x,t)-x\|\le \e$ and $\|G_U(x,t)-x\|\le \e$ for all $x\in U$, the homotopy map $K_1$ connecting $\varphi \circ \psi$ and $\text{Id}_M$ satisfy $\|K(x,t)-x\|\le \e$ for all $x\in M$. The case for $K_2$ connecting $\psi \circ \varphi$ and $\text{Id}_K$ is similar.
 Applying Theorem~\ref{thm::stability} gives $\e^2$-interleaving of the PHTs of $M$ and $K$. 
\end{proof}

\section{Discussion and Future Work}\label{sec:future-work}

In this paper we have outlined an algebraic construction of shape space that uses cohomology, sheaves, and ``higher'' derived variations on these.
We do this while also maintaining close contact with the simpler, tomographic integral transform that is based on Euler characteristic \cite{schapira-tom}. 
Our shape space works with shapes more general than manifolds, such as semialgebraic sets, and the ability to triangulate our shapes is more a feature of convenience than an absolute necessity.
By leveraging connections to persistent homology, we are able to also metrize our algebraic shape space in various ways and prove stability and approximation results for these.

We note that the use of sheaves to provide a rich mathematical theory for summarizing shapes is in many ways a natural evolution of previous attempts to construct shape space.
Approaches that use landmarks or diffeomorphisms each naturally lead to shape spaces modeled via fiber bundles.
Since sheaves provide a formalism for generalizing fiber bundles, where the rank can drop from point to point, we hope that our framework leads to a generalization of the landmark and diffeomorphism-based approaches as well.
However, there are important differences between these approaches---namely the use of similarity transformations and group actions---that need to be incorporated more fully into our framework as well.
We conclude with three directions for future research:

\begin{enumerate}
    \item Can we characterize the image of our shape space sheaf? $$\F: \CS(\R^d)^{op} \hookrightarrow \D^b(\Shv(\bS^{d-1}\times \R))$$ 
    In more detail, if someone gives us a derived sheaf $\calG \in \D^b(\Shv(\bS^{d-1}\times \R))$ can we certify if there is a shape $M\in \CS(\R^d)$ such that $\PHT(M)\cong \calG$? Even more fundamentally, if an oracle claims that a collection of barcodes is the result of sampling $\PHT(M)$ at directions $v_1,\ldots,v_n$, is there any way of knowing if the oracle is telling the truth? If they are, can we provide an approximate reconstruction of the shape $M$ just from this information? All of these problems are open and the last question bears some resemblance to Minkowski's theorem \cite{klain2004minkowski} for reconstructing polyhedra from normal and facet data.
    \item Can we leverage the topological simplification procedures of persistent homology \cite{ELZ-2000} to perform automated simplification of shape space? As pointed out in \cite[\S 3.5]{morozov2008homological}, there must be obstructions for dimensions 3 and higher because of the existence of the Poincar\'e homology sphere. Nevertheless, understanding obstructions to these simplifications would shed algorithmic light on some of the deepest questions in algebraic topology, such as the Poincar\'e conjecture, which was established by Perelman using PDE methods.
    \item Finally, can we embed traditional models of shape space into our current model? Does an optimal transport problem between shapes establish a corresponding optimal transport problem on sheaves? By the continuity property of the PHT, this must certainly be true at the set-theoretic level, but can we promote this correspondence to something with more algebraic structure? Preliminary work in this direction indicates that isotopies of shapes should lift to zigzags of sheaf morphisms, where sheaf cohomology can be used as an obstruction to comparing two shapes.
\end{enumerate}

\section*{Acknowledgements}

The authors would like to acknowledge conversations with Kirsten Wickelgren, Kate
Turner, Ezra Miller, Mark Goresky and Henry Kirveslahti. 
The authors would like to acknowledge partial funding from NASA Contract 80GRC020C0016, HFSP
RGP005, NSF DMS 17-13012, NSF BCS 1552848, NSF DBI 1661386, NSF IIS
15-46331, NSF DMS 16-13261, as well as high-performance computing partially
supported by grant 2016-IDG-1013 from the North Carolina Biotechnology Center.
Any opinions, findings, and conclusions or recommendations expressed in this material are those of the author(s) and do not necessarily reflect the views of any of the funders.

\bibliographystyle{plain}
\bibliography{allref.bib}

\newpage
\appendix
\section{Infinity Category Version}\label{app:infinity}

Our main task of gluing PHTs together runs into a subtle problem of ``gluing in the derived category,'' which has motivated some of the development of $\infty$-categories \cite[p.12]{lurieha}.
In this section we provide a brief overview of this theory and show how our gluing result for PHTs (Theorem~\ref{thm:main-hpty-sheaf}) follows from this more general theory.
As the theory of $\infty$-categories is quite involved, we give only a sketch of the main ideas and refer to \cite{lurieha} for a complete treatment of the subject.

At a high-level, the $\infty$-category perspective works by maintaining a data structure that keeps track of maps between objects and separates out composition in a clear way.
By doing so, it allows us to consider compositions that are identified with other morphisms only ``up to homotopy'' or via a similar identification. 
The usual refrain is that an $\infty$-category not only has 1-morphisms, like an ordinary category does, but also 2-morphisms that relate 1-morphisms, 3-morphisms that relate 2-morphisms, and all the way up \emph{ad infinitum}.
The data structure that organizes all these morphisms is called a simplicial set, and an $\infty$-category is a simplicial set that has certain ``composition'' properties that allow us to fill out a simplex.

The way this perspective addresses the lack of gluing in the derived category is subtle.
As mentioned in Section \ref{sec:derived-sheaf-theory}, the derived category lacks limits and colimits, but instead has homotopy limits and colimits.
The difference between limits and homotopy limits is that although there are objects that one might call a limit (like the kernel of a morphism), they typically lack the uniqueness and functoriality properties that are required of a genuine limit.
The infinity category perspective repairs this by keeping track of all the morphisms that might exist mapping to a candidate limit object and uses the ``higher order'' morphisms described above to address the surfeit or lack of ordinary 1-morphisms.
This is just one bit of evidence that the $\infty$-category theory approach might be the right foundation for modern pure mathematics, which should in turn percolate into future applied mathematics.

% The classical derived category discussed in Section \ref{sec:derived-sheaf-theory} is not well-behaved in the sense that it lacks limits and colimits, in particular it does not remember why maps are homotopic. \justin{Will dumb this down for the masses.} This defect can be remedied by looking at the ``higher" categorical analogue of the derived category, which contains homotopy coherent structures. 

% In this section, we provide the required background on infinity categories and explain how the sheaf axiom in Theorem~\ref{thm:main-hpty-sheaf} is automatically satisfied in the context of infinity categories. Refer to \cite{lurieha} for a complete treatment of the subject.  

\subsection{Simplicial Sets and Infinity Categories}

Let $\Delta$ denote the simplex category, whose objects are non-empty totally ordered finite sets and whose morphisms are order-preserving maps. 
A \define{simplicial set} is a functor $S:\Delta^{\mathrm{op}} \to \mathrm{Set}$. 
It is typical to call $S([n])$ the set of $n$-simplices in $S$. Here $[n]$ denotes the totally ordered set $\{0<1<\cdots < n\}$ for $n\ge 0$.
In this setting, the \define{standard $n$-simplex} is the simplicial set $\Delta^n:= \mathrm{Hom}_{\Delta}(-,[n]). $. 
The $i$th inner horn $\Lambda^n_i$ is the subcomplex of $\Delta^n$ obtained by removing the $n$-simplex and its $i^{th}$ face.  

\begin{defn}[Infinity Category]
An $\infty$-category is a simplicial set $\mathcal{C}$ which has the
following property: for any $0 < i < n$, any map $f: \Lambda^n_i \to \mathcal{C} $ admits an extension $f : \Delta^n \to \mathcal{C}$.
\end{defn}

\begin{exmp}[Nerve of a Category]
    Every ordinary category $\mathbf{C}$ has an associated $\infty$-category via the nerve construction $\mathcal{N}(\mathbf{C})$.
    The 0-simplices of the nerve are objects of $\mathbf{C}$, the 1-simplices are morphisms of $\mathbf{C}$, the 2-simplices are strings of compositions of morphisms of length 2, e.g., $g\circ f:x\to y \to z$, the 3-simplices are strings of composition of length 3 and so on.
\end{exmp}

Morphisms $f,g:x\to y$ for any objects $x,y$ in an infinity category are said to be homotopic if there is a 2-simplex of the form:
% https://q.uiver.app/?q=WzAsMyxbMCwxLCJ4Il0sWzEsMCwieSJdLFsyLDEsInkiXSxbMCwxLCJmIl0sWzAsMiwiZyIsMl0sWzEsMiwiXFxtYXRocm17aWR9X3kiXV0=
\[\begin{tikzcd}
	& y \\
	x && y
	\arrow["f", from=2-1, to=1-2]
	\arrow["g"', from=2-1, to=2-3]
	\arrow["{\mathrm{id}_y}", from=1-2, to=2-3]
\end{tikzcd}\]
It can be seen that homotopy of morphisms is an equivalence relation on the set of morphisms between any two objects.

\begin{defn}(Homotopy category)
The homotopy category of an infinity category $\mathcal{C}$, denoted $\mathrm{h}\mathcal{C}$, is the category whose objects are the objects of $\mathcal{C}$ and whose morphisms are homotopy classes of morphisms in $\mathcal{C}$. 
\end{defn}

\subsection{Homological Algebra in Infinity categories}
%\shreya{cohomological grading }

Let $\mathcal{A}$ be an abelian category with enough injective objects. We denote $\text{ch}(\mathcal{A})$ to be the category of chain complexes valued in $\mathcal{A}$. This category can be promoted to a differentially graded category by enriching it in the category of chain complexes of abelian groups $\mathrm{ch(\bf{Ab})}$. In other words, the hom-sets in $\cha$ can be replaced by mapping complexes as below:  
If $A^\bullet$ and $B^\bullet$ are chain complexes, $$ \underline{\hom}(A^{\bullet},B^{\bullet})[d]= \prod_{n\in \mathbb{Z}} \hom_{\mathcal{A}}(A^n,B^{n+d}) $$ and the differentials $\partial: \underline{\hom}(A^{\bullet},B^{\bullet})[d-1] \to \underline{\hom}(A^{\bullet},B^{\bullet})[d] $ is given by $\partial f_n= \partial_B\circ f_{n-1} - (-1)^d f_{n}\circ \partial_A$.

%The Dold-Kan correspondence gives an equivalence of categories, 
%$DK: \text{Func}(\Delta^{op},\mathcal{A}) \to \textbf{ch}(\mathrm{Ab})$. So, this allows us to think of $\mathrm{ch}(\mathcal{A})$ as category enriched in simplicial sets. 
% From now on, we refer to $\text{ch}(\mathcal{A})$ as a simplicially enriched category. 
 
We can build an infinity category from a category enriched in $\text{ch(\bf{Ab})}$ by taking the differentially graded nerve, which is similar to the construction outlined above. 
%The nerve construction is explained in detail in \cite{lurieha}.

\begin{exmp}[Differentially graded nerve of $\text{ch}(\mathcal{A})$]
\hfill
\begin{itemize}
    \item A 0-simplex of $N_{dg}(\cha)$ is a chain complex $A^\bullet$. 
    \item A 1-simplex of $N_{dg}(\cha)$ is a pair of chain complexes $A^\bullet, B^\bullet$ and a morphism $f\in \underline{\mathrm{hom}}(A^\bullet,B^\bullet)_0$ such that $\partial f= 0$. In other words, 1- simplices between $A^\bullet$ and $B^\bullet$ are morphisms $f$ such that each of the squares commute.  
    
\[\begin{tikzcd}
	\cdots & {A^{-1}} & {A^0} & {A^{1}} & \cdots \\
	\cdots & {B^{-1}} & {B^0} & {B^1} & \cdots
	\arrow["{\partial_A}", from=1-2, to=1-3]
	\arrow["{\partial_A}", from=1-3, to=1-4]
	\arrow["{\partial_B}", from=2-2, to=2-3]
	\arrow["{\partial_B}", from=2-3, to=2-4]
	\arrow["f", from=1-2, to=2-2]
	\arrow["f", from=1-3, to=2-3]
	\arrow["f", from=1-4, to=2-4]
	\arrow[from=1-4, to=1-5]
	\arrow[from=2-4, to=2-5]
	\arrow[from=2-2, to=2-1]
	\arrow[from=1-2, to=1-1]
\end{tikzcd}\]
    \item A 2-simplex of  $N_{dg}(\cha)$ consists of a triple of chain complexes $A^\bullet,B^\bullet$ and $C^\bullet$ and morphisms $f\in \underline{\mathrm{hom}}(A^\bullet,B^\bullet)_0$, $g\in \underline{\mathrm{hom}}(B^\bullet,C^\bullet)_0$ and $h\in \underline{\mathrm{hom}}(A^\bullet,C^\bullet)_0$ such that $\partial f=\partial g=\partial h =  0$ and a morphism $h'\in \hom(A^\bullet,C^\bullet)$ such that $\partial h' = g\circ f - h$. 
% https://q.uiver.app/?q=WzAsMyxbMSwwLCJCXlxcYnVsbGV0Il0sWzAsMSwiQV5cXGJ1bGxldCJdLFsyLDEsIkNeXFxidWxsZXQiXSxbMSwwLCJmIl0sWzAsMiwiZyJdLFsxLDIsImgiXSxbMSwyLCJoJyIsMix7ImN1cnZlIjozfV0sWzYsNSwiIiwyLHsic2hvcnRlbiI6eyJzb3VyY2UiOjIwLCJ0YXJnZXQiOjIwfX1dXQ==
\[\begin{tikzcd}[column sep=huge]
	& {B^\bullet} \\
	{A^\bullet} && {C^\bullet}
	\arrow["f", from=2-1, to=1-2]
	\arrow["g", from=1-2, to=2-3]
	\arrow[""{name=0, anchor=center, inner sep=0}, "h", from=2-1, to=2-3]
	\arrow[""{name=1, anchor=center, inner sep=0}, "{h'}"', curve={height=18pt}, from=2-1, to=2-3]
	\arrow[shorten <=2pt, shorten >=2pt, Rightarrow, from=1, to=0]
\end{tikzcd}\]
    \item and so on.

\end{itemize}

\end{exmp}

 \begin{rmk}
 The homotopy category of the differentially graded nerve  $hN_{dg}(\cha)$ is the usual homotopy category of chain complexes in $\mathcal{A}$, denoted as $\mathcal{K(A)}$. So, the $\infty$-category $N_{dg}(\cha)$ provides an $\infty$-categorical version of homological algebra in $\mathcal{A}$. 
 In particular, there exists a derived infinity category which is the $\infty$-localization of the category of chain complexes in $\mathcal{A}$ at the class of quasi-isomorphisms. 
 Equivalently, by \cite{lurieha}, the derived $\infty$-category can be defined to be the differentially graded nerve of the category of chain complexes of injectives in $\mathcal{A}$. 
 \end{rmk}

Let $\text{ch}^+(\mathcal{A})$ denote the full
subcategory of $\text{ch}(\mathcal{A})$ spanned by those chain complexes that are bounded from below i.e. for small $n<0$, $A^n\simeq 0$.

%This category provides an infinity categorical version of homological algebra in $\mathcal{A}$. In particular, there exists a derived infinity category which is the infinity-localization of the category of chain complexes in $\mathcal{A}$ at the class of quasi-isomorphisms. Equivalently (as shown in Lurie), the derived catgeory can also be defined as follows. 

\begin{defn}(Derived $\infty$-Category) 
    The derived $\infty$-category of $\mathcal{A}$ is the differentially graded nerve of $\text{ch}(\mathcal{A}_{\text{inj}})$, that is $\mathcal{D^+(A)} = \mathcal{N}_{dg}(\text{ch}^+(\mathcal{A_\text{inj}}))$ where $\mathcal{A}_{\text{inj}}$ is the full subcategory of $\mathcal{A}$ spanned by the injective objects. 
\end{defn}

This category is a stable $\infty$-category and carries a triangulated structure.  
In particular, the homotopy category of this category, $h\mathcal{D}^{+}(\mathcal{A})$ is the classical derived category in homological algebra. However, unlike the derived category in classical homological algebra, this category has finite limits and colimits by virtue of being a stable infinity category \cite[Prop.~1.1.3.4]{lurieha}.

As in the case of classical homological algebra, one can also study the functoriality of the derived $\infty$-category. 
Suppose $\mathcal{F}:\mathcal{A}\to \mathcal{B}$ is an additive functor between abelian categories i.e. the functor preserves the 0 object and preserves direct sums. 
The object-wise application of the functor $\mathcal{F}$ gives a functor $\text{ch}(\mathcal{F}):\text{ch}(\mathcal{A})\to \text{ch}(\mathcal{B})$ which is an enriched functor over chain complexes of abelian groups. 
The functoriality of the differentially graded nerve construction gives a functor $\mathcal{N}_{dg}(\text{ch}(\mathcal{F})): \mathcal{N}_{dg}(\text{ch}(\mathcal{A}))\to \mathcal{N}_{dg}({\text{ch}(\mathcal{B}})) $. 
In particular, this functor commutes with finite limits and colimits because the functor preserves sums and the terminal object 0. 
Similar to the case of classical homological algebra this functor $\mathcal{N}_{dg}(\text{ch}(\mathcal{F}))$ induces a right derived functor $RF: D(\mathcal{A})\to D(\mathcal{B})$ by taking its left Kan extension along the projection $\rho_A: \mathcal{N}_{dg}(\text{ch}(\mathcal{A})) \to D(A) $.

%\shreya{complete the definition properly }
%\begin{defn}(Right derived functor)
%The right derived functor $R\mathcal{F}:\mathcal{D}(\mathcal{A})\to \mathcal{D(B)}$ of $\mathcal{F}:\mathcal{A}\to \mathcal{B}$ is the left Kan extension of $\mathcal{K(F)}$ along the projection $\rho$. 
%\end{defn}

A surprising property of functors of stable infinity categories is that left exact (equivalently right exact) functors are exact, i.e., they commute with finite limits and colimits \cite[Prop.~1.1.4.1]{lurieha}. 
In particular, by the universal property of $\mathcal{D}^+(\mathcal{A})$ \cite[Thm.~1.3.3.2]{lurieha} we have that a left exact functor $\mathcal{F}:\mathcal{A}\to\mathcal{B}$, gives a left exact (and hence exact) functor, called the \emph{right derived functor}, $RF:\mathcal{D}^+(\mathcal{A})\to \mathcal{D}^+(\mathcal{B})$.

\subsection{The PHT viewed as an object of the Derived $\infty$-Category.}

We can now specialize to the case where the abelian categories are $\mathcal{A}= \Shv(M\times\bS^{d-1}\times\bR)$ and $\mathcal{B}=\Shv(\bS^{d-1}\times\bR)$ where $M\in \CS(\bR^{d})$. 
As defined before, let $f:M\times\bS^{d-1}\times\bR\to \bS^{d-1}\times\bR$ be the projection on the last two factors. 
Further, let $\iota_M: Z_M\to M\times\bS^{d-1}\times\bR$ be an embedding of spaces. 
The composition $f\circ \iota_M$ is written as $f_M$ (see Definition~\ref{defn:derived-PHT-sheaf}).

\begin{defn}[PHT: Infinity Category Version]\label{defn:PHT-infinity}
    The $\infty$-categorical definition of the PHT of $M\in \CS(\bR^d)$ is given by:
    \[\pht(M) = R(f_M)_*\Bbbk_{Z_M} \in \mathcal{D}^+(\Shv(\bS^{d-1}\times \bR)) \]
    where $\mathcal{D}^+(\Shv(\bS^{d-1}\times \bR))$ is the derived $\infty$-category of bounded below complexes.
\end{defn}

Since the right derived functor is exact (and hence commutes with finite limits and colimits), we have that the sheaf axiom for the assignment $M\mapsto \pht(M)$ for $M\in \CS(\mathbb{R}^d)$ is automatic. Precisely, for a finite covering $\{M_i \hookrightarrow M\}_{i\in I}$ we have: 
\begin{equation*}
 R(f_M)_*\Bbbk_{Z_M} = Rf_* (\iota_M)_*\mathbb{k}_{Z_M} = Rf_* \varprojlim(\iota_{M_i})_* \mathbb{k}_{Z_{M_i}} =  \varprojlim Rf_*(\iota_{M_i})_*\mathbb{k}_{Z_{M_i}} = \varprojlim R(f_{M_{i}})_*\mathbb{k}_{Z_{M_i}}  
\end{equation*}

%\begin{prop}
%Right derived functors are exact. \shreya{Explain ``exact'' in this category and remark how this is different than the usual notion of exactness.} 
%\end{prop}

%In particular, this means that right derived functors commute with finite limits and colimits \shreya{Please explain how this implication works}. So, the sheaf axiom for the PHT can be seen to be automatically true:

\end{document}